\documentclass[11pt]{amsart}

\renewcommand{\familydefault}{\sfdefault}
\newcommand{\private}[1]{}
\usepackage{amssymb, amsmath, amscd, amsthm, color, epsfig,url, graphicx}

\usepackage[all]{xy}          % for xy-pic pictures
\xyoption{dvips}              % for xy-pic pictures

\makeatletter
\renewcommand\l@subsection{\@tocline{2}{0pt}{2pc}{5pc}{}}
\makeatother

\newcommand{\R}{{\mathbb R}}
\newcommand{\Z}{{\mathbb Z}}

\newcommand{\Map}{\operatorname{Map}}

\newcommand{\Aut}{\operatorname{Aut}}
\newcommand{\Ho}{\operatorname{H}}

\newcommand{\coker}{\operatorname{coker}}

\newcommand{\Lk}{{\mathcal{L}}}
\newcommand{\HLk}{{\mathcal{H}}}

\newcommand{\LD}{{\mathcal{LD}}}
\newcommand{\HLD}{{\mathcal{HD}}}

\newcommand{\LW}{{\mathcal{LW}}}
\newcommand{\HLW}{{\mathcal{HW}}}

\newcommand{\LV}{{\mathcal{LV}}}
\newcommand{\HLV}{{\mathcal{HV}}}

\newcommand{\ord}{\operatorname{ord}}

\newcommand{\defect}{\operatorname{def}}

\newcommand{\x}{\times}

\newcommand{\LPB}{{\oplus_l C[\vec{d}_l+s_l;\, \Lk_{m}^n, c_l(\Gamma)]}}

\newcommand{\HPB}{{\oplus_l C[\vec{d}_l+s_l;\, \HLk_{m}^n, c_l(\Gamma)]}}

\newcommand{\FLPB}{{\oplus_l C[\vec{d}_l+s_l;\, L, c_l(\Gamma)]}}

\newcommand{\FHPB}{{\oplus_l C[\vec{d}_l+s_l;\, H, c_l(\Gamma)]}}

\newcommand{\del}{{\partial}}

\theoremstyle{plain}
\newtheorem{thm}{Theorem}[section]
\newtheorem{prop}[thm]{Proposition}
\newtheorem{lemma}[thm]{Lemma}
\newtheorem{cor}[thm]{Corollary}

\theoremstyle{definition}
\newtheorem{defin}[thm]{Definition}
\newtheorem{example}[thm]{Example}
\newtheorem{def/ex}[thm]{Definition/Example}

\theoremstyle{remark}
\newtheorem{rem}[thm]{Remark}
\newtheorem{rems}[thm]{Remarks}

\newcommand{\refS}[1]{Section~\ref{S:#1}}
\newcommand{\refT}[1]{Theorem~\ref{T:#1}}
\newcommand{\refC}[1]{Corollary~\ref{C:#1}}
\newcommand{\refP}[1]{Proposition~\ref{P:#1}}
\newcommand{\refD}[1]{Definition~\ref{D:#1}}

\newcommand{\refE}[1]{equation~$(\ref{E:#1})$}

\begin{document}

%%%%%%%%%%%%%%%%%%%%%%%%%%%%%%%%%%%%%%%%%%%%%%

\title[Integrals and cohomology of homotopy links]{Configuration space integrals and the cohomology of the space of homotopy string links}

%%%%%%%%%%%%%%%%%%%%%%%%%%%%%%%%%%%%%%%%%%%%%%

\author{Robin Koytcheff}
\address{Department of Mathematics, Brown University, Providence, RI}
\email{robink@math.brown.edu}
\urladdr{http://www.math.brown.edu/\~{}robink}

\author{Brian A. Munson}
\address{Department of Mathematics, U.S. Naval Academy, Annapolis, MD}
\email{munson@usna.edu}
%\urladdr{???}

\author{Ismar Voli\'c}
\address{Department of Mathematics, Wellesley College, Wellesley, MA}
\email{ivolic@wellesley.edu}
\urladdr{http://palmer.wellesley.edu/\~{}ivolic}


\thanks{The first author was supported by the National Science Foundation grant DMS 1004610.  The third author was supported in part by the National Science Foundation grant DMS 0805406.}

%%%%%%%%%%%%%%%%%%%%%%%%%%%%%%%%%%%%%%%%%%%%%%%%%%%%%%%%%%%%%%%%%%%%%%%%%%%%%%%%%%%%%%%%%%%%%%%%%%%%%%%%

\begin{abstract}
%{\bf Version: \today}
Configuration space integrals have been used in recent years for studying the cohomology of spaces of (string) knots and links in $\mathbb{R}^n$ for $n>3$ since they provide a map from a certain differential graded algebra of diagrams to the deRham complex of differential forms on the spaces of knots and links.  We refine this construction so that it now applies to the space of homotopy string links -- the space of smooth maps of some number of copies of $\mathbb{R}$ in $\mathbb{R}^n$ with fixed behavior outside a compact set and such that the images of the copies of $\R$ are disjoint -- even for $n=3$.   We further study the case $n=3$ in degree zero and show that our integrals represent a universal finite type invariant of the space of classical homotopy string links.  As a consequence, we deduce that Milnor invariants of string links can be written in terms of configuration space integrals.
\end{abstract}

\maketitle

\tableofcontents

\parskip=4pt
\parindent=0cm

%%%%%%%%%%%%%%%%%%%%%%%%%%%%%%%%%%%%%%%%%%%%%%%%%%%%%%%%%%%%%%%%%%%%%%%%%%%%%%%%%%%%%%%%%%%%%%%%%%%%%%%%

\section{Introduction}\label{S:Intro}

%%%%%%%%%%%%%%%%%%%%%%%%%%%%%%%%%%%%%%%%%%%%%%%%%%%%%%%%%%%%%%%%%%%%%%%%%%%%%%%%%%%%%%%%%%%%%%%%%%%%%%%%

%The title is not the greatest.  If we change it, we should change the Milnor paper title accordingly (and the bibliography file).  One problem is that we're really defining a diagram complex that produces cohomology classes of homotopy links in various dimension, and then specialize to classical ones, but the title only reflects the latter part of the story.

This paper is concerned with the study of the cohomology of the space of homotopy string links (or long homotopy links) $\HLk_m^n$ using \emph{configuration space integrals}, also known as \emph{Bott-Taubes integrals}.  This is the space of smooth maps of $m$ copies of $\R$ in $\R^n$ where the images of the various copies of $\R$ are disjoint and where the map is fixed outside some compact set (see \refD{(Ho)Links}).  Our main results are  

\begin{enumerate}
\item[(i)] For $m\geq 1$ and $n\geq 4$, there exists a certain differential algebra of diagrams $\HLD^d_k$, bigraded by two natural numbers called the \emph{defect} $d$ (or \emph{degree} in the terminology of \cite{CCRL}) and \emph{order} $k$.  There exists a differential algebra map
\begin{equation}\label{E:HolinkMapIntro}
I_{\HLk}\colon \HLD^d_k \longrightarrow \Omega^{k(n-3)+d}(\HLk_m^n),
\end{equation}
where $\Omega^*$ stands for the deRham complex of differential forms (\refT{IntegralsAreAlgebraMaps}).  Defining the main degree in $\HLD^d_k$ to be $k(n-3)+d$ makes $\HLD$ into a (singly graded) differential graded algebra and the above map into a map of differential graded algebras.
When the defect $d=0$, the induced map in cohomology is injective.
\item[(ii)]  For $m\geq1$ and $n=3$, we get a similar map for defect $d=0$ (which coincides with main degree zero when $n=3$):
\begin{equation}\label{E:HolinkMapIntro_n3}
I_{\HLk}\colon \HLD_k^0 \longrightarrow \Omega^0(\HLk_m^n),
\end{equation}
which takes closed forms to closed forms and is injective in cohomology.  This map produces all finite type invariants of homotopy string links (\refT{UniversalFTHoLinks}).
\item[(iii)] As a consequence of the previous result, we can express Milnor invariants of homotopy string links in $\R^3$ (\refT{Milnor}) completely in terms of configuration space integrals.  Using the weight systems for these invariants, we can explicitly write down these formulae up to lower order finite type invariants (which themselves can be expressed as configuration space integrals).
\end{enumerate}

%\begin{tabular}{ll}
%{\bf Theorem 1:} & For $m\geq 1$ and $n\geq 3$, there exists a certain differential algebra of of diagrams $\HLD$ and a differential algebra map
%\begin{equation}\label{E:HolinkMapIntro}
%I_{\HLk}\colon \HLD^*\longrightarrow \Omega^*(\HLk_m^n),
%\end{equation}
%where $\Omega^*$ stands for the deRham complex of differential forms (\refT{IntegralsAreAlgebraMaps}). \\
%{\bf Theorem 2:} & In degree zero and for $n=3$, this map produces all finite type invariants of homotopy string links (\refT{UniversalFTHoLinks}). \\
%{\bf Theorem 3:} & As a consequence of the previous result, we obtain configuration space integral expressions for Milnor invariants of homotopy string links in $\R^3$ (\refT{Milnor}).
%\end{tabular}

The first two results parallel those for string (i.e.~long) knots $\mathcal{K}^n$, i.e.~embeddings of $\R$ in $\R^n$ \cite{BT, CCRL, CCRL:Struct, Thurs}.  More generally, they parallel results for string links $\Lk_m^n$, i.e.~embeddings of $m$ copies of $\R$ in $\R^n$, where all maps are always prescribed outside some compact set.  
In the process of obtaining our results, we provide an erratum to \cite{V:B-TLinks}, which considered the case of string links.
At the same time, these results are also very different from the case of string knots/links. To explain, we first briefly review the standard construction of the  map  
\begin{equation}\label{E:OldLinkMapIntro}
\overline{I}_{\Lk}\colon \LD\longrightarrow \Omega^*(\Lk_m^n)
\end{equation}
corresponding to that in \eqref{E:HolinkMapIntro} and is familiar from the literature \cite{CCRL, V:B-TLinks}.  In particular, $\LD$ is a familiar diagram complex associated to the space of string links.

To produce forms on $\Lk_m^n$, one first creates fiber bundles of configuration spaces over this space. Each bundle depends on a diagram in $\LD$.  A diagram has vertices that abstractly represent configurations of points on and off a link, and its edges prescribe a way to pull back copies of the volume $(n-1)$-form from the sphere $S^{n-1}$ to the total space of the bundle.  We then integrate this pullback form along the fiber, thereby producing a form on $\Lk_m^n$.  One of the main reasons this construction works is that ordinary embedded links behave well with respect to restriction, i.e.~the restriction map for links is a fibration by the Isotopy Extension Theorem.  

The situation is different for $\HLk_m^n$ because homotopy links are not embeddings and the restriction map is far from a fibration (see Section \ref{S:ProblemBundlesVertices}).  Thus the obvious generalization of the above fails to extend to $\HLk_m^n$.  The  main contribution of this paper is a refinement of the construction of the fiber bundles which makes it possible to integrate over $\HLk_m^n$. The short explanation of this refinement is that, in the construction of $\overline{I}_{\Lk}$, only vertices of the diagram determine the bundle, while in our construction, both vertices and edges are relevant.  This leads to breaking up the diagram according to its ``grafts" (see \refD{GraftComp} and \refD{graft}) and the construction of what is essentially a product bundle over the set of graft components.  In this fashion we construct a new map
\begin{equation}\label{E:LinkMapIntro}
I_{\Lk}\colon \LD\longrightarrow \Omega^*(\Lk_m^n),
\end{equation}
identify a subcomplex $\HLD\subset\LD$, exhibit the map from equation \eqref{E:HolinkMapIntro}, and show that the   diagram
$$
\xymatrix{
\HLD   \ar@{^{(}->}[r] \ar[d]_{I_{\HLk}} &  \LD \ar[d]^{I_{\Lk}} \\
\Omega^*(\HLk_m^n) \ar[r] &  \Omega^*(\Lk_m^n)
}
$$
commutes.
After we define $I_{\Lk}$ and show how it restricts to the map $I_{\HLk}$, we show in \refP{SameForms} that the old integration map $\overline{I}_{\Lk}$  and our map $I_{\Lk}$ produce the same form.  Thus our construction is indeed a refinement of the one considered by others.

One interesting attribute of our construction of $I_{\HLk}$ is that this map can be defined even when $n=3$, which is not the case with $I_{\Lk}$.  The reason is that the issue of the vanishing of the integration along a certain part of the boundary of the bundle, the so-called \emph{anomalous faces},  is not present for homotopy links (see Remark \ref{R:NoAnomalous}).  This is potentially an exciting feature since it means that the map $I_{\HLk}$ might contain interesting information about the topology of the space of classical homotopy string links.

%closer to being a quasi-isomorphism for $n=3$; this is already widely believed for $\overline{I}_{\Lk}$ (and hence $I_{\Lk}$) for $n>3$, but now we can  conjecture the same for $I_{\HLk}$ and $n\geq 3$.

The anomalous face also makes an appearance in the study of finite type invariants of knots and links via configuration space integrals \cite{Thurs, V:SBT, V:B-TLinks}.  As stated in (ii) above, we extend this study to the case of homotopy string links.  The difference is that, for (string) knots and links, these integrals represent a universal finite type invariant only up to an indeterminacy due to the non-vanishing of anomalous faces (see Section \ref{S:AnomalousFix}).  However, this is not a problem for homotopy string links and in \refT{UniversalFTHoLinks} we give the correspondence between weight systems (functionals on diagrams with defect zero satisfying some relations) and finite type invariants of homotopy string links without any indeterminacy.  
%Again, this correspondence is in the case of knots or ordinary embedded links only understood up to the anomalous indeterminacy, but our \refT{UniversalFTHoLinks} says that for homotopy links this indeterminacy goes away.

\refT{UniversalFTHoLinks} connects to other work that has been done on finite type invariants of homotopy string links.  To show that $I_{\HLk}$ represents the universal finite type invariant of homotopy string links, we first show that the zeroth cohomology of the complex $\HLD$ gives a certain vector space of diagrams that has already been studied \cite{BN:HoLink, Mellor-HoInv, MelThurs-HoInv, V:B-TLinks}.  Our construction, however, is dictated by geometry -- we have arrived at $\HLD$ by looking for spaces we could integrate over to get forms on $\HLk_m^n$.  Further, we are concerned with all $n\geq 3$, and for $n=3$ and degree zero (which is also defect zero) we happen to have obtained the ``correct" diagrams and relations.  This means that our approach is indeed a generalization, with a new perspective, of existing work.

Since Milnor invariants of homotopy string links are known to be finite type, \refT{UniversalFTHoLinks} immediately gives a novel construction for Milnor invariants entirely in terms of configuration space integrals (as mentioned in (iii)).    Further, some connections between tree diagrams and Milnor invariants arise naturally from our construction, and this will be pursued in future work.
% \cite{MV:MilnorHoLinks}.  
More details about the planned work on Milnor invariants are given in \refS{MilnorInvariants}.

The philosophy in this paper is thus to reconstruct all the ingredients of the map \eqref{E:LinkMapIntro}, but in an improved and refined fashion, and then show at every important instance of the construction how everything works when one restricts to the case of homotopy string links.  Consequently, we have had to be precise and detailed about the definition and structures in the diagram complex $\LD$, the fiber bundles mentioned earlier, the defect zero case, etc.  This has required us to fill in some of the details that have been missing from the literature.  Some instances of this are:
\begin{itemize}
\item the graph complex $\LD$ is now defined purely combinatorially (it had largely been done through pictures before, and mainly for the case of knots);
\item the correspondence between the shuffle product on $\LD$ and the wedge product on $\Omega^*(\Lk_m^n)$ is elucidated;
\item the STU and IHX relations  in defect zero are derived from the graph complex;
\item essentially all the details of the proof that configuration space integrals represent a universal finite type invariant of embedded string links and homotopy string links are given (the most complete proof for knots is in \cite{V:SBT}; in particular, our work provides an erratum to \cite{V:B-TLinks}, which treated the string links case);
%\item many examples are given throughout the exposition.
\end{itemize}
In addition, the work here unifies and extends many seemingly disparate results in the subject of configuration space integrals (the case $n>3$ is in literature usually treated separately from the case $n=3$).   All of this makes for a self-contained and thorough treatment of how configuration space integrals are used in knot and link theory.  We hope that in addition to establishing some new and useful results, this paper will serve as a practical and a beneficial introduction to the subject.

Finally, it is worth noting where the results from this paper fit into the larger program of studying homotopy string links (and embedded string links) in the context of \emph{manifold calculus of functors}.  To that end, the second and the third author have developed its multivariable version \cite{MV:Multi}, as well a cosimplicial model for the functor calculus Taylor tower for homotopy string links \cite{MV:Links}.  Using this model, the plan is to show that the map $I_{\Lk}$ factors through the Taylor tower and that this tower classifies finite type invariants.  This can hopefully be used to reprove the Habegger-Lin classification of homotopy links \cite{HabLin-Classif} as well as to extend some of their results to ordinary links.  
%rephrase this
Along the way, the authors plan to study Milnor invariants in the context of manifold calculus, 
%using \cite{MV:MilnorHoLinks}, 
which continues the exploration of the connection between configuration space integrals and Milnor invariants, as well as \cite{M:Milnor}, which connects manifold calculus of certain generalizations of homotopy links to generalizations of Milnor invariants.

%%%%%%%%%%%%%%%%%%%%%%%%%%%%%%%%%%%%%%%%%%%%%%%%%%%%%%%%%%%%%%%%%%%%%%%%%%%%%%%%%

\subsection{Organization of the paper}

%%%%%%%%%%%%%%%%%%%%%%%%%%%%%%%%%%%%%%%%%%%%%%%%%%%%%%%%%%%%%%%%%%%%%%%%%%%%%%%%%

\begin{itemize}
\item In Section \ref{S:Links}, we define the spaces of string links and homotopy string links, make some observations about them, and set some notation and conventions.

\item In Section \ref{S:Diagrams}, we define the diagram complex $\LD$ and its subcomplex $\HLD$.  \refS{LinkGraphs} contains  the detailed definition of $\LD$ and \refS{AlgebraicDiagrams} discusses the differential and the shuffle product on this graded vector space.  The subcomplex $\HLD$ is identified in \refS{HomotopyLinkGraphs}.  In \refS{Degree0}, we show that $\LD^0$ and $\HLD^0$ consist of trivalent diagrams modulo STU and IHX relations, plus an extra relation for $\HLD^0$ (\refP{DegreeZero}).   In that section we also describe the correspondence between trivalent and chord diagrams (\refT{Trivalent=Chord}).  Many examples are given throughout.

\item In Section \ref{S:ChainMap}, we construct the map $I_\Lk$ in several steps.  After reminding the reader about compactifications of configuration spaces in \refS{Compactification}, we first recall in \refS{BundlesVertices} the standard way of building a bundle of compactified configuration spaces over the space of string links from a diagram $\Gamma\in\LD$. In \refS{ProblemBundlesVertices}, we show why this procedure fails to give bundles over the space of homotopy string links.  Guided by how this procedure fails, we then go back to the complex $\LD$, define the graft components of a diagram in \refS{GraftComponents}, and rework the definition of the bundle of configuration spaces based on these components in \refS{BundlesVerticesEdges}.  The upshot is that these new bundles can now be defined over the space of links for any $\Gamma\in\LD$ or over the space of homotopy links for any $\Gamma\in\HLD\subset\LD$.  In \refS{Forms}, we return to the main goal -- producing forms on the space of (homotopy) links -- and show how the edges of a diagram give a prescription for pulling back the product of volume forms to our bundles.  Finally in \refS{Integrals} we describe how this pullback form can be pushed forward along the fiber of the bundle to $\Lk_m^n$ or $\HLk_m^n$ and give some examples.  \refP{SameForms} states that the forms obtained using the standard definition of the bundles over $\Lk_m^n$ and using our refined one are the same.  This allows us to unify the old configuration space integral approach for string links with a new one for homotopy string links. Our \refT{IntegralsAreAlgebraMaps} in \refS{AlgebraMap} shows that this integration is compatible with all the structure on $\LD$.

\item In Section \ref{S:FTlinks}, we study the case of classical homotopy string links ($n=3$) and prove that configuration space integrals represent a universal finite type invariant for this space (\refT{UniversalFTHoLinks}).  We begin by discussing in \refS{AnomalousFix} the anomalous face mentioned above and then review finite type theory and its connection to the combinatorics of chord diagrams in \refS{FTInvariants}.  \refS{IntegralsAndFT} is finally devoted to the proof of \refT{UniversalFTHoLinks}. In \refS{MilnorInvariants}, we deduce some quick consequences of \refT{UniversalFTHoLinks} in regard to Milnor invariants.  That section is meant to set the stage for the further study of Milnor invariants using configuration space integrals.

\end{itemize}

%%%%%%%%%%%%%%%%%%%%%%%%%%%%%%%%%%%%%%%%%%%%%%%%%%%%%%%%%%%%%%%%%%%%%%%%%%%%%%%%%

\subsection{Acknowledgements}

%%%%%%%%%%%%%%%%%%%%%%%%%%%%%%%%%%%%%%%%%%%%%%%%%%%%%%%%%%%%%%%%%%%%%%%%%%%%%%%%%

We would like to thank Phil Hirschhorn, Greg Arone, Blake Mellor, and Tom Goodwillie for helpful conversations.  We are also indebted to Victor Turchin for his careful reading of an early draft of this paper.  We thank the referee for numerous useful comments.
%This paper was partly written while the second author was visiting the University of Virginia and he would like to thank the mathematics department for its hospitality.
%
The second author would like to thank Wellesley College for their hospitality, as this work was partially completed during his stay there. The third author would like to thank the University of Virginia's Department of Mathematics for its hospitality; this paper was partially written while he was on leave there.
%%%%%%%%%%%%%%%%%%%%%%%%%%%%%%%%%%%%%%%%%%%%%%%%%%%%%%%%%%%%%%%%%%%%%%%%%%%%%%%%%%%%%%%%%%%%%%%%%%%%%%%%

\section{Spaces of string links and homotopy string links}\label{S:Links}

%%%%%%%%%%%%%%%%%%%%%%%%%%%%%%%%%%%%%%%%%%%%%%%%%%%%%%%%%%%%%%%%%%%%%%%%%%%%%%%%%%%%%%%%%%%%%%%%%%%%%%%%

In this section, we define the spaces of string links and homotopy string links and set some conventions.

\subsection{Definitions and basic facts}

\begin{defin}\label{D:MapSpace}
Let $m\geq 1$, $n\geq 2$ be integers. Let $\sqcup_m\R$ denote the disjoint union of $m$ copies of the real line.  Let $\Map_c(\sqcup_m\R,\R^n)$ denote the space of smooth maps $\sqcup_m\R\to\R^n$ which outside some compact subset of $\sqcup_m\R$ 
agree with 
%$m$ linear maps 
the map which on the $i$th copy of $\R$ is given by 
$$t\longmapsto \left(t,|t|\left(\frac{m+1}{2}-i\right) ,0,0,\ldots,0\right).$$ 
This space is endowed with the $\mathcal{C}^\infty$ topology.
\end{defin}

The following is clear.

\begin{prop}\label{P:mapsarecontractible}
$\Map_c(\sqcup_m\R,\R^n)$ is contractible.
\qed
\end{prop}

%We now introduce the spaces of links and homotopy links which play a significant role in this work.

\begin{defin}\label{D:(Ho)Links}\ \ 
\begin{itemize}
\item Let $\Lk_m^n\subset \Map_c(\sqcup_m\R,\R^n)$ denote the space of \emph{string (or long) links in $\R^n$ with $m$ strands}. It consists of those maps $L\in\Map_c(\sqcup_m\R,\R^n)$ which are smooth embeddings (one-to-one maps whose derivatives are of maximal rank everywhere). A path in this space is called an \emph{isotopy}. 
\item Let $\HLk_m^n\subset \Map_c(\sqcup_m\R,\R^n)$ denote the space of \emph{string (or long) homotopy links in $\R^n$ with $m$ strands}. It consists of those maps $H\in\Map_c(\sqcup_m\R,\R^n)$ such that if $x$ and $y$ are points in distinct copies of $\R$, then $H(x)\neq H(y)$. A path in this space is called a \emph{link-homotopy}.
\end{itemize}
\end{defin}

Note that $\HLk_m^n$ is an example of a space of \emph{link maps}, studied by the second author in \cite{GM:LinksEstimates,  M:LinkNumber, M:Milnor} from the perspective of the manifold calculus of functors; one motivation for the current work was to continue this thread of inquiry.

Throughout the paper, we will often drop the adjectives ``string" and ``long", and refer to these objects as ``links'' and ``homotopy links''. Each link and homotopy link is oriented in the sense that all copies of $\R$ are given the usual orientation.  The images of the copies of $\R$ will be called \emph{strands}.

In the literature, a more common picture for string links is to take the $i$th copy of $\R$ to $t \mapsto (t,i,0,0, \ldots, 0)$ outside the fixed compact set.  There is a clear correspondence between these string links and string links under our definition; in particular we think of the unlink as in Figure \ref{Fig:parallel}.  We have chosen our definition for technical reasons related to defining configuration space integrals for string links.  This technicality is also related to an error in \cite{V:B-TLinks} which we will correct.

\begin{figure}[h]
\input{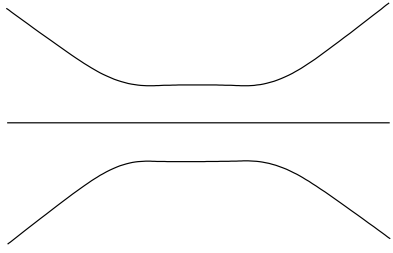}
\caption{The string unlink with three strands.}
\label{Fig:parallel}
\end{figure}

The following corollary is immediate from \refP{mapsarecontractible}.

\begin{cor}\label{C:HLK1isContractible}
$\HLk_1^n$ is contractible. 
\qed
\end{cor}

In Section \ref{S:FTlinks} we will be interested in $\Ho^0(\HLk_m^3)$, i.e.~the space of real-valued invariants of $m$-strand homotopy links in $\R^3$, so we discuss deRham forms on link spaces below.  First we make an observation which will be useful in Section \ref{S:FTlinks}.

\begin{rem}
\label{R:HtpyLinks2Links}
By general position, every homotopy link is link-homotopic to an embedded link. Moreover, by the remark following Definition 1.5 in \cite{HabLin-Classif}, we can approximate a link-homotopy between embedded links by one which consists of isotopies and ``crossing changes'' of a strand with itself. A crossing change is a homotopy which takes place in the interior of a ball containing only two segments of a single strand, and the two segments cross during the homotopy. To check that something is an invariant of $\HLk_m^3$, it thus suffices to check that it is an invariant of $\Lk_m^3$ and that it remains unchanged under such crossing changes.  This observation will be used in the proof of \refP{HoLinkIntegralIsFiniteType} and will also allow us to connect the main results of Section \ref{S:FTlinks} to Milnor invariants since these are in fact invariants of embedded links that are also invariant under such crossing changes.
\end{rem}

\subsection{Smooth structure and differential forms}
\label{S:SmoothStructure}

In this subsection, we give a brief sketch of the smooth structure and differential forms on spaces of links.  The space $\Map_c(\sqcup_m\R,\R^n)$ can be given the structure of a smooth paracompact infinite-dimensional manifold (see section 3.1 of \cite{Brylinski:Loop} as well as \cite{PressleySegal:Loop}; strictly speaking, these references treat the case of maps of $S^1$, but the local picture is the same in our case).  Both $\Lk_m^n$ and $\HLk_m^n$ are open subsets of $\Map_c(\sqcup_m\R,\R^n)$, as the latter space has the $\mathcal{C}^\infty$ topology.  Using similar ideas to the ones in the references mentioned, one can (with some effort) give each of these spaces the structure of a paracompact smooth infinite-dimensional manifold.  

Another useful perspective on the smooth structure is via a \emph{diffeology} on a space $X$, which consists of a collection of maps to $X$ from open subsets $U \subset \R^k$, $k\geq 0$, called \emph{plots}, which must satisfy certain conditions.  When $X$ is a smooth manifold, one can take the plots to be precisely all the smooth maps into $X$ \cite{IglesiasZemmour}.
%  I removed H-MV-SC, as I don't think it tells us anything that we really have use for
We are interested in the case where $X$ is the infinite-dimensional manifold of smooth maps from a compact manifold $K$ to a manifold $M$.  In this case, this diffeology coincides with the diffeology where a plot $U\to X$ is precisely a smooth map $U\x K \to M$ (see Lemma A.1.7 of \cite{Waldorf:TransgrLoopI}).  In particular, for $X=\Lk_m^n$, a plot $U \to \Lk_m^n$ is a smooth map $U \x (\sqcup_m \R) \to \R^n$ such that each slice $\{u\} \x (\sqcup_m \R) \to \R^n$ is a string link.  A diffeology on $\HLk_m^n$ can be defined similarly.

For any manifold $M$, we let $\Omega^*(M)$ denote the deRham cochain complex of differential forms on $M$.  It is a differential graded algebra where the algebra structure is given by the wedge product of forms. The ground ring for all cohomology groups will be $\R$.  
This complex can be defined even when $M$ is an infinite-dimensional manifold.  For example, one can consider forms on open subsets of the topological vector space on which $M$ is locally modeled and then impose the sheaf condition to construct forms on all of $M$.  
Under certain conditions on $M$ (such as paracompactness), which are satisfied by loopspaces and hence by $\Lk_m^n$ and $\HLk_m^n$, the cohomology of this complex computes the cohomology of $M$.  See section 1.4 of \cite{Brylinski:Loop} for details.

Using the perspective of diffeology, we could equivalently define forms on $M$ using the open sets $U \subset \R^k$ mapping into $M$.  Specifically, a form $\omega$ is an assigment  of a form $\omega_\psi$ on $U$ to each $\psi\colon U \to M$ such that the assignment to a plot arising from a smooth map $h\colon V\to U$ is $h^*\omega_\psi$.  A form $\omega$ in the sense of \cite{Brylinski:Loop} mentioned above gives rise to a form in diffeology by taking $\omega_\psi = \psi^*\omega$.  Conversely, one can reconstruct a form on an infinite-dimensional manifold from its behavior on every finite-dimensional open set mapped into it.  At some point it will be convenient to think of forms on $\Lk_m^n$ and $\HLk_m^n$ in this way, namely as determined by their behaviors on finite-dimensional manifolds mapped into them.

\section{Diagram complexes for the spaces of string links and homotopy string links}\label{S:Diagrams}

%%%%%%%%%%%%%%%%%%%%%%%%%%%%%%%%%%%%%%%%%%%%%%%%%%%%%%%%%%%%%%%%%%%%%%%%%%%%%%%%%%%%%%%%%%%%%%%%%%%%%%%%

In this section we construct a diagram complex $\LD$ and a subcomplex $\HLD$ that will serve as combinatorial prescriptions for producing cohomology classes on spaces of links and homotopy links.  The complex $\LD$ has been considered before \cite{CCRL, V:B-TLinks}, but the definition of $\HLD$ appears to be new.  As mentioned in the Introduction, in this section we also fill in some details in the definition and properties of $\LD$.

%%%%%%%%%%%%%%%%%%%%%%%%%%%%%%%%%%%%%%%%%%%%%%%%%%%%%%%%%%%%%%%%%%%%%%%%%%%%%%%%%%%%%%%%%%%%%%%%%%%%%%%%

\subsection{Diagram complex for the space of string links}\label{S:LinkGraphs}

%%%%%%%%%%%%%%%%%%%%%%%%%%%%%%%%%%%%%%%%%%%%%%%%%%%%%%%%%%%%%%%%%%%%%%%%%%%%%%%%%%%%%%%%%%%%%%%%%%%%%%%%

While reading this section, the reader is encouraged to refer to Figure \ref{Fig:DiagramExample}.

For a set $S$, we let $SP_2(S)$ be the 2-fold symmetric product, 
$$
SP_2(S)=(S\times S)/\Sigma_2,
$$ 
where $\Sigma_2$ acts on the product $S\times S$ by permuting the coordinates. We think of $SP_2(S)$ as the set of nonemtpy subsets of $S$ of cardinality at most two. We denote points in $SP_2(S)$ as sets $\{s_1,s_2\}$ where $s_1,s_2\in S$, with the understanding that the cardinality of this set is one when $s_1=s_2$.

\begin{defin}
A \emph{diagram} $\Gamma$ is a triple 
$$\Gamma=(V(\Gamma),E(\Gamma),b_{\Gamma})
$$ 
where 
\begin{itemize}
\item $V(\Gamma)$ is a finite ordered set called the \emph{vertices} of $\Gamma$;
\item  $E(\Gamma)$ is a finite set called the \emph{edges} of $\Gamma$; and 
\item$b_{\Gamma}:E(\Gamma)\to SP_2(V(\Gamma))$ is a map.
\end{itemize}
For an edge $e\in E(\Gamma)$ with $b(e)=\{v,w\}$, we say that $e$ \emph{joins $v$ with $w$}. When it is clear which diagram $\Gamma$ we are speaking of we will write $(V,E,b)$ in place of $(V(\Gamma),E(\Gamma),b_\Gamma)$.
\end{defin}

The particular diagrams we study have a significant amount of extra structure. As we do not wish to impose cumbersome notation on the reader, we will continue to denote a diagram $\Gamma$ with extra structure as a triple $(V,E,b)$, despite the possible ambiguity. Before describing the extra structures, we need some definitions and terminology.

%we no longer require the definition of a partition for the definition of a link diagram
%\begin{defin}
%A \emph{partition} of a set $S$ is a set $\mathcal{P}$ of nonempty subsets of $S$ whose union is $S$ and whose intersections are pairwise disjoint. Each element $P\in\mathcal{P}$ is called a \emph{block} of the partition. An ordering for a partition is an ordering of $\mathcal{P}$ and an ordering for each block.
%\end{defin}\bfn{we may not need this definition after all; our ``partitions'' allow some sets to be empty.}

\begin{defin}
For a diagram $\Gamma=(V,E,b)$ and an edge $e\in E$, an \emph{orientation} of $e$ is a choice of injective map $b(e)\to\{-1,1\}$.
\end{defin}

Note that for an edge $e$ such that $b(e)$ consists of a single vertex, there are still two possible orientations, just as there are in the case where $b(e)$ consists of two distinct vertices.

\begin{defin}\label{D:DiagramPath}
Let $v,w$ be vertices in a diagram $\Gamma=(V,E,b)$. A \emph{path} between $v$ and $w$ is a sequence $\{e_i\}_{i=1}^k$ of edges $e_i$ such that $v\in b(e_1), w\in b(e_k)$ and $b(e_i)\cap b(e_{i+1})\neq\emptyset$ for all $i$. The \emph{length} of a path $\{e_i\}_{i=1}^k$ is equal to $k$, the number elements in the sequence of edges.
\end{defin}

Thus the orientations of edges, if they are present, are ignored for the purposes of defining a path.

One other definition we will have use for later is that of a connected component.

\begin{defin}\label{D:ConnComp}
Let $v$ be a vertex in a diagram $\Gamma=(V,E,b)$. The \emph{connected component of $\Gamma$ containing $v$} is the subdiagram $(V',E',b')$, where $V'$ is the set of all vertices $w$ that can be connected by a path to $v$, $E'$ is the set of all edges that can appear in such paths, and $b'$ is the restriction of $b$. A diagram is called \emph{connected} if it has a single connected component.
\end{defin}

%In pictures, connected components of a diagram are determined by removing the segments and taking the connected components of what is left. 

Fix integers $m\geq 1$, $n\geq 3$, and let $I_1,\ldots, I_m$ be copies of the unit interval, each of which we will call a \emph{segment} for short. The space $\sqcup_iI_i$ is an ordered set according to the natural ordering of $\{1,\ldots, m\}$ and the natural ordering of $I$. Thus for $x,y\in\sqcup_iI_i$, $x\leq y$ whenever $x\in I_i$ and $y\in I_j$ and $i<j$, and when $i=j$, $x\leq y$ if this inequality holds under the usual ordering of $I=[0,1]$.

\begin{defin}
\label{D:Diagrams}
Given integers $m\geq 1$, $n\geq 3$ as above, a \emph{link diagram} is a diagram $\Gamma=(V,E,b)$ together with the following extra structure. For the set $V$ of vertices, we have
\begin{itemize}
\item A decomposition 
$$V= V_{seg}\sqcup V_{free}$$
into ordered (possibly empty) sets, the elements of which are called \emph{segment} and \emph{free} vertices respectively. In addition, we require that the induced ordering of $(V_{seg},V_{free})$ as an ordered pair of ordered sets agrees with the ordering of $V$.
%\item A partition of 
%$$V_{end}=V_{end,1}\sqcup\cdots\sqcup V_{end,m}
%$$ 
%by the choice of a bijection $end:V_{end}\to\sqcup_i \del I_i$ which gives rise to the ordering of $V_{end}$ according to the ordering on $\sqcup_iI_i$.
\item A decomposition of
$$V_{seg}=V_{seg,1}\sqcup\cdots\sqcup V_{seg,m}
$$ 
into disjoint sets determined by the equivalence class of an injective function $seg:V_{seg}\to\sqcup_i (I_i-\del I_i)$ where $V_{seg,k}=seg^{-1}(I_k-\del I_k)$, and which gives rise to the ordering of $V_{seg}$ according to the ordering of $\sqcup_iI_i$ described above. Two such injections $s,s'$ are equivalent if they give rise to the same decomposition of $V_{seg}$ and the same ordering on each of the sets in this decomposition according to the natural ordering of $\sqcup_iI_i$.
%\item An ordering of $V$ given by the natural ordering of the ordered pair $(V_{seg},V_{free})$ of ordered sets.
\end{itemize}
For the set $E$ of edges, we have a decomposition 
$$E=E_{chord}\sqcup E_{mixed}\sqcup E_{free}\sqcup E_{loop}
$$ into 
\begin{itemize}
\item \emph{chords}, joining distinct segment vertices;
\item \emph{mixed edges}, joining a free vertex with a segment vertex;
\item \emph{free edges}, joining distinct free vertices; and
\item \emph{loops}, joining a segment vertex with itself,
\end{itemize}
 respectively. Moreover, each free vertex must have a path to a segment vertex. 
 
The \emph{valence} of a vertex $v$ is defined as follows. If $v$ is a free vertex, its valence is the number of edges joining $v$ to another vertex. If $v$ is a segment vertex, it is the number of edges joining $v$ to a vertex other than itself, plus twice the number of loops joining $v$ to itself, plus two. The valence of each vertex in a link diagram is required to be at least three. In addition,
\begin{itemize}
\item If $n$ is even, the set $E$ of edges is ordered.
%\bfn{convention about ordering free, loops, etc.? {\iv Nothing in particular, as long as all edges (which include free, mixed, chords, and loops) are ordered.}}
\item If $n$ is odd, each edge $e\in E$ is oriented.
\end{itemize}
%beginnings of more precise but uglier definition of the edge sets.
%\begin{itemize}
%\item $E_{chord}=b^{-1}(SP_2(V_{seg}-\Delta_{V_{seg}})$,
%\item $E_{mixed}=b^{-1}()$,
%\item $E_{free}=b^{-1}(SP_2(V_{free})-\Delta_{V_{free}})$, and
%\item $E_{loop}=b^{-1}(\Delta_V)$.
%\end{itemize}
\end{defin}

%Each of the vertices in the above definition will correspond to configuration points, some of which lie on the image of a link, and some off it, and edges will correspond to maps keeping track of the directions between such points in space.\bfn{eliminate or move elsewhere.}

\begin{rem}
Another terminology for segment and free vertices is ``external" and ``internal", respectively.  This is because, in the case of knots, one has diagrams consisting of only one segment and if one is working with closed knots rather than long ones, the segment is drawn as a circle and free vertices are drawn inside it -- hence ``internal".  The vertices on the circle are then ``external".  We decided that this terminology is misleading for our situation and prefer to call the vertices ``segment" (those represented as lying on the $m$ segments) and ``free" (those not lying on the segments; these abstractly correspond to configuration points that are free to move in $\R^n$).
\end{rem}

We will also distinguish ``arcs'' of a link diagram, which will be important when we define the differential, and should explain our seemingly strange definition of the valence of a segment vertex, as arcs contribute to the valence without counting as edges themselves.

\begin{defin}
For a link diagram $\Gamma$, an \emph{arc} of $\Gamma$ is a pair $(v_1,v_2)$ of distinct segment vertices with $v_1<v_2$ whose images under the injection $seg:V_{seg}\to\sqcup_i(I_i-\del I_i)$ lie in the same segment, and such that the image of no other segment vertex lies between them.
\end{defin}

We assume all possible arcs are present in any link diagram.  Although arcs are not edges, it is useful to treat them as such at times, and so for an arc $a=(v_1,v_2)$ we define $b(a)=\{v_1,v_2\}$.

%Arcs are not edges, but are thought of as contributing to the valence of a vertex, which should explain the convention for computing valence defined above. 

We pictorially represent a diagram in the plane with the intervals drawn as horizontal line segments, appearing in order from left to right and oriented from left to right, and each vertex as a point and each edge as an arc between vertices. Segment vertices are drawn on the intervals, and we think of arcs as segments in the intervals which lie between adjacent segment vertices. See Figure \ref{Fig:DiagramExample} below.

%\begin{defin}
%The \emph{valence} of a free vertex $v$ in a link diagram $\Gamma$ is defined to be the number of edges joining $v$ to a vertex other than itself plus twice the number of loops joining $v$ to itself. If $v$ is an end vertex, its valence is this sum plus one, and if $v$ is a segment vertex, its valence is this sum plus two.
%\end{defin}

\begin{figure}[h]
\input{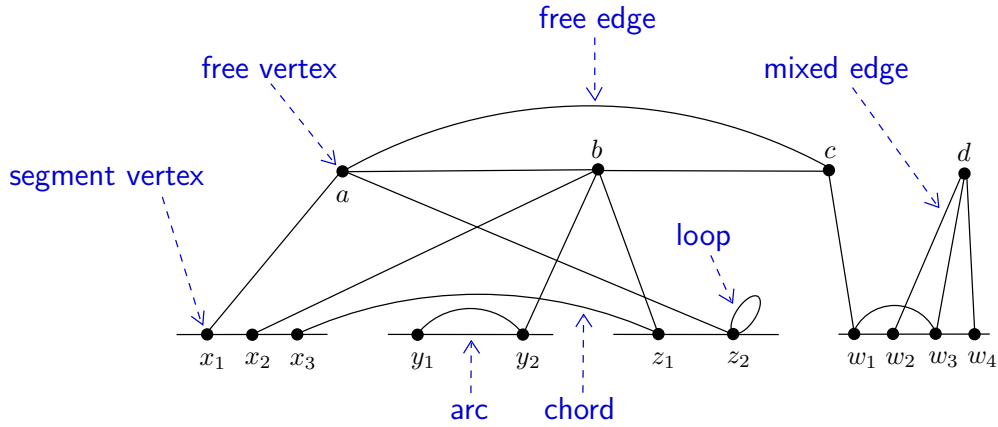}
\caption{A diagram with four segments.   Its edges may be labeled or oriented.  Each vertex is at least trivalent (the valence of, say, vertex $z_2$, is five).  
%The diagram has six edge-components.
}
\label{Fig:DiagramExample}
\end{figure}

\begin{defin}
\label{D:IsomorphismOfDiagrams}
Link diagrams $\Gamma=(V(\Gamma),E(\Gamma),b_\Gamma)$ and $\Gamma=(V(\Gamma'),E(\Gamma'),b_{\Gamma'})$ are \emph{isomorphic} if there are order-preserving bijections
$\phi_V:V(\Gamma)\to V(\Gamma')$ and $\phi_E:E(\Gamma)\to E(\Gamma')$ respecting the decomposition of the vertex set such that if $\phi_V^\ast:SP_2(V(\Gamma))\to SP_2(V(\Gamma'))$ denotes the induced map, then the diagram
$$
\xymatrix{
E(\Gamma) \ar[r]^-{b_\Gamma}\ar[d]_{\phi_E} & SP_2(V(\Gamma))\ar[d]^{\phi_V^\ast}\\
E(\Gamma')\ar[r]^-{b_{\Gamma'}} & SP_2(V(\Gamma'))
}
$$
commutes. In addition, if $n$ is odd (so that each edge is oriented), for each edge $e\in\Gamma$, the injections $o_{\Gamma,e}:b_\Gamma(e)\to\{-1,1\}$ and $o_{\Gamma',\phi_E(e)}:b_{\Gamma'}(\phi(e))\to\{-1,1\}$ must satisfy $o_{\Gamma',\phi_E(e)}\circ\phi^*_V=o_{\Gamma,e}$. In this case we say the pair $(\phi_V,\phi_E)$ is an \emph{isomorphism}.

If a pair $(\phi_V,\phi_E)$ are simply bijections (i.e., not necessarily order-preserving) satisfying all of the subsequent properties in addition to $\phi_V$ being order-preserving on the segment vertices, then we say that the pair $(\phi_V,\phi_E)$ is an \emph{isomorphism of unlabeled diagrams}.
\end{defin}

Note that ``order-preserving'' for edges is only relevant when the edge set is ordered. 

\begin{defin}
Define \emph{defect} of a link diagram $\Gamma=(V,E,b)$ to be
\begin{equation}\label{E:Degree}
\defect(\Gamma)  =   2|E|-3|V_{free}|-|V_{seg}|.
%-|v_{end}(\Gamma)|
\end{equation}
This is the first of the two gradings we will have in our bigraded complex.
\end{defin}

Notice that, because we require the valence of each vertex in a link diagram to be at least three, the defect is nonnegative (there are no free vertices whose valence is less than three and there are no segment vertices whose valence is zero; if it were otherwise, the defect could be made arbitrarily negative). Thus we can think of the defect as a measure of the failure of $\Gamma$ to be trivalent. Indeed, when $\defect(\Gamma) =0$, the segment and free vertices are precisely trivalent (in particular, $\Gamma$ cannot contain loops). We will revisit such diagrams in Section \ref{S:Degree0}. 

\begin{rem}
\label{defect}
The reader acquainted with the work in \cite{CCRL} knows that Cattaneo, Cotta-Ramusino, and Longoni call this number the \emph{degree} of a diagram.  We use the term ``defect" instead, to avoid confusion with the degree of the differential form on the space of links which a diagram gives rise to via the integration map.  Then, as suggested by the referee, we can take this latter degree to be the main degree (Definition \ref{D:MainDegree}) in our bigraded complex, which will make the integration map a map of differential \emph{graded} algebras.
\end{rem}

%\begin{rem}
%There is no need to account for the endpoints in the formula \eqref{E:Degree} since every diagram contains exactly $2m$ of them.  Thus the differential, to be defined below, raises the degree by one regardless of whether we subtract $2m$ from the right side of \eqref{E:Degree} or not.  The endpoints are a byproduct of working with long links rather than the usual closed links, and they signify the possibility of configuration points that are moving along the strands of an $m$-component link to escape to infinity.  Since this can always happen, the endpoints have to be included in every diagram.  On the other hand, unless a configuration point is being used to measure a direction to some other point, it is not needed and we discard it.  This is why we require that all segment vertices must have at least one edge emanating from them; the edges correspond precisely to the directions between configuration points.
%\end{rem}

\begin{defin}\label{D:DiagramSpaces}
%\bfn{I've used the same reference label you did for your definition}
When $n$ is even (resp. odd), define $\LD^{d}_{even}$ (resp. $\LD^{d}_{odd}$) to be real vector spaces generated by isomorphism classes of link diagrams $\Gamma$ of defect $d$ modulo subspaces generated by the relations
\begin{enumerate}
\item  If $\Gamma$ contains more than one edge joining two vertices,
%, or if $\Gamma$ has a free vertex with a loop, 
then $\Gamma=0$;
\item  If $n$ is odd  and $\Gamma$ and $\Gamma'$ are link diagrams such that a permutation of the vertices of $\Gamma'$ results in a link diagram isomorphic to $\Gamma$, then 
$$\Gamma=(-1)^{\sigma}\Gamma',$$ 
where 
\begin{align*}
\sigma= \ & \text{(order of the permutation vertices)} \\
        + &  \text{(number of edges with different orientation)};
\end{align*}
%\ifn{Both here and in the even case, the Italians also say ``the order of the (cyclic) permutation of the segment vertices".  Depending on what our labeling ends up being in the end, we might need to include the cyclic part.}
% and the number of loops with different half-edge ordering;  
\item  If $n$ is even and $\Gamma$ and $\Gamma'$ are link diagrams such that a permutation of the vertex and edge sets of $\Gamma'$ result in a link diagram isomorphic to $\Gamma$, then 
$$\Gamma=(-1)^{\sigma}\Gamma',$$ 
%in $\LD^{o,d}_{even}$ and $\BD^{o,d}_{even}$, 
where
\begin{align*}
\sigma= \ & \text{(order of the permutation of segment vertices)} \\
        + & \text{(order of the permutation of the edges)}.
\end{align*} 
\end{enumerate}
%\bfn{just to be clear: we don't count the contribution of having to reorder vertices other than the segment ones here? {\iv No we don't.  It's explained later in the context of integration why we don't care about reordering free vertices for $n$ even.}}

Finally define the graded vector spaces
$$\LD_{even}=\bigoplus_d \LD^{d}_{even}  \ \ \text{and}\ \ \LD_{odd}=\bigoplus_d \LD^{d}_{odd}.$$ 
\end{defin}

When there is no danger of confusion, i.e.~when $n$ is understood, we will refer to both $\LD^d_{even}$ and $\LD^d_{odd}$ as $\LD^d$ and to both $\LD_{even}$ and $\LD_{odd}$ as $\LD$.

\begin{rem}
The reader might argue that we should simply disallow multiple edges between a given pair of vertices rather than mod out by the subspace of such diagrams, but the differential, defined below, can introduce such edges.
\end{rem}

\begin{rem} Even in a fixed degree and after all the relations are imposed, $\LD^d$ is still an infinite dimensional vector space: Consider for example the diagram consisting of three segments with one segment vertex on each segment, and a single free vertex with three edges which join it to the segment vertices. This is a diagram of degree zero.  Overlaying copies of this diagram (that is, introducing new segment and free vertices and edges in a similar fashion ) gives an infinite list of degree zero diagrams which are clearly independent in the vector space structure.
\end{rem}

%%%%%%pasted in def'n of order here
\begin{defin}\label{D:order}
Define the \emph{order} of a diagram $\Gamma$ to be
$$
\ord(\Gamma)=|E(\Gamma)|-|V(\Gamma)_{free}|.
$$
This will be the second grading in our bigraded complex.
\end{defin} 
%It is easy to see if $\Gamma$ has order $k$, then each summand of $\delta(\Gamma)$ has order $k$ as well. Thus 
Thus for each $d=\defect(\Gamma)$ and each $k=\ord(\Gamma)$, we have subspaces $\LD^d_k \subset \LD^d$ and $\HLD^d_k \subset \HLD^d$.

Note that in the case
% $n=3$ 
%of interest here, 
of defect zero, we also necessarily have
$$
\ord(\Gamma)=\frac{1}{2}(|V(\Gamma)_{seg}|+|V(\Gamma)_{free}|).
$$
In general, a diagram $\Gamma$ of defect $d$ and order $k$ satisfies $|V(\Gamma)_{seg}|+|V(\Gamma)_{free}|=2k-d$ and $|E(\Gamma)|=|V(\Gamma)_{free}| + k$.  Hence the number of vertices of such a diagram is fixed, and the number of edges is bounded.  This means that $\LD_k^d$ and $\HLD_k^d$ are finite-dimensional for any $d,k$. 
%since the above equation fixes the number of vertices, and edges cannot be added indefinitely since a diagram with more than one edge joining a pair of vertices is set to zero.  

The following definition of degree of a diagram was suggested by the referee and motivated by the fact it coincides with the degree of the differential form on the link space resulting from applying the integration map to the diagram.

\begin{defin}\label{D:MainDegree}
Define the \emph{(main) degree} of a diagram $\Gamma$ to be
$$
|\Gamma| = (n-1)|E(\Gamma)| - n|V(\Gamma)_{free}| - |V(\Gamma)_{seg}|.
$$
Equivalently, for a diagram $\Gamma$ of defect $d$ and order $k$, 
$$|\Gamma|=k(n-3)+d.$$
We will use this definition to make (singly) graded complexes out of the bigraded complexes $\LD^*_*$ and $\HLD^*_*$.  
\end{defin}

Note that for $n=3$ the degree coincides exactly with the defect.

%%%%%%%%%%%%%%%%%%%%%%%%%%%%%%%%%%%%%%%%%%%%%%%%%%%%%%%%%%%%%%%%%%%%%%%%%%%%%%%%%%%%%%%%%%%%%%%%%%%%%%%%

\subsection{Algebraic structures on the diagram complex}\label{S:AlgebraicDiagrams}

%%%%%%%%%%%%%%%%%%%%%%%%%%%%%%%%%%%%%%%%%%%%%%%%%%%%%%%%%%%%%%%%%%%%%%%%%%%%%%%%%%%%%%%%%%%%%%%%%%%%%%%%

We now discuss the differential and the product on the space of diagrams which will make it into a differential graded algebra.
%differential graded algebra?

%%%%%%%%%%%%%%%%%%%%%%%%%%%%%%%%%%%%%%%%%%%%%%%%%%%%%%%%%%%%%%%%%%%%%%%%%%%%%%%%%%%%%%%%%%%%%%%%%%%%%%%%

\subsubsection{The differential}

%%%%%%%%%%%%%%%%%%%%%%%%%%%%%%%%%%%%%%%%%%%%%%%%%%%%%%%%%%%%%%%%%%%%%%%%%%%%%%%%%%%%%%%%%%%%%%%%%%%%%%%%

The differential of a diagram will be a signed sum of diagrams obtained from the original by ``contracting'' certain edges or arcs. We begin with some terminology and conventions.

\begin{defin}
Let $S$ be a nonempty set, and let $s,t\in S$. Define 
$$R_{t\to s}:SP_2(S)\longrightarrow SP_2(S)
$$ 
by 
$$
R_{t\to s}(T)=
\begin{cases}
T, & \mbox{if }t\notin T;\\
(T-\{t\})\cup\{s\}, & \mbox{if }t\in T.
\end{cases}
$$
\end{defin}

Thus the map $R_{t\to s}$ replaces $t$ with $s$. Let $\Gamma=(V,E,b)$ be a link diagram and $e$ be a mixed or free edge of $\Gamma$, or one of its arcs, and suppose $b(e)=\{v,w\}$, where $v<w$ in the ordering of the vertices. In case $e$ is an arc, we suppose it is represented by the pair $(v,w)$. Note that $e$ necessarily joins distinct vertices. 

\begin{defin}
With $\Gamma$ and $e$ as above, define $\Gamma/e=(V',E',b')$ to be the link diagram such that
\begin{itemize}
\item $V'=V-\{w\}$ with the induced ordering of vertices,
\item $E'=E-\{e\}$ with the induced ordering/orientation of edges (if applicable), and
\item $b'=R_{w\to v}\circ b$, restricted to $E'$.
\end{itemize}
\end{defin}

We often refer to $\Gamma/e$ as the diagram $\Gamma$ with the edge/arc $e$ contracted. The function $R_{w\to v}$ above simply replaces an edge joining $w$ with a vertex $u$ with the edge which joins $v$ to $u$ instead. This can create a loop 
%in the case of an edge $e'\neq e$ which also joins the vertices $v$ and $w$, or 
in the case of a chord between adjacent segment vertices when the arc between them is contracted. Note that the degree is increased by contraction of a mixed/free edge or arc: if $\Gamma$ has degree $d$, then $\Gamma/e$ has degree $d+1$. The differential is a signed sum of diagrams made from $\Gamma$ by contracting all possible edges and arcs. We will use the ``position'' function to help keep track of these signs.

\begin{defin}
Suppose $S$ is a finite ordered set. Define the \emph{position function} to be the unique order-preserving bijection  
$$\mathrm{pos}:S\to\{1,2,\ldots, |S|\}.
$$ 
When $x\in S$, we write $\mathrm{pos}(x)$ for the value of this function at $x\in S$, or $\mathrm{pos}(x:S)$ when we wish to emphasize the underlying ordered set $S$.
\end{defin}

%\bfn{cut and pasted pretty much everything about degree here for convenience}

\begin{defin}
The differential 
\begin{equation}\label{E:DiagramBoundary}
\delta\colon \LD^d_k\longrightarrow \LD^{d+1}_k
\end{equation}
is the unique linear extension to $\LD^d_k$ of the map defined on a diagram $\Gamma$ by
\begin{equation}
\delta(\Gamma)= \sum_{
\text{free edges, mixed edges, and arcs $e$ of $\Gamma$}}\epsilon(e)\Gamma/e.
\end{equation} 

The number $\epsilon(e)$ is equal to $\pm 1$ depending on the parity of $n$ and on the orderings of vertices and edges in the following way: Suppose the free/mixed edge or arc $e$ connects vertices $v$ and $w$.
\begin{itemize}

\item If $n$ is odd and $e$ in an edge or an arc oriented so the edge joins $v$ to $w$, then

\begin{equation}
%\label{E:OddSign}
\epsilon(e)=\begin{cases}
(-1)^{\mathrm{pos}(w:V)}, & v<w,  \\
-(-1)^{\mathrm{pos}(v:V)},  & w<v.
\end{cases}
\end{equation}

\item If $n$ is even and $e$ is a free edge or a mixed edge, then 
\begin{equation}\label{E:EvenSign}
\epsilon(e)=(-1)^{\mathrm{pos}(e:E)+|V_{free}|+1}, 
\end{equation}
and if $e$ is an arc, then 
\begin{equation}
%\label{E:OddSign}
\epsilon(e)=(-1)^{\mathrm{pos}(\max\{v,w\})}.
\end{equation}

\end{itemize}
\end{defin}

An example of the differential is given in Figure \ref{Fig:DifferentialExample}.

\begin{figure}[h]
\input{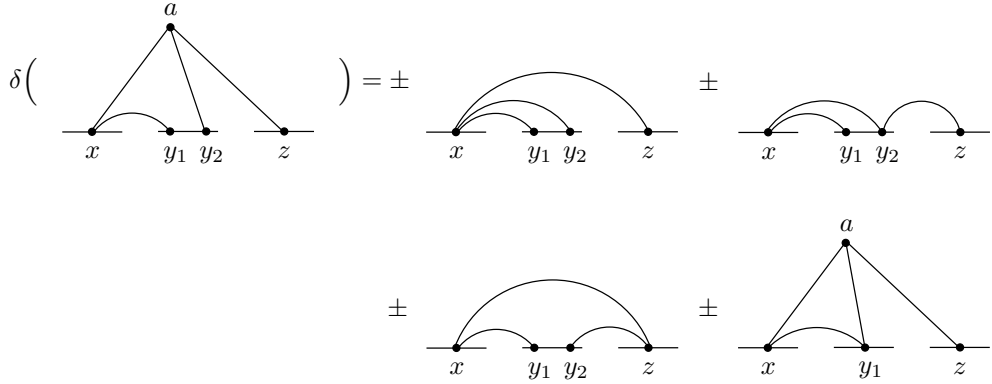}
\caption{An example of the differential. The signs depend on the parity of $n$.}
\label{Fig:DifferentialExample}
\end{figure}

It is easy to see that $\delta$ does indeed raise the defect by 1 and leaves the order unchanged.  This implies that the main degree of $\delta \Gamma$ is $|\delta \Gamma| = |\Gamma| + 1$.  Thus we can turn $\LD^*_*$ into a singly-graded complex (or differential graded algebra) where the term in grading $g$ is given by 
$$\bigoplus_{d,k: \ k(n-3)+d=g} \LD^d_k$$
(and similarly for $\HLD^*_*$).  Since each $\LD^d_k$ is finite-dimensional, each term as above is finite-dimensional.

%%%%%%%%%%%%%%%%%%%%%%%%%%%%%%%%%%%%%%%%%%%%%%%%%%%%%%%%%%%%%%%%%%%%%%%%%%%%%%%%%%%%%%%%%%%%%%%%%%%%%%%%

\subsubsection{The shuffle product}
\label{S:ShuffleProduct}
%%%%%%%%%%%%%%%%%%%%%%%%%%%%%%%%%%%%%%%%%%%%%%%%%%%%%%%%%%%%%%%%%%%%%%%%%%%%%%%%%%%%%%%%%%%%%%%%%%%%%%%%

The shuffle product on the space of diagrams associated to knots was first considered in \cite{CCRL:Struct}.  Here we extend it to link diagrams as well as provide more details about its construction.

Consider two link diagrams $\Gamma_1=(V(\Gamma_1),E(\Gamma_1),b_{\Gamma_1})$ and $\Gamma_2=(V(\Gamma_2),E(\Gamma_2),b_{\Gamma_2})$. Let 
$$seg_i:V(\Gamma_i)_{seg}\longrightarrow \sqcup_i(I_i-\del I_i)$$ be representatives of the equivalence class of the partition function for the segment vertices. Moreover, choose isomorphism class representatives for each diagram so that their vertex and edge sets are disjoint. Call an injective map 
$$
j:V(\Gamma_1)_{seg}\sqcup V(\Gamma_2)_{seg}\longrightarrow \sqcup_i (I_i-\del I_i)
$$ \emph{admissible} if its restriction to $V(\Gamma_i)_{seg}$ is in the same equivalence class as $seg_i$ for $i=1,2$.

\begin{defin}
With $\Gamma_1$ and $\Gamma_2$ and an admissible map $j$ as above, define 
$$
\Gamma_1\cdot_j \Gamma_2=(V(\Gamma_1\cdot_j \Gamma_2),E(\Gamma_1\cdot_j \Gamma_2),b_{\Gamma_1\cdot_j \Gamma_2})
$$ to be the diagram such that 
\begin{itemize}
\item The set $V(\Gamma_1\cdot_j \Gamma_2)=V(\Gamma_1)\sqcup V(\Gamma_2)$;
\item The set $E(\Gamma_1\cdot_j \Gamma_2)=E(\Gamma_1)\sqcup E(\Gamma_2)$, and the orientations (if applicable) for edges are those induced by the orientations of elements of $E(\Gamma_1)$ and $E(\Gamma_2)$;
\item The map $b_{\Gamma_1\cdot_j \Gamma_2}=b_{\Gamma_1}\sqcup b_{\Gamma_2}$;
\item The set $V(\Gamma_1\cdot_j \Gamma_2)$ is decomposed as 
$$V(\Gamma_1\cdot_j \Gamma_2)_{seg}\sqcup V(\Gamma_1\cdot_j \Gamma_2)_{free}
$$ where 
\begin{itemize}
\item $V(\Gamma_1\cdot_j \Gamma_2)_{seg}=V(\Gamma_1)_{seg}\sqcup V(\Gamma_2)_{seg}$, with ordering induced by the injection $j$, 
\item $V(\Gamma_1\cdot_j \Gamma_2)_{free}=V(\Gamma_1)_{free}\sqcup V(\Gamma_2)_{free}$, with ordering induced by the ordered pair $(V(\Gamma_1)_{free},V(\Gamma_2)_{free})$, and hence 
\item $V(\Gamma_1\cdot_j \Gamma_2)$ is ordered by the ordered pair of ordered sets $(V(\Gamma_1\cdot_j \Gamma_2)_{seg},V(\Gamma_1\cdot_j \Gamma_2)_{free})$;
\end{itemize}
\item The ordering of $E(\Gamma_1\cdot_j \Gamma_2)$ is that induced by the ordered pair $(E(\Gamma_1),E(\Gamma_2))$ of ordered sets. 
\end{itemize}
\end{defin}

\begin{defin}\label{D:Shuffle}
For link diagrams $\Gamma_1$ and $\Gamma_2$, define their \emph{shuffle product} $\Gamma_1\bullet\Gamma_2$ by 
\begin{equation}
\Gamma_1\bullet\Gamma_2=\sum_{[\text{admissible }j]}\epsilon(\Gamma_1,\Gamma_2)\Gamma_1\cdot_j\Gamma_2
\end{equation}
where the sum is over equivalence classes of admissible maps, and where
$$\epsilon(\Gamma_1,\Gamma_2)=
\begin{cases}
(-1)^{|E(\Gamma_1)||V(\Gamma_2)_{seg}|}, & \text{$n$ even}; \\
1, & \text{$n$ odd}.
\end{cases}
$$
\end{defin}

From the definition of the main degree $|\Gamma|$, it is easy to see that $|\Gamma_1 \bullet \Gamma_2| = |\Gamma_1| + |\Gamma_2|$.
%
%In order to make this product graded-commutative, we need to define a new degree, or grading, on $\LD$.  Define
%$$
%|\Gamma|=
%\begin{cases}
%|E(\Gamma)|+|V(\Gamma)_{seg}|, & \text{$n$ even}; \\ 
%|V(\Gamma)_{free}|+|V(\Gamma)_{seg}|, & \text{$n$ odd}.
%\end{cases}
%$$
%
%This new grading agrees mod 2 with the degree $2|E(\Gamma)| - |V(\Gamma)_{seg}| - 3|V(\Gamma)_{free}|$.  However, it is in general it is not equal to the degree.  This is the reason that in \refP{MapsOfAlgebras} and \refT{IntegralsAreAlgebraMaps}, the integration maps are maps of differential algebras, not differential \emph{graded} algebras.
%
Moreover, from straightforward unravellings of the definitions, one can prove the following propositions.

\begin{prop}
The shuffle product is graded-commutative; that is, 
$$
\Gamma_1\bullet\Gamma_2=(-1)^{|\Gamma_1||\Gamma_2|}\Gamma_2\bullet\Gamma_1
$$ 
\end{prop}

%\ifn{Does this, and the proposition below it, need proof or is it clear?}\bfn{As always, I'd be happy if there were proofs for these.}

\begin{prop}
The differential $\delta$ is a derivation with respect to the shuffle product. That is,
$$
\delta(\Gamma_1\bullet\Gamma_2)=\delta(\Gamma_1)\bullet\Gamma_2+(-1)^{|\Gamma_1|}\Gamma_1\bullet\delta(\Gamma_2).
$$
\end{prop}

Hence

\begin{prop}
The diagram complex $(\LD, \delta, \bullet)$  is a commutative differential graded  algebra (CDGA) with unit and its cohomology $\Ho^*(\LD)$ is thus a  commutative graded algebra.
\end{prop}

%
%Then we have
%
%\begin{thm}[\cite{V:B-TLinks}, Theorem 2.5]
%\label{T:DiagramsAreComplexes}  The coboundary $\delta$ endows $\LD$ with the structure of a cochain complex with respect to the grading by degree $d$. 
%\end{thm}

\begin{rem}
The authors of \cite{CCRL:Struct} use a different grading to make the shuffle product graded-commutative.  It is easy to check that $|\Gamma|$ (which is the degree of the form $\Gamma$ produces) agrees with their grading mod 2.
\end{rem}

\begin{rem}
There is also a coproduct on $\LD$, analogous to the one given in \cite{CCRL:Struct}.  Since we will not use this structure (shuffle product, on the other hand, will be needed in future work), we will only remark that this should give $\LD$ the structure of a Hopf algebra, and the map appearing in \refT{IntegralsAreAlgebraMaps} induces a map of Hopf algebras in (co)homology.
\end{rem}

%%%%%%%%%%%%%%%%%%%%%%%%%%%%%%%%%%%%%%%%%%%%%%%%%%%%%%%%%%%%%%%%%%%%%%%%%%%%%%%%%%%%%%%%%%%%%%%%%%%%%%%%

\subsection{A subcomplex for the space of homotopy string links}\label{S:HomotopyLinkGraphs}

%%%%%%%%%%%%%%%%%%%%%%%%%%%%%%%%%%%%%%%%%%%%%%%%%%%%%%%%%%%%%%%%%%%%%%%%%%%%%%%%%%%%%%%%%%%%%%%%%%%%%%%%

A homotopy string link need not be an embedding. As such, integration over $\HLk_m^n$ will Ênot be possible in as general a way as prescribed on the complex $\LD$ (see Section \ref{S:Bundles} for more details) due to  possible self-intersections of the components of the link.Ê In this section we will identify a subcomplex $\HLD$ of $\LD$ for which it will be possible to carry out the integration and construct elements of $\Omega^*(\HLk^n_m)$.

\begin{defin}\label{D:HoDiagrams}

Define the space of \emph{homotopy link diagrams}, denoted $\HLD$,
%\ifn{Need shorter name, possibly Milnor diagrams.{\bf Brian:} if we prove, like we say we're going to, that these give rise to Milnor's invariants, then I'm all for it.  {\iv The problem is that they give rise to all finite type invariants of homotopy links, and only some of those are Milnor invariants.}} 
to be the subspace of $\LD$ generated by diagrams $\Gamma$ which
\begin{enumerate}
\item contain no loops; and 
\item 
%(and their shuffle products) 
satisfy the condition that if there exists a path between distinct vertices on a given segment, then it must pass through a vertex on another segment.
\end{enumerate}

%
%\begin{enumerate}
%\item Diagrams containing at most one segment vertex on each segment;  
%%\item Diagrams containing no loops of edges (so those diagrams that have trivial first homology when segments are disregarded);
%\item Shuffle products of diagrams from (1) with at most one segment vertex on a common segment.
%\end{enumerate}
\end{defin}

Some examples of homotopy link diagrams are given in Figure \ref{Fig:HoDiagramExamples}.

\begin{figure}[h]
\input{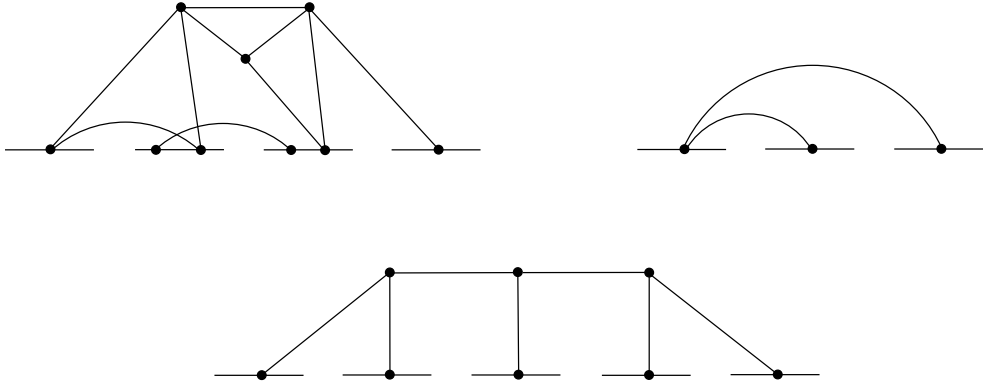}
\caption{Some examples of homotopy link diagrams (without decorations).  The bottom one is a tree of the sort that will give rise to finite type invariants in Section \ref{S:FTlinks}.}
\label{Fig:HoDiagramExamples}
\end{figure}

\begin{prop}
$\HLD$ is a differential subalgebra of $\LD$.
\end{prop}

%\begin{proof}
%We need to show that $\delta(\HLD)\subset\HLD$.  But this is clear since contractions of edges preserve segment vertices.\bfn{but we also contract arcs, which do not preserve them. \iv True that.}  Thus contracting edges on a path from one segment vertex to another that goes through a third segment only produces a shorter path, but one which still goes through all three segment vertices. \bfn{This clarifies things for me, at least a little: Note that $\delta$ identifies free vertices with segment vertices (mixed edges), free vertices with free vertices (free edges), and consecutive segment vertices on the same segment (arcs). Since a path of edges between segment vertices on distinct segments can be more or less described by a sequence of segment and free vertices (some pairs of vertices may have more than one edge), it is clear that the graphs $\Gamma'$ which appear in the expression for $\Gamma$ have this property if $\Gamma$ does. I didn't say this very well.}
%\end{proof}

\begin{proof}
%\bfn{attempt at proof of above using different language.}
To show $\HLD$ is a subcomplex of $\LD$, we must show that $\delta(\HLD)\subset \HLD$. Write $\Gamma=(V,E,b)$ and $\Gamma/e=(V',E',b')$, where $e=\{v,w\}$ with $v<w$.  Suppose $\Gamma\in \HLD$.  We want to show each term $\Gamma/e$ appearing in $\delta(\Gamma)$ is in $\HLD$.
Suppose to the contrary that $\Gamma/e$ is not an element of $\HLD$.  There are two cases.  The first case is that $\Gamma/e$ has a loop.  
% then if it is at a free vertex, 
Then either (a) $\Gamma$ itself has a loop or (b) $\Gamma$ has a chord joining adjacent vertices on a segment or (c) $\Gamma$ has multiple edges between a pair of vertices.  Situations (a) and (b) are impossible since $\Gamma\in HLD$.  In situation (c),  $\Gamma$ is set to zero, so $\delta(\Gamma)$ is also set to zero (so no terms $\Gamma/e$ appear in $\delta(\Gamma)$).  This covers the first case.

The second case is that $v_1,v_2$ are distinct segment vertices lying on the same segment of $\Gamma/e$, and there is a path $\alpha=\{e_i\}_{i=1}^k$ of edges from $v_1$ to $v_2$ which does not pass through a vertex on a different segment than the one on which $v_1$ and $v_2$ lie. In this case it is enough to show that there is a path between vertices on the same segment in $\Gamma$ which also does not pass through a vertex lying on a different segment.

Let $\alpha=\{e_i\}_{i=1}^k$ be a path in $\Gamma/e$ as above, of minimal length.  We have $v_1\in b'(e_1),v_2\in b'(e_k)$ and $b'(e_i)\cap b'(e_{i+1})\neq\emptyset$ for all $i$.  We may assume that $v_1,v_2$ are segment vertices in $\Gamma$, for otherwise $e$ joins $v_1$ or $v_2$ to a segment vertex, and then $\{e,e_1,\ldots, e_k\}$ or $\{e_1,\ldots, e_k,e\}$ is a path in $\Gamma$ joining vertices on the same segment without passing through another segment.   Now if $\alpha$ has the property that $b(e_i)\cap b(e_{i+1})\neq\emptyset$ for all $i$, then $\alpha$ itself is a path between $v_1$ and $v_2$ in $\Gamma$, contradicting the fact that $\Gamma\in\HLD$.  
So 
%$\alpha$ is a sequence of edges in $\Gamma$ such that $v_1\in b(e_1)$ and $v_2\in b(e_k)$. 
let $j$ be the smallest integer such that $b(e_j)\cap b(e_{j+1})=\emptyset$.  We have $b'(e_j)\cap b'(e_{j+1})\neq\emptyset$, and $b'=R_{w\to v}\circ b$ for some $v,w$, so necessarily $w\in b(e_j)$ or $w\in b(e_{j+1})$. Without loss of generality assume $w\in b(e_j)$. Then it must be that $v\in b(e_{j+1})$, and in this case the edge $e$ satisfies $b(e_j)\cap b(e)\neq\emptyset$ and $b(e_{j+1})\cap b(e)\neq\emptyset$.  Since $\alpha$ has minimal length, $\{e_1,\ldots, e_j,e,e_{j+1}, \ldots, e_k\}$ forms a path $\alpha'$ in $\Gamma$ between $v_1$ and $v_2$ in $\Gamma$.  
%Note that the only new vertices which $\alpha'$ passes through are the vertex $w$ itself. 
We will be done if we can argue that $w$ cannot be a segment vertex lying on a segment different from $v_1$ and $v_2$ %unless the path $\alpha$ already passed through such a segment vertex. 
But this is clear: if $w$ is such a vertex, then $v$ is such a vertex in $\Gamma/e$, and
%since $v<w$, we must have that $v$ is also a vertex lying on the same segment (because $e$ is necessarily the arc between them), in which case 
the original path $\alpha$ passes through this segment, a contradiction.

That $\HLD$ is closed under the shuffle product is clear since this product does not create new paths of edges.  
%Namely, each path of edges in every summand of $\Gamma_1\bullet\Gamma_2$ comes from either $\Gamma_1$ or $\Gamma_2$, so if condition (2) is satisfied for these diagrams, it is satisfied for the product.  Similarly, if  $\Gamma_1$ and $\Gamma_2$ contain no loops, the same will be true for $\Gamma_1\bullet\Gamma_2$.
\end{proof}

A few words of clarification and justification for \refD{HoDiagrams} are in order.  Our definition of $\HLD$ excludes diagrams which contain a chord connecting two vertices on a single segment.  It also excludes all possible diagrams which, via contractions of edges, might produce such a chord.  
%In addition, chords can produce loops via contractions of arcs, and hence loops are excluded.  Thus chords and anything that has to do with them (via the differential) is excluded.  
What we are trying to capture geometrically are linking phenomena which ``ignore'' the knotting of each strand. 
%First, when we say a ``path from a segment", this of course means that a path starts from a segment vertex.  It is also understood that, while a path is traversed, a vertex has to be entered and departed via different half-edge.  Our definition thus eliminates diagrams which contain paths starting at a segment vertex and end at a segment vertex on the same segment, while passing through no vertices (i.e.~they are chords) or going through free vertices only. It is possible that the initial and the final segment vertex are the same, in which case we have eliminated the possibility of loops of edges (since each free vertex is required to have a path to a segment vertex, we can use one of these paths to construct an illegal loop going up that path, around the loop of free vertices, and back down the same path). We have also eliminated the loops on segment vertices (traversing a loop since at a segment vertex means that we leave it via one half-edge and come back to it via a different one).  
The reason for this is simple: there is no knotting of individual strands in $\HLk_m^n$, as they may pass through themselves.  Once integration over diagrams is defined in Section \ref{S:Integrals}, it will be clear that a chord between segment vertices captures something about linking between those segments.  So when the segment vertices lie on the same segment, this means a chord between them captures something about self-linking, or knotting, of that segment.  
%A path between segment vertices on the same segment that passes only through free vertices will also attempt to capture knotting information (these situations corresponds to points in $\R^n$ being able to move around a single strand in a sphere of directions and thus produce nontrivial forms).  
Similarly, integrals that correspond to loops will also only contain information about single strands.

\subsection{Diagram complexes in defect zero}\label{S:Degree0}

%%%%%%%%%%%%%%%%%%%%%%%%%%%%%%%%%%%%%%%%%%%%%%%%%%%%%%%%%%%%%%%%%%%%%%%%%%%%%%%%%%%%%%%%%%%%%%%%%%%%%%%%

In Section \ref{S:FTlinks} we will focus on the case $n=3$ of classical links to see which link invariants (elements of $\operatorname{H}^0(\Lk_m^3)$ and $\operatorname{H}^0(\HLk_m^3)$) can be obtained via configuration space integrals from our diagram complexes.  As we will see in Section \ref{S:Integrals}, when $n=3$, defect zero diagrams will correspond to degree zero forms (although in general the degree of forms corresponds to the main degree, not the defect), so we want
\begin{equation}
0 =   2|E(\Gamma)|-3|V(\Gamma)_{free}|-|V(\Gamma)_{seg}|.
%-|v_{end}(\Gamma)|
\end{equation}
It was already noted in the discussion following equation \eqref{E:Degree} in Section \ref{S:LinkGraphs} that these are precisely the trivalent diagrams.

%To see what is in the kernel of the differential in degree zero, we first introduce another grading on $\LD$ and $\HLD$ in order to restrict our attention to certain finite-dimensional subspaces of these spaces.  

%%%%%%%%%%%%%%%removed def'n of order from here

Let 
\[
\Ho^0(\LD^*_k):=Z^0(\LD^*_k) :=\ker (\delta: \LD^0_k \to \LD^1_k)
\]
denote the subspace of degree zero cocycles (i.e.~degree zero cohomology classes) in the complex $\LD^*_k$.  Similarly, let $\Ho^0(\HLD^*_k)$ denote the subspace of degree zero cocycles (i.e.~degree zero cohomology classes) in the complex $\HLD^*_k$.
(Note that in either case, for $n>3$ this is not the same as the degree zero cocycles of the singly graded complex graded by $|\Gamma|$.)

To understand the kernel of the differential $\delta: \LD^0_k \to \LD^1_k$, we will examine the cokernel of its adjoint (i.e.~dual) $\delta^*: (\LD_k^1)^* \to (\LD_k^0)^*$.  Let 
\[
\Ho_0(\LD^*_k):=\coker (\delta^*: (\LD^1_k)^* \to (\LD^0_k)^*)
\]
and similarly, let $\Ho_0(\HLD^*_k)$ denote the corresponding cokernel in the complex $\HLD^*_k$.

%Each of $\LD_k^0$ and $\HLD_k^0$ (for any $k$) has a canonical basis, up to the sign of each element (i.e.~an orthogonal transformation of the vector space), given by one diagram for each isomorphism class of unlabeled diagram.  
%We can take this same set of diagrams as a basis for the dual vector space $(\LD_k^0)^*$ (resp. $(\HLD_k^0)^*$) of funxionals on diagrams.  Let $\Gamma^*$ denote such a basis element given by the diagram $\Gamma$.  
Let $|\Aut(\Gamma)|$ denote the size of the group of automorphisms of $\Gamma$ as an unlabeled diagram. As defined at the end of Definition \ref{D:IsomorphismOfDiagrams}, these are automorphisms of graphs without any labels or edge orientations, but they must fix the segment vertices pointwise.
Consider the inner product $\LD^d_k \otimes \LD^d_k \to \R$ which on diagrams is given by 
$$\langle \Gamma_1, \Gamma_2\rangle = \delta_{\Gamma_1, \Gamma_2} |\Aut(\Gamma_1)|  \,\,(= \delta_{\Gamma_1, \Gamma_2}|\Aut(\Gamma_2)| ),$$ 
where $\delta$ here is the Kronecker $\delta$.
This gives an isomorphism $\xymatrix{ \LD^d_k  \ar[r]^-\cong & (\LD^d_k)^*}$ via $\Gamma \mapsto \langle \Gamma, - \rangle$.  Thus we can represent elements of $(\LD^d_k)^*$ by linear combinations of diagrams.  We will write $\Gamma^*$ for the element of $(\LD^d_k)^*$ which is the image of the diagram $\Gamma$ under this isomorphism.
%This gives isomorphisms $\LD_k^0\cong (\LD_k^0)^*$ and $\HLD_k^0 \cong (\HLD_k^0)^*$, and hence isomorphisms $\ker \delta \cong \coker \delta^*$.  
% in each degree.  
%\ifn{We may want to really show that this is the adjoint, but I'm ok with leaving it like this.  Also, like you mentioned in an older footnote, do we need a reference for this kernel-cokernel relationship between the differential and its adjoint?  It's clear to me that this is true from the various emails you wrote.}
%\textbf{Warning:} 
For the rest of Section \ref{S:Diagrams}, a drawing of a diagram $\Gamma$ will often mean the element $\Gamma^*$.  
%We may write either ``diagram" or ``funxional" to clarify whether we are talking about an element of $\LD_k^d$ (resp. $\HLD_k^d$) or an element of $(\LD_k^d)^*$ (resp. $(\HLD_k^d)^*$).

To understand $\delta^*: (\LD^1_k)^* \to (\LD^0_k)^*$, note that any diagram in $\LD^1_k$ has precisely one 4-valent vertex, as shown in the top of  Figures \ref{Fig:STURelation}, \ref{Fig:IHXRelation}, and \ref{Fig:1TRelation}. In Figures \ref{Fig:STURelation} and \ref{Fig:1TRelation}, $i$ is a segment vertex, and in Figure \ref{Fig:IHXRelation} it is a free vertex.
It is not hard to see that the adjoint $\delta^*$ ``blows up" four-valent vertices in all possible ways, as shown in these three figures.  In the first two figures, there are ${4 \choose 2}/2=3$ possibilities, corresponding to the possible ways of pairing four vertices.
%In the case of $(\LD^0_k)^*$, 
The image of $\delta^*$ is generated by three types of (linear combinations of) diagrams.  Each type of generator is a sum of the diagrams shown in one of the figures with certain coefficients to be determined.  The signs arise from the labeling conventions associated to edge contractions (in particular recall that free vertices always have higher labels than segment ones, so $i<j$ in the left picture on the bottom of of Figure \ref{Fig:STURelation}).  In the first two of these figures, each diagram resulting from the blowup of a vertex is the same outside of the pictured portions as the other two diagrams in the triple.

%\ifn{In STU relation, the first diagram has that sign because it's assumed $i<j$.  So we have to have free vertices always labeled with higher labels than segment vertices, which is the case.}

\begin{figure}[h]
\input{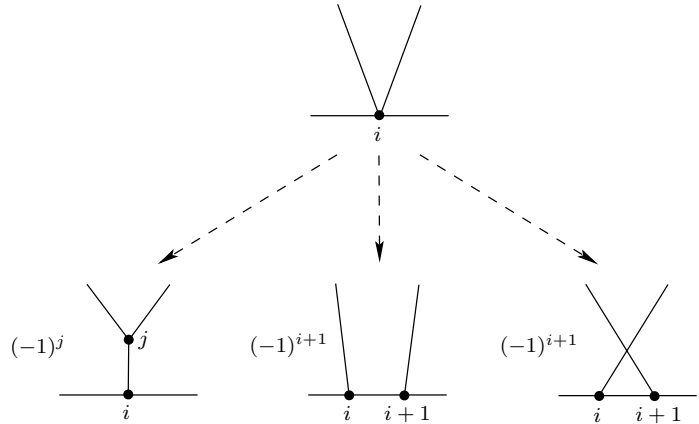}
\caption{Blowups giving rise to the STU relation.}
\label{Fig:STURelation}
\end{figure}

\begin{figure}[h]
\input{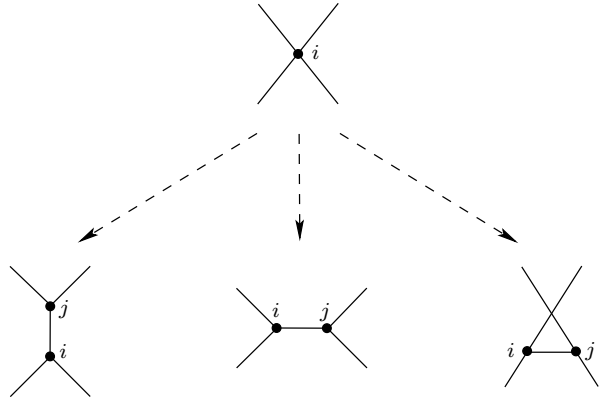}
\caption{Blowups giving rise to the IHX relation.}
\label{Fig:IHXRelation}
\end{figure}

\begin{figure}[h]
\input{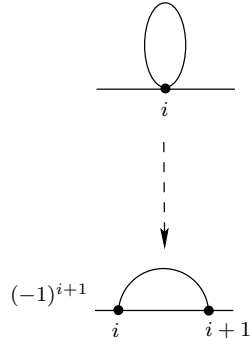}
\caption{Blowup giving rise to the 1T relation.}
\label{Fig:1TRelation}
\end{figure}

Now we determine the coefficients in each sum that gives a generator of the image of $\delta^*$.  Call the three diagrams in Figure \ref{Fig:STURelation} (with the indicated signs) $S$, $T$, and $U$.  Call the three diagrams in Figure \ref{Fig:IHXRelation} (with all signs $+1$) $I$, $H$, and $X$.  Call the diagram in Figure \ref{Fig:1TRelation} $1T$.  

\begin{prop}
\label{P:AutFactors}
The image of $\delta^*$ is generated by elements of the form 
\begin{itemize}
\item
$S^* + T^* + U^*$, 
\item
$ I^* +  H^* + X^*$, and 
\item
$(1T)^*$.
\end{itemize}
\end{prop}

The statement of this Proposition is certainly not new.  For example, it appears in \cite[Section 3]{Long:Classes}.  It is a consequence of the following folklore result: Given a graph complex whose differential $\delta$ is a signed sum of edge contractions, the image of $\delta^*$ in the dual complex is the signed sum of edge expansions (i.e.~with all coefficients $\pm 1$), provided the duality is given by $\langle \Gamma, \Gamma'\rangle = \delta_{\Gamma, \Gamma'} |\mathrm{Aut}(\Gamma)|$.  (For the $1T$ term, note that, by Definition \ref{D:IsomorphismOfDiagrams}, the diagram with the loop has the same number of automorphisms as the diagram with the isolated chord, i.e.~for the parts pictured, there are no nontrivial automorphisms.)  We could not find a proof of either this more general statement or the statement of Proposition \ref{P:AutFactors} (the``unitrivalent case"), so we prove the Proposition here.

%Move proof to appendix?
\begin{proof}[Proof of \refP{AutFactors}]
The proof amounts to checking that the linear combination we get from blowing up the 4-valent vertices in Figures \ref{Fig:STURelation} (the STU case), \ref{Fig:IHXRelation} (the IHX case), and \ref{Fig:1TRelation} (the 1T case) is in fact the one with all coefficients equal to 1.

\emph{The 1T case}:  Let $L$ denote the 4-valent diagram in the top of Figure \ref{Fig:1TRelation}.  Here the diagram $1T$ with the isolated chord is the only diagram whose image under $\delta$ contains $L$.  (The only other potential such diagram is not in the complex, since by Definition \ref{D:Diagrams}, loops can only occur at segment vertices.)  Thus  
\begin{equation}
\label{PairingDeltaL}
 %\delta^* L^* (\Gamma) = 
 \langle \delta^* L^*, \Gamma\rangle = \langle L, \delta \Gamma\rangle = \left \{ \begin{array}{ll} |\Aut (L)|  & \mbox{if } \Gamma = 1T \\ 0 & \mbox{for all other } \Gamma, \end{array} \right.
\end{equation}
recalling that we take the pairing given by $\langle \Gamma_i, \Gamma_j \rangle = \delta_{ij} |\Aut(\Gamma_i)|$.  Since $\Aut(L) \cong \Aut(1T)$, this implies that $\delta^* L^*= (1T)^*$.

\emph{The STU case}:  Let $V$ denote the 4-valent diagram in the top of Figure \ref{Fig:STURelation}.  Note that $\delta$ of any of the three diagrams $S, T, U$ contains precisely one diagram $V$, and that these are the only diagrams whose image under $\delta$ can have such a term.  Thus  
\begin{equation}
\label{PairingDeltaV}
% \delta^* V^* (\Gamma) = 
  \langle \delta^* V^*, \Gamma\rangle = \langle V, \delta \Gamma\rangle = \left \{ \begin{array}{ll} |\Aut (V)|  & \mbox{for all } \Gamma\in \{S,T,U\} \\ 0 & \mbox{for all other } \Gamma, \end{array} \right.
\end{equation}

There are two cases:
\begin{enumerate}
\item  All three diagrams $S,T,U$ are in distinct isomorphism classes;
\item  $T$ is isomorphic to $U$.  
\end{enumerate}

%In case (1), this implies that $\delta V^* = S^* + T^* + U^*$.
%In case (2), this implies that $\delta^* V^* = S^* +T^* = S^* + \frac{1}{2}(T^* + U^*)$.  

Now note that in either case, $\Aut(T) \cong \Aut(U)$, since any automorphism must fix the free vertices attached to $i$ and $i+1$, and the diagrams agree outside of the pictured part.  Similarly, in either case, $\Aut(S) \cong \Aut(V)$ since an automorphism of $S$ must fix the edge from $i$ to $j$, and since $S$ and $V$ agree outside of that edge.  We next analyze $\Aut(S)$.  Under automorphisms of $S$, the vertices attached to $j$ must either have singleton orbits or be in the same 2-point orbit.  Note that the case of singleton orbits corresponds precisely to case (1) above, while a 2-point orbit corresponds to case (2).  In case (1), $\Aut(S) \cong \Aut(T) \,\,( \cong \Aut(U))$, for the same reason that $\Aut(T)\cong \Aut(U)$.  Hence $|\Aut(V)|=|\Aut(\Gamma)|$ for all $\Gamma \in \{S,T,U\}$.  Thus equation (\ref{PairingDeltaV}) implies that in this case $\delta^* V^* = S^* + T^* + U^*$.

In case (2), the index $[\Aut(S): \Aut(S)_k]=2$, where $k$ is one of the vertices attached to $j$ and where $\Aut(S)_k$ denotes the subgroup of $\Aut(S)$ fixing $k$.  But $\Aut(S)_k \cong \Aut(T)$ .  So in this case $|\Aut(V)|=|\Aut(S)|=2|\Aut(T)| = 2|\Aut(U)|$.  So in this case, we conclude $\delta^* V^* = S^* + 2T^* = S^* + 2U^* = S^* + T^* + U^*$.
%The result for the STU case now follows from our above analysis of what $\delta^*$ is in cases (1) and (2).  (Remember that we are working over $\R$, so multiplying generators by arbitrary integers gives equivalent generators.)   

\emph{The IHX case}:  Let $F$ be the 4-valent diagram in the top of Figure \ref{Fig:IHXRelation}.  Let $e$ be the edge pictured in any of the three diagrams $I,H,X$ (by abuse of notation).  In this IHX case, it is possible that $\delta$ of any of the diagrams $\Gamma=I, H, X$ has more than one term isomorphic to $F$.  In fact, the number of such terms in $\delta \Gamma$ is given by the size of the orbit of $e$ (considered as an unordered pair) under the automorphism group of $\Gamma$.  This number is $[\Aut(\Gamma):\Aut(\Gamma)_e]$, where $\Aut(\Gamma)_e$ is the subgroup of automorphisms taking $e$ to itself (the subgroup of $\Aut(\Gamma)$ which fixes $\Delta$ setwise).  Thus we have 
\begin{equation}
\label{PairingDeltaF}
% \delta^* F^* (\Gamma) = 
  \langle \delta^* F^*, \Gamma\rangle = \langle F, \delta \Gamma\rangle = \left \{ \begin{array}{ll} [\Aut(\Gamma):\Aut(\Gamma)_e] |\Aut (F)|  & \mbox{for all } \Gamma\in \{I,H,X\} \\ 0 & \mbox{for all other } \Gamma \end{array} \right.
\end{equation}

This equation tells us that if we write $\delta^* F^* = \sum_\Gamma c_\Gamma \Gamma^*$, then we have $$c_\Gamma |\Aut(\Gamma)|= [\Aut(\Gamma):\Aut(\Gamma)_e] |\Aut (F)|.$$ 

%If all of $I,H,X$ are distinct, then these are the factors by which to multiply the terms $I^*, H^*, X^*$.  However, sometimes some of these diagrams are isomorphic, as happened in the STU case, and this is what remains to be resolved.

We now analyze the index above using another subgroup of $|\Aut(\Gamma)|$.  Let $\Delta$ be the subgraph of $\Gamma=I, H$, or $X$ (again abusing notation) consisting of $e$, the edge joining $i$ and $j$, and the edges incident to the endpoints of $e$.  %For any $\Gamma$, $\Aut(\Gamma)$ acts on the set of (combinatorial) embeddings of $\Delta$ into $\Gamma$.  
Let $\Aut(\Gamma)_\Delta$ denote the subgroup of $\Aut(\Gamma)$ which fixes every vertex of $\Delta$.    
%Call this group $G_\Gamma$ (where $\Gamma=I,H$ or $X$).
%What was going to be called that??
Clearly $\Aut(\Gamma)_\Delta < \Aut(\Gamma)_e$ since $\Aut(\Gamma)_\Delta$ fixes $e$ (even as an ordered pair).  So we can consider the index $[\Aut(\Gamma)_e: \Aut(\Gamma)_\Delta]$.    This index is the order of the group $\Aut(\Gamma)_e|_\Delta$ of automorphisms in $\Aut(\Gamma)_e$ restricted to automorphisms of $\Delta$; in other words, it is the group of restrictions to $\Delta$ of automorphisms of $\Gamma$ that fix $\Delta$ setwise.  Abbreviate this group $G_\Gamma$.  Since $G_\Gamma$ is a subgroup of 
$$\Aut(\Delta) \cong \Sigma_2 \wr \Z/2 = \Z/2 \ltimes (\Z/2 \x \Z/2) \cong D_4$$ 
(the group of symmetries of a square), the index in question is either 1,2,4, or 8.  
%But one can check that the intersections of any two of $\Aut(I)_e|_\Delta, \Aut(H)_e|_\Delta, \Aut(X)_e|_\Delta$ have order at most 2.
%Probably something slightly different should be intersected

We divide our argument into two cases:
\begin{enumerate}
\item all of $I,H,X$ are distinct isomorphism classes;
\item at least two of $I,H,X$ are isomorphic.
\end{enumerate}

\emph{Claim}: for $\Gamma \in \{I,H,X\}$, the indices $[\Aut(\Gamma)_e: \Aut(\Gamma)_\Delta]= |G_\Gamma|$
%differs by at most a factor of 2 for any two $\Gamma \in \{I,H,X\}$ and that 
are all equal in case (1); in case (2), two of the diagrams are isomorphic, and this index is twice as large for the third diagram as for either of the two isomorphic ones.  
%Let $V\cong \Z/2 \x \Z/2$ be the subgroup generated by the automorphism of $\Delta$ that swaps each pair of vertices adjacent to 

\emph{Proof of claim}:  Let $\sigma, \tau \in D_4$ denote the two elements that only swap two vertices (which as vertices of $I, H$, or $X$ must be adjacent to the same endpoint of $e$).  Let $\rho \in D_4$ be be a rotation\footnote{Note that this does \emph{not} correspond to a rotation of any of the pictures of $\Delta$ in Figure \ref{Fig:IHXRelation}.} by $\pi/2$.  Identify $D_4$ with $\Aut(\Delta)$ via an embedding of $\Delta$ as shown in $I$.

First notice that $H \cong X$ if and only if at least one of $\sigma, \tau, \rho$ (or $\rho^{-1}$) is in $G_I$.  One can also check that for any other element $\alpha \in D_4$ (meaning for $\alpha \in \{\rho^2, \rho\sigma, \sigma\rho\}$), we always have
\begin{equation}
\label{RemainingD4Elements}
\alpha \in G_I \iff \alpha \in G_H \iff  \alpha \in G_X.
\end{equation}  
Thus in case (1), none of $\sigma, \tau, \rho, \rho^{-1}$ is in $G_I$, and whatever remaining elements are in $G_I$ are also in $G_H$ and $G_X$.  Interchaging the roles of $I,H,X$, we get that $G_I \cong G_H \cong G_X$, which is what we wanted to show.

To finish case (2), suppose $H\cong X$.  Note first that none of $\sigma, \tau, \rho, \rho^{-1}$ can be an element of $G_H$ or $G_X$; thus in this case, $|G_I| \geq 2|G_H| \,\,(=2|G_X|)$.   However, if both $\sigma$ and $\tau$ are elements of $G_I$, then their product $\sigma \tau (= \tau \sigma = \rho^2)$ is in both $G_H$ and $G_X$.   This together with (\ref{RemainingD4Elements}) implies $|G_I| \leq 2 |G_H| (=2|G_X|)$.  So in this case  $|G_I| = 2|G_H| = 2 |G_X|$.  Interchanging the roles of $I,H,X$  finishes the proof of the Claim in case (2).
%Suffices to show it has to be less than 8 because then whenever one of them has 4 the other two have to have 2.
%Would be sufficient, but isn't true.

The right-hand side of equation (\ref{PairingDeltaF}) can be rewritten:
\begin{align*}
[\Aut(\Gamma): \Aut(\Gamma)_e] |\Aut(F)| 
&=|\Aut(\Gamma)| |\Aut(F)| / |\Aut(\Gamma)_e|  \\
&=|\Aut(\Gamma)|  |\Aut(F)| |\Aut(\Gamma)_\Delta| / |G_\Gamma|
\end{align*}

So the coefficient $c_\Gamma$ of $\Gamma^*$ in $\delta^* F^*$ is equal to $\left( |\Aut(F)| |\Aut(\Gamma)_\Delta| \right) / |G_\Gamma|$.  Note that the groups $\Aut(\Gamma)_\Delta$ are isomorphic for all $\Gamma\in \{I,H,X\}$ since the diagrams agree outside of the pictures.  Thus the quantity in parentheses is independent of $\Gamma$.

By the Claim, we see that in case (1), the coefficients $c_\Gamma$ are the same for all $\Gamma \in \{I,H,X\}$.  Since we are working over $\R$, we can $F$ divide by this number to get an equivalent generator with all $c_\Gamma$ equal to 1.  In case (2), we may suppose again without loss of generality that $H\cong X$.  In this case, we showed that $|G_I| = 2|G_H| = 2|G_X|$.  Thus $2 c_I = c_H = c_X$.  So an appropriate multiple of $\delta^* F^*$ is equal to $I^* + 2 H^* = I^* + 2 X^* = I^* + H^* + X^*$.
\end{proof}
%end proof that could have been in the appendix

Proposition \ref{P:AutFactors} implies that $\Ho_0(\LD^*_k)$ is the quotient of $\LD^0_k$ by all diagrams of the three types listed in its statement.  
Equivalently, $ \Ho^0(\LD^*_k) = \ker \delta$ is for each $k$ generated by trivalent diagrams 
such that the pairing with any of these three types of diagrams gives zero.
The three types of relations by which we quotient to get $\Ho_0(\LD^*_k)$ are called the \emph{STU relation}, the \emph{IHX relation}, and the \emph{1T relation}.  We will sometimes also use this terminology to describe the conditions that diagrams in $\Ho^0(\LD^*_k)$ must satisfy (see also Remark \ref{R:Dualizing}).

%The same relations but without the factors $|\Aut(\Gamma)|$ are customarily called the STU, IHX, and 1T relations.

%Now consider the subspace of $\LD$ generated by the \emph{STU, IHX}, and \emph{1T relations} pictured in Figures \ref{Fig:STURelation}, \ref{Fig:IHXRelation}, and \ref{Fig:1TRelation} (in the STU and IHX relations, the three diagrams are supposed to be the same outside the pictured portions).
%\ifn{There's also the ``closure" relation (see Survey of B-T integration paper) which could be useful.  It essentially says that even if segment vertices are ``shuffled", we can redraw the vertices so that the connected components are next to each other, rather than mixed up.  This probably corresponds to being able to use Fubini's Theorem and integrate over components kind of separately.}

\begin{rem}\label{DropIHXand1T}
Bar-Natan \cite{BN:Vass} has shown that the IHX relation
%and 1T relations 
follows from the STU relation.  
%His proof can be easily adapted to show that the modified IHX 
%%and 1T 
%relation follows from the modified STU relation.
\end{rem}

We now consider the case of $\Ho^0(\HLD^*_k)$, where there are some additional observations to be made. First, the 1T relation is now vacuous since $\HLD$ contains no diagrams with chords connecting vertices on the same segment. Second, suppose that the two loose edges in the top diagram of Figure \ref{Fig:STURelation} belong to a loop of edges with all vertices except $i$ free. This is depicted in Figure \ref{Fig:HoLinkLoop}.

%We want to similarly take the subspace of $\HLD$ generated by STU and IHX  relations (1T relation is vacuous in this case since homotopy link diagrams do not contain chords on single strands), but we first need to check that all the diagrams involved in those relations are indeed defined in $\HLD$. 
%
%\begin{prop}\label{P:RelationsDescend}
%The STU and IHX relations descend to $\HLD$.  
%\end{prop}
%\ifn{Maybe this can be said better.  Maybe ``descend" is not the right word.}
%
%\begin{proof}
%We need to check that the IHX and STU relations do not require a diagram that has a connected component with more than one vertex on a segment. For the IHX relation, there is no issue because all the diagrams in this relation require the same segment vertices.  
%
%For the STU relation, if vertices $i$ and $i+1$ in the second or third picture belong to the same connected component of the diagram, then the first diagram would contain a loop of edges going through vertex $j$ (since all the diagrams look the same outside the pictured portions).  But diagrams containing loops are zero.
%\end{proof}
%
%\begin{rem}
%The above proof of course also works for diagrams in $\LD$ and shows that the two segment vertices in the second and third diagram in the STU relation always lie in different connected components.   
%\end{rem}

\begin{figure}[h]
\input{HoLinkLoop.pstex_t}
\caption{}
\label{Fig:HoLinkLoop}
\end{figure}

Then blowing up vertex $i$ can only result in one diagram, namely the (leftmost) diagram $S^*$ from the STU relation.  The other two would correspond to diagrams with paths between two segment vertices on the same segment that only go through free vertices, and such diagrams are not elements of $\HLD$.  We thus reduce the STU relation in $\Ho^0(\HLD^*_k)$ to the condition that the diagram in Figure \ref{Fig:SpecialHoLinkLoop} pairs to zero with any diagram in $\Ho^0(\HLD^*_k)$.
%, i.e., its corresponding funxional vanishes on diagrams in $\Ho^0(\HLD^*_k)$.

\begin{figure}[h]
\input{SpecialHoLinkLoop.pstex_t}
\caption{}
\label{Fig:SpecialHoLinkLoop}
\end{figure}

This relation extends to all diagrams with loops of free edges and not just those that are separated from a segment by a single mixed edge.  Namely, the STU relation can be applied repeatedly to any path between the loop of free edges and a segment (there are always such paths since every free vertex must have a path to a segment vertex) and the situation can be reduced to that of Figure \ref{Fig:SpecialHoLinkLoop}.  An example is given in Figure \ref{Fig:HoLinkLoopExample}, where ``$=0$" again means that this diagram pairs to 0 with any diagram.
%its corresponding funxional vanishes on any diagram in $\Ho^0(\HLD^*_k)$.  

\begin{figure}[h]
\input{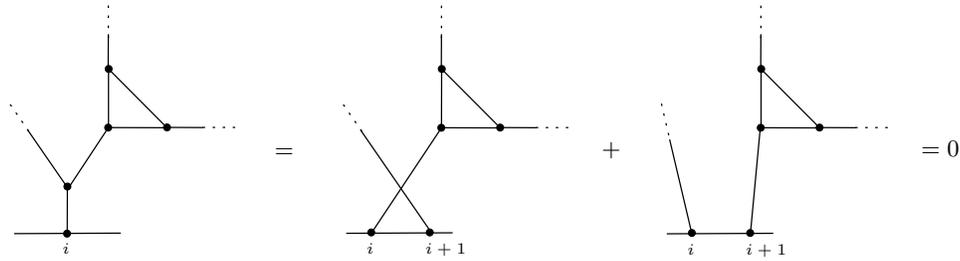}
\caption{An example of how a diagram with a loop of free edges pairs to 0 with any diagram.}
\label{Fig:HoLinkLoopExample}
\end{figure}

\begin{rem}\label{R:OnlyWayToGetLoopsRevisited}
At first glance, it might seem that the diagram from Figure \ref{Fig:HoLinkLoop} should not be permitted in $\HLD$ since repeated contractions of its edges would eventually produce a loop at vertex $i$, and loops have been excluded from $\HLD$.  However, such contractions would first produce a double edge between vertex $i$ and another free vertex, and a diagram with a double edge would already be zero by definition of $\LD$.  %This is precisely what was already discussed in Remark \ref{R:OnlyWayToGetLoops}.
\end{rem}

\begin{rem}\label{R:SwitchingChords}
There is another interesting consequence of the STU relation in $\HLD^0_k$ which we will have use for in future work when we study Milnor Invariants in more detail.  Namely, suppose that the same two loose edges in the top diagram of Figure \ref{Fig:STURelation} end on the same segment.  In other words, suppose the picture is as in Figure \ref{Fig:STUHoLinkBlowup}, where the dots indicate that there might be other segment vertices between those pictured.

\begin{figure}[h]
\input{STUHoLinkBlowup.pstex_t}
\caption{}
\label{Fig:STUHoLinkBlowup}
\end{figure}

Then the STU relation gotten from blowing up this diagram is just $T^* + U^*=0$, since the $S$ diagram contains a path between two segment vertices that goes through only a free vertex.  We thus get a special case of the STU relation in $\HLD^0_k$, given in Figure \ref{Fig:SpecialSTUHoLink}, i.e.~the two sides of the equation are equal on all diagrams in $\HLD^0_k$.

\begin{figure}[h]
\input{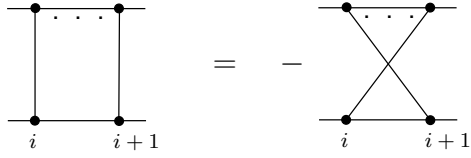}
\caption{A consequence of the STU relation in $\Ho^0(\HLD)$.}
\label{Fig:SpecialSTUHoLink}
\end{figure}

\end{rem}

We now collect the observations made so far.  Recall that we use the pairing on diagrams given by $\langle \Gamma_i, \Gamma_j\rangle = \delta_{ij} |\Aut(\Gamma_i)|$.

\begin{prop}\label{P:DegreeZero}  For each $k\geq 0$,

\begin{itemize}  
\item  $\Ho^0(\LD^*_k)$ consists of all linear combinations $\alpha = \sum \>a_i\Gamma_i$ of trivalent diagrams $\Gamma_i$ 
%modulo the 
which satisfy the

\vskip 4pt
\begin{itemize}
\item STU relation, i.e., $(S^*+ T^* + U^*)(\alpha) \equiv \langle S + T + U, \alpha\rangle=0$
\item IHX relation, i.e., $(I^*+ H^* + X^*)(\alpha) \equiv \langle I + H + X, \alpha\rangle=0$, and
\item 1T relation, i.e., $( (1T)^*) (\alpha) \equiv \langle 1T, \alpha \rangle=0$.
\end{itemize}

\vskip 6pt
\item  $\Ho^0(\HLD^*_k)$ consists of all linear combinations $\alpha = \sum \>a_i \Gamma_i$ of trivalent diagrams $\Gamma_i$ 
%modulo the
which satisfy the

\vskip 4pt
\begin{itemize}
\item STU relation and
\item IHX relation, and
%\item 1T relation;
\item ``H1T relation", which is that $\Gamma^*(\alpha) = \langle \Gamma, \alpha\rangle= 0$ for any $\Gamma$ containing a closed path of edges.
%\item The relation from Figure \ref{Fig:SpecialSTUHoLink}.
\end{itemize}

\end{itemize}
\end{prop}

\begin{rems}\label{R:Dualizing} \

\begin{enumerate}

\item  
Regarding descriptions of $\Ho^0(\LD^*_k)$ and $\Ho^0(\HLD^*_k)$ in previous literature, one difference is that Mellor \cite{Mellor-HoInv} and Mellor-Thurston \cite{MelThurs-HoInv} work with the variant of $\HLD^0_k$ consisting of unitrivalent diagrams without segments and without the STU relation (but they keep the other relations).  In fact, the only reason we listed the last relation for $\Ho^0(\HLD^*_k)$ (we could have left it out since it follows from the STU relation) is so that our description would exactly match those in \cite{Mellor-HoInv, MelThurs-HoInv}.

\item Define a \emph{tree} to be a connected diagram such that there is a unique path of minimal length between any pair of distinct vertices, and define a \emph{leaf} to be a mixed edge or chord of a tree (so a leaf has at least one associated segment vertex). Define a \emph{forest} to be a diagram whose connected components are all trees. Since elements of $\HLD^0_k$ are (sums of) trivalent diagrams without loops, every element is a sum of forests, each of whose trees has at most $m$ leaves, where $m$ is the number of distinct segments, and such that the segment vertices associated with the leaves all lie in distinct segments (that is, there is at most one segment vertex on each segment for a given tree in the forest). This was alluded to in the description of Figure \ref{Fig:HoDiagramExamples}, where the bottom diagram is such a tree.

%Because of the requirement that diagrams are trivalent and contain no loops, $\HLD^0_k$ is given precisely by forests with up to $m$ leaves (recall that $m$ is the number of segments), with each leaf ending on a different segment, modulo the relations above. By a tree we here mean the collection of vertices and edges, but not segments,  where the leaves are the segment vertices, i.e.~mixed edges.  A forest is, as usual, a disjoint union of trees. (For each tree in our forests, the root can be taken to be the leftmost segment vertex.)
% This was alluded to in the description of Figure \ref{Fig:HoDiagramExamples}, where the bottom diagram is such a tree.
\end{enumerate}
\end{rems}

%------------------------------------

%In order to obtain real-valued (rather than diagram-valued) invariants, we will be considering funxionals on $\LD^0_k$ and $\HLD^0_k$.

\begin{defin}\label{D:WeightSystems}
Define the space of \emph{degree $k$ link weight systems $\LW_k$} as the vector space 
$$((\LD^0_k)^*/(STU, IHX, 1T))^*,$$ 
where $(-)^*$ denotes the dual vector space, and where $STU$ is the relation that $S^*+T^*+U^*=0$, etc.  Similarly, define the space of \emph{degree $k$ homotopy link weight systems $\HLW_k$} as the vector space 
$$((\HLD^0_k)^*/(STU, IHX, H1T))^*.$$
%of $W \in (\LD^0_k)^*$ satisfying
%\begin{itemize}
%\item $W(S+T+U)=0$
%\item $W(I+H+X)=0$
%\item $W(1T)=0$.
%\end{itemize}
% of $W\in (\HLD^0_k)^*$ satisfying 
%\begin{itemize}
%\item $W(S+T+U)=0$
%\item $W(I+H+X)=0$
%\item $W$ vanishes on diagrams containing a closed loop of edges.
%\end{itemize}
\end{defin}

%Using Proposition \ref{P:DegreeZero} and (the inverse to) the isomorphism $\xymatrix{\LD^d_k \ar[r]^-\cong & (\LD^d_k)^*}$ given by $\alpha \mapsto \langle - , \alpha \rangle$, we immediately conclude the following:
Since a vector space is canonically isomorphic to its double dual, we have the following.

\begin{prop}
There are canonical isomorphisms 
$$\LW_k \cong \Ho^0(\LD^*_k) \ \ \ \ \text{and} \ \ \ \  \HLW_k\cong \Ho^0(\HLD^*_k).$$
\end{prop}

Since $(\LD^0_k)^*$ and $(\HLD^0_k)^*$ are spaces of diagrams, we can think of a weight system $W$ as a functional on diagrams such that $W(S^*+T^*+U^*)=0$, etc.  This is how weight systems are typically defined, and this is how we will think of them in Section \ref{S:FTlinks}, where we will denote elements (diagrams) of $(\LD^0_k)^*$ and $(\HLD^0_k)^*$ by letters without the superscripts $*$.

\begin{rems} \
\begin{enumerate}
\item
The real reason we introduced the grading by order is that weight systems of order $k$ are precisely finite type $k$ invariants; see Theorems \ref{T:UniversalFTLinks} and \ref{T:UniversalFTHoLinks}.
%%Robin: I had put in this item below.
%%\item
%%In the context of the graph complex, the space of weight systems is most naturally described as a \emph{quotient} of funxionals \emph{modulo} some relations.  (The difference between the two descriptions may be attributed to the fact that the original work on the graph complex treated the case $n>3$ without fully relating it to the known results in the case $n=3$.)
\item 
%One can identify $\Ho^0(\LD^*_k)$ with any of the descriptions of $\LW_k$ (and $\Ho^0(\HLD^*_k)$ with $\HLW_k$). 
%%as is commonly done in literature:   By construction, the kernel of the differential $\delta$ in degree zero can be identified with the space of all weight systems; an element of the kernel is a linear combination of trivalent grafphs, where the coefficient a graph is the value of a weight system on this graph. 
The above identification of weight systems with cocycles of diagrams can be used to reconcile integration from the graph complex with the integration of weight systems commonly found in the literature on finite type invariants.  That is, the map $\Ho^0(\LD^*_k) \to \Ho^0(\Lk_m^3)$ can be thought of as a map $\LW_k\to \Ho^0(\Lk_m^3)$.   
%with $\LW_k$ as the vector space of funxionals satisfying the STU, IHX, and 1T relations.  
We will discuss this in \refS{FTlinks}.
\end{enumerate}
\end{rems}

%------------------------------------

We make one last observation, which we will use in \refS{FTlinks}.  We defined $\LW_k$ as the dual to a quotient of $(\LD^0_k)^*$ by certain relations.
Instead of considering trivalent diagrams modulo these relations, one can reduce to the case of diagrams containing only chords, i.e.~\emph{chord diagrams}.  
That is, note that in $(\LD^0_k)^*/(STU, IHX, 1T)$, any trivalent diagram $\Gamma$ can be rewritten as a sum of chord diagrams using the STU relation repeatedly.  The resulting complex inherits a different relation as follows:  Because the trivalent diagram in the STU relation can have both of its ``loose" edges also ending in segments (necessarily different segments in the case of $\HLD$), applying the STU relation twice gives what is know as the \emph{4T} relation, depicted in Figure \ref{Fig:4TRelation}.

  \begin{figure}[h]
\input{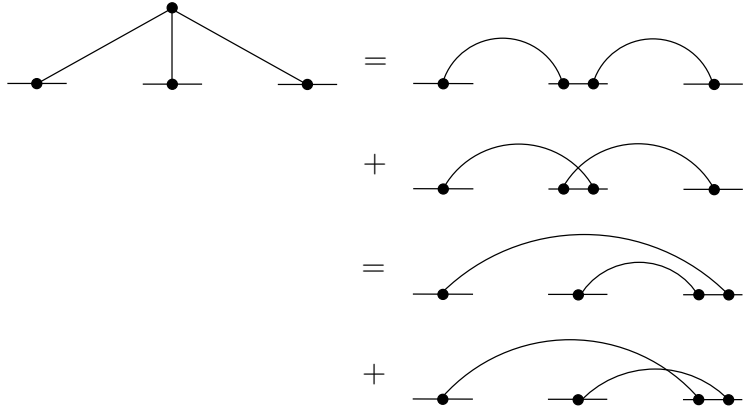}
\caption{Applying the STU relation to the middle and the right mixed edge produces the equality of the two pairs of chord diagrams.  Any time four chord diagrams differ in two places as pictured, one obtains such an equality, called the 4T relation.   The three arcs belong to distinct segments in the case of $\HLD$ and some or all of them could belong to same segment in the case of $\LD$.  An arbitrary permutation of the order of the three arcs in all the pictures is allowed.}
\label{Fig:4TRelation}
\end{figure}

Denote by
 $$
\mathcal{LC}^0_k \ \ \ \text{and}\ \ \ \mathcal{HC}^0_k
$$
the $\R$-vector spaces generated by chord diagrams on $m$ segments with $k$ chords ending on $2k$ distinct vertices (since defect zero implies trivalence, two chords cannot end in a common segment vertex).  For the latter space, there can be no chords with both endpoints on the same segment.  We will call these the \emph{link chord diagrams} and \emph{homotopy link chord diagrams}.  As in the case of  trivalent diagrams, the duals $(\mathcal{LC}^0_k)^*$ and $(\mathcal{HC}^0_k)^*$ can be identified as spaces of chord diagrams.
%In the same way that we identify weight systems with elements of $\Ho^0$ of the graph complex, we can identify diagrams with elements of $\Ho_0$ of the graph complex.  That is, the map $\Gamma \mapsto \langle \Gamma, - \rangle$ associates to a diagram a funxional, which represents an element of $\Ho_0(\LD^*_k)$ (or $\Ho_0(\HLD^*_k)$).
Using the relationship between the STU relation and the 4T relation, we have the following straightforward generalization of \cite[Theorem 6]{BN:Vass}.  

\begin{thm}\label{T:Trivalent=Chord}
There are isomorphisms
\begin{align*}
%\Ho_0(\LD^*_k) & \cong 
 (\LD^0_k)^*/(STU, IHX, 1T)\cong (\mathcal{LC}^0_k)^*/(4T,1T) \ \ \ \  \text{ and } \ \ \ \ 
%\Ho_0(\HLD^*_k) & \cong 
(\HLD^0_k)^*/(STU, IHX, H1T) \cong (\mathcal{HC}^0_k)^*/ 4T.
\end{align*}
Each isomorphism sends a diagram with no free vertices to itself and a diagram with free vertices to the sum of chord diagrams obtained from it via the STU relation.
\end{thm}

%----------------

%Note that we have only listed the STU and 1T relations in the statement of the theorem when describing $\Ho^0(\LD)_k$, since all other relations follow from the STU relation.  Also, since there are no chords on single segments in homotopy link diagrams,  1T relation is vacuous and that is why it was removed from the second part of the statement.

Now denote by 
$$
\mathcal{LCW}_k\ \ \ \text{and}\ \ \ \mathcal{HCW}_k
$$
the vector spaces of functionals on $(\mathcal{LC}^0_k)^*/(4T, 1T)$ and $(\mathcal{HC}^0_k)^*/4T$, respectively. Dualizing \refT{Trivalent=Chord}, we thus have isomorphisms
\begin{equation}\label{E:WeightIsos}
\LW_k\cong\mathcal{LCW}_k\ \ \ \text{and}\ \ \ \HLW_k\cong\mathcal{HCW}_k.
\end{equation}

\refT{Trivalent=Chord} will be used in the proof of \refT{UniversalFTHoLinks}.

%------------------------------------

%%%%%%%%%%%%%%%%%%%%%%%%%%%%%%%%%%%%%%%%%%%%%%%%%%%%%%%%%%%%%%%%%%%%%%%%%%%%%%%%%%%%%%%%%%%%%%%%%%%%%%%%

\section{Configuration space integrals and cohomology of homotopy string links}\label{S:ChainMap}

%%%%%%%%%%%%%%%%%%%%%%%%%%%%%%%%%%%%%%%%%%%%%%%%%%%%%%%%%%%%%%%%%%%%%%%%%%%%%%%%%%%%%%%%%%%%%%%%%%%%%%%%

%%%%%%%%%%%%%%%%%%%%%%%%%%%%%%%%%%%%%%%%%%%%%%%%%%%%%%%%%%%%%%%%%%%%%%%%%%%%%%%%%%%%%%%%%%%%%%%%%%%%%%%%

\subsection{Compactification of configuration spaces}\label{S:Compactification}

%%%%%%%%%%%%%%%%%%%%%%%%%%%%%%%%%%%%%%%%%%%%%%%%%%%%%%%%%%%%%%%%%%%%%%%%%%%%%%%%%%%%%%%%%%%%%%%%%%%%%%%%

In this section we review the standard construction of a compactification of configuration spaces over which we will integrate to produce invariants. This is necessary since integrals over the ordinary open configuration space may not converge.  The original compactification is due to Fulton and MacPherson \cite{FM} and Axelrod and Singer \cite{AS}.
%, but we follow Sinha's \cite{S:Compact} alternative construction (also considered by Kontsevich and Soibelman \cite{KoSo}).

\begin{defin}
For a manifold $M$, let
$$
C(p,M)=\{ (x_1, x_2, ..., x_p)\in M^p \colon x_i\neq x_j\ \text{for}\ i\neq j\}
$$
be the configuration space of $p$ points in $M$. When $M=\R$, the configuration space has $p!$ components, and in this case $C(p,\R)$ will mean the component consisting of those $(x_1,\ldots, x_p)$ such that $x_1<\cdots<x_p$.  Similarly, when $M=S^1$, $C(p, S^1)$ will mean one component where the points $x_1,\ldots, x_p$ are in a fixed cyclic order.
\end{defin}
%\ifn{This is where we somehow have to make things consistent with the segment labels on diagrams.  It seems that here we only allow linear labeling on configuration points, but if we are going to define the shuffle product like the Italians do, then it seems that the labels can be mixed and not strictly increasing.  Maybe the only care that has to be taken is spelling out how the labels on the diagram correspond or are associated to configuration points via the integration map.}\bfn{i don't understand what you're saying. i don't see how this is somehow incompatible with the shuffle product.}

For a submanifold $Y$ of a manifold $X$, the blowup $\mathrm{Bl}(X,Y)$ is the result of removing $Y$ and replacing it by the sphere bundle of its normal bundle.  Equivalently, this is the result of removing an open tubular neighborhood of $Y$.

%\footnote{The Fulton--MacPherson compactification in algebraic geometry is defined exactly as the Axelrod--Singer compactification, but with projective blowups instead of real oriented blowups.}.     

\begin{defin}
\label{D:Compactification}
For a compact manifold $M$, the (Fulton--MacPherson) compactification $C[p,M]$ is defined as the closure of the image of 
\[
C(p,M) \xymatrix{\ar@{^(->}[r] & } M^p \x \prod_{S\subset \{1,...,p\} \\ |S|\geq 2} \mathrm{Bl}(M^S, \Delta_S)
\]
where $\Delta_S=\{(x,x,\ldots, x)\in M^S\}$ is the thin diagonal in $M^S$.  
For $M=\R^n$, $C[p,\R^n]$ is considered as the subspace of $C[p+1, S^n]$ where the last point is fixed at $\infty$.  
\end{defin}

First, here are some general properties of $C[p,\R^n]$ that are relevant for our purposes.  Proofs can be found in \cite{S:Compact}:
\begin{enumerate}
\item The space $C[p, \R^n]$ is a manifold with corners homotopy equivalent to $C(p, \R^n)$;
\item The boundary of $C[p, \R^n]$ is given by points colliding or escaping to infinity;
\item The directions and relative rates of collision are recorded, so that a $k$-stage collision (points coming together or going to infinity in $k$ different stages rather than all of them doing this at the same instance)  gives a point in a codimension $k$ stratum of $C[p, \R^n]$.  These $k$ stages are the \emph{screens} explained below.
\end{enumerate}

The last property in particular says that codimension one faces of $C[p, \R^n]$ consist of configurations where some subset of the points has come together or escaped to infinity at the same time. These faces are of particular interest since they play a role in checking whether some differential form obtained on the space of links is closed (i.e.~they are relevant for an application of Stokes' Theorem).

%\rfn{Should I remove some of these last two paragraphs in view of the stuff I added below?}

Some elaboration is necessary in order to define configuration space integrals for string links.  A stratum of $C[p,M]$ is labeled by a collection $\{S_1,...,S_k\}$ of distinct subsets $S_i\subset \{1,...,p\}$ with $|S_i|\geq 2$ and satisfying the condition 
\[
S_i \cap S_j\neq \emptyset \Rightarrow \mbox{ either } S_i \subset S_j \mbox{ or } S_j \subset S_i  
\]
In other words, the $S_i$ are pairwise nested or disjoint.  For each set $S_i$ in the collection, we can think of the points in $S_i$ as having collided.  If there is an $S_j \subset S_i$ in the collection, we can think of the points in $S_j$ as having first collided with each other and then with the remaining points in $S_i$.  Two strata indexed by $\{S_1,...,S_k\}$ and $\{S'_1,...,S'_j\}$ intersect precisely when the set $\{S_1,...,S_k, S'_1,...,S'_j\}$ satisfies the above condition.  In that case, that is the set which indexes the intersection.

Roughly speaking, each $S_i$ corresponds to an ``infinitesimal configuration" or \emph{screen}, and all the screens together encode directions and relative rates of collision, as follows.  
Let $s_i=|S_i|$.  
In the case where all $S_i$ are disjoint, the screen corresponding to $S_i$ is a point $\vec{u}_{S_i}\in (C(s_i, T_x M))/(\R^n \rtimes \R_+)$, where $\R^n \rtimes \R_+$ is the group of translations and (oriented) scalings of $\R^n \cong T_x M$.  
In the case where all the $S_i$ are nested, say as $S_1 \subset ... \subset S_k$, the screen $\vec{u}_{S_i}$ is a point in $(C(s_i - s_{i-1}+1, T_x M))/(\R^n \rtimes \R_+)$ (where we set $s_0=0$).
In general, $\vec{u}_{S_i}$ is a configuration of points in $T_x M$, modulo the action of $(\R^n \rtimes \R_+)$.  
Each of the $p$ points in a limiting configuration (i.e., a configuratoin in the boundary of $C[p, \R^n]$) corresponds to a point in possibly multiple screens; if the point is indexed by $j\in \{1,...,p\}$, it corresponds precisely to one point in each $S_i$ that contains $j$.
The number of points in $\vec{u}_{S_i}$ is obtained by taking the points in $S_i$ and, for each maximal proper $S_j$ contained in $S_i$, replacing the points in $S_j$ by a single point; i.e., all the points in $\vec{U}_{S_i}$ become one point in $\vec{U}_{S_j}$ when $S_i \subset S_j$.
From this description, one can verify that the stratum labeled $\{S_1,...,S_k\}$ has codimension $k$.  Again, more precise details can be found in  \cite{FM} and \cite{AS}.

\begin{rem}
An alternative but equivalent definition of this compactification was given by Sinha in \cite{S:Compact}, which is as follows.  Suppose $M$ is a compact submanifold of $\R^N$.  Then for all $1\leq i<j<k\leq p$ we have maps 
\begin{equation}\label{E:CompactMaps}
v_{ij}=\frac{x_j-x_i}{|x_j-x_i|}\in S^{N-1},\ \ \ \ \ a_{ijk}=\frac{|x_i-x_j|}{|x_i-x_k|}\in [0,\infty],
\end{equation}
whose domain is $C(p,M)$ and where $[0,\infty]$ denotes the one-point compactification of $[0,\infty)$. These maps measure the direction and relative rates of  collision of configuration points respectively.  Adding this information to the configuration space is achieved by considering the map \begin{align}
\gamma\colon   C(p, M) & \longrightarrow M^p\times (S^{N-1})^{p \choose 2}  \times  [0,\infty]^{p  \choose 3}\label{E:CompactificationMap}  \\
             (x_1, ..., x_p) & \longmapsto (x_1, ..., x_p, v_{12}, ..., v_{ij}, ..., v_{(p-1)p}, a_{123}, ..., a_{ijk}, ..., a_{(p-2)(p-1)p}).\notag
\end{align}

%\begin{defin}\label{D:Compactification}
The closure of the image of $\gamma$ turns out to be diffeomorphic (as a manifold with corners) to $C[p, M]$. That is,
$$C[p,M]=\overline{\gamma(C(p, M))}\subset M^p\times (S^{N-1})^{{p \choose 2}}\times  [0,\infty]^{p  \choose 3}.$$  
%\end{defin}

Since $C[p,\R^n]$ is defined via $C[p+1, S^n]$, we would have to take $N$ above to be $n+1$ if we were to use this definition.  Then the unit vector difference maps $v_{ij}$ would land in $S^n$, rather than the more geometrically obvious candidate, $S^{n-1}$.  (If one tries to use $S^{n-1}$, the maps $v_{ij}$ cannot extend from $\R^n$ to $S^n$; this is why \refD{Compactification} is better suited for our purposes.)
%Since $\R^n$ is itself not compact, we think of $C[p, \R^n]$ as the subspace of $C[p+1, S^n]$ where $S^n=\R^n\cup\infty$ and the last configuration point is the point at infinity.  This is why we will have to consider the faces obtained by points escaping to infinity.
\end{rem}

\subsection{Bundles of compactified configuration spaces}\label{S:Bundles}

%%%%%%%%%%%%%%%%%%%%%%%%%%%%%%%%%%%%%%%%%%%%%%%%%%%%%%%%%%%%%%%%%%%%%%%%%%%%%%%%%%%%%%%%%%%%%%%%%%%%%%%%

%We are interested in defining a cochain map, called $I_{\Lk}$, from $\LD$ to the deRham complex of the space of links, $\Omega^*(\Lk_m^n)$. It will be constructed in roughly two steps as follows:  first, given $\Gamma\in\LD$, we will construct a certain bundle of configuration spaces over $\Lk_m^n$; second, we integrate along the fiber of this bundle a certain differential form.  There is already a standard recipe for doing this which was initiated in the case of knots ($m=1$) in \cite{BT} and fully developed in \cite{CCRL}.  A straightforward generalization to links ($m>1$) was then given in \cite{V:B-TLinks}.
%
%However, this way of defining $I_{\Lk}$ is not suitable for producing forms on $\HLk_m^n$ by restriction to the subcomplex $\HLD$.  We cannot even build a suitable bundle for integration along the fiber, so that the map $I_\Lk$ does not even induce a map from $\HLD$ to the deRham complex on the space of homotopy string links. This will be illustrated in Section \ref{S:BundlesVerticesEdges}.

Given $\Gamma\in\LD$, we will construct in this section a certain bundle of configuration spaces over $\Lk_m^n$.  There is already a standard recipe for doing this which was initiated in the case of closed knots ($m=1$) in \cite{BT} and fully developed in \cite{CCRL}.  Generalizing this recipe to long knots or closed links is straightforward.  In generalizing to homotopy string links, two issues arise.
First, some care needs to be taken to extend this construction to ordinary string (i.e.~long) links.  
%A straightforward generalization to links ($m>1$) was then given in \cite{V:B-TLinks}. 
%Actually not so straightforward and not correctly done there.
The second and perhaps more serious issue is that even after extending to string links, this construction fails to even produce a bundle over $\HLk_m^n$ by restriction to the subcomplex $\HLD$ as we will see in Section \ref{S:BundlesVerticesEdges}.

%, one cannot restrict $I_{\Lk}$ defined this way to the subcomplex $\HLD$ and obtain a map $I_{\HLk}$ whose image lies in $\Omega^*(\HLk_m^n)$.

%  In other words, if we define our bundles as was done in 
%\cite{CCRL} or \cite{V:B-TLinks}, we will not get a commutative diagram
%\begin{equation}\label{E:CommutingIntegrals}
%\xymatrix{
%\HLD   \ar@{^{(}->}[r] \ar[d]_{I_{\HLk}} &  \LD \ar[d]^{I_{\Lk}} \\
%\Omega^*(\HLk_m^n) \ar[r] &  \Omega^*(\Lk_m^n)
%}
%\end{equation}\bfn{it doesn't seem to me that the issue is whether or not this diagram commutes. rather, it seems to be that the map $I_H$ defined as the pullback as above isn't really a map at all, at least not one we can immediately understand in terms of integration due to various issues.}

Resolving the first issue essentially just relies on our definition of string links, in which different components approach infinity in different directions (see Definition \ref{D:MapSpace}), as well as properties of the Fulton-MacPherson compactification.  We take the standard bundle construction, as in \cite{CCRL, V:B-TLinks}, as our starting point, and we describe how to make the construction work for string links in Section \ref{S:BundlesVertices}.

To fix the second issue, we devise a more refined way of constructing bundles which works over both $\Lk_m^n$ and $\HLk_m^n$.  
%(this material essentially comes from Section 3.2 of  \cite{V:B-TLinks}).  
In Section \ref{S:BundlesVerticesEdges} we refine this construction to produce bundles over spaces of homotopy string links. The difference between the two approaches can be summarized very succinctly:  in the standard approach, only vertices of a diagram are taken into account in the construction of bundles, whereas in the new approach, we will take into account both vertices and edges. We will show the compatibility of the approaches in Section \ref{S:Integrals}.

%in such a way from now want to define bundles of configuration spaces over $\Lk_m^n$ and $\HLk_m^n$ such that a diagram in $\LD$ (and $\HLD$) gives a way to integrate this bundle along the fiber. Intuitively, the vertices of our diagrams correspond to configuration points, some of which, the segment vertices, are on the (homotopy) link, and others of which, the free vertices, are not. The edges between these vertices give a prescription for pulling back forms which will be integrated along the fibers of these bundles. What we will construct, then, are configuration spaces, some of whose points lie on the strands of a link and the others of which are free to move in $\R^n$. The cases of $\Lk_m^n$ and $\HLk_m^n$ are different, so we begin with the easier case of $\Lk$.

%%%%%%%%%%%%%%%%%%%%%%%%%%%%%%%%%%%%%%%%%%%%%%%%%%%%%%%%%%%%%%%%%%%%%%%%%%%%%%%%%%%%%%%%%%%%%%%%%%%%%%%%

\subsubsection{Bundles of compactified configuration spaces from vertices of a diagram}\label{S:BundlesVertices}

%%%%%%%%%%%%%%%%%%%%%%%%%%%%%%%%%%%%%%%%%%%%%%%%%%%%%%%%%%%%%%%%%%%%%%%%%%%%%%%%%%%%%%%%%%%%%%%%%%%%%%%%

A diagram $\Gamma\in\LD$ will define a configuration space where the segment vertices of $\Gamma$ correspond to points moving along a link in $\R^n$ and free vertices correspond to points that are free to move anywhere in $\R^n$.

Suppose $\Gamma\in\LD$ has $i_j$ segment vertices on the $j$th segment, $1\leq j\leq m$, and $s$ free vertices.
For any link $L \in \Lk^n_m$, the evaluation map
\begin{equation}
%\label{E:LinkEvaluation}
ev_\Gamma(L) \colon \prod_{j=1}^m C(i_j, \R)\longrightarrow C\left[\sum_{j=1}^m i_j, \R^n\right]
\end{equation}
%on the interiors of $C[i_j, \R]$ 
is given by evaluating the $j$th strand of $L$ on $i_j$ configuration points.  In other words, it is given by
$$
\left(L,(x^1_1,\ldots, x^1_{i_1}),\ldots, (x^m_1,\ldots, x^m_{i_m})\right)\mapsto \left(L(x^1_1),\ldots, L(x^1_{i_1}),\ldots, L(x^m_1),\ldots, L(x^m_{i_m}),\right).
$$

Let $C\left[ \vec{i}; \coprod_{j=1}^m \R \right]$ denote the closure of the image of $ev_\Gamma(L)$,
%the restriction of $ev_\Gamma$ to $\{L\} \x \prod_{j=1}^m C(i_j, \R)$ for any $L \in \mathcal{L}^n_m$, 
where we think of $\vec{i}$ as $(i_1,\ldots,i_m)$.  Suppressing the dependence on $L$ will be justified by the next lemma.  So far it is clear that for any $L\in \Lk^n_m$, $C\left[ \vec{i}; \coprod_{j=1}^m \R \right]$ is compact and that its interior is diffeomorphic to $\prod_{j=1}^m C(i_j, \R)$.

\begin{lemma}
For any $L\in \mathcal{L}^n_m$, the space $C\left[ \vec{i}; \coprod_{j=1}^m \R \right]$ has the structure of a manifold with corners, independent of $L$.
\end{lemma}
\begin{proof}
We will show that the manifold with corners structure comes from that on $C[i_1+\ldots+i_m+1, S^n]$.  First note that all the added limit points are in the boundary of $C[i_1+\ldots+i_m, \R^n]$.  
%Where all the $Q$ points are away from $\infinity \in S^3$, $C_{q_1,...,q_k} \left[\coprod_1^k \R\right] $ locally looks like $C_{q_1}[S^1] \x ... \x C_{q_k}[S^1]$.  So it suffices to consider the corner structure when some points are near infinity.   
Around such a point, a neighborhood in $C[i_1+\ldots+i_m, \R^n]$ has various strata, points of which are described by collections of screens, as was exaplained after Definition \ref{D:Compactification}.  To describe the corresponding neighborhood in $C\left[ \vec{i}; \coprod_{j=1}^m \R \right]$, we replace these spaces of screens by similar spaces which have lower dimension, but will have the same codimension in the latter space.  

At a collision of $s\geq 2$ points at a point $x$ away from $\infty$, the space $C(s, T_x\R^n)/(\R^n \rtimes \R_+)$ is replaced by $C(s, T_x L)/(\R \rtimes \R_+)$, where by abuse of notation $L$ also denotes the image of $L$ and where $\R \rtimes \R_+ < \R^n \rtimes \R_+$ is the subgroup of translations and scalings of $T_x \R^n$ which take $T_x L$ to itself.

Consider a collision of $\geq 1$ points with $\infty$, first as just a configuration in $C[I, \R^n] \subset C[I+1, S^n]$, where $I:=i_1+...+i_m$.  Consider a stratum incident to that configuration, labeled by $\{S_1,..,S_k\}$.   Let $S_i$ be a set containing the $(I+1)^\mathrm{th}$ point $\infty$, and let $s_i=|S_i|$.  We can describe the ``screen-space" corresponding to $S_i$ as $C(s_i, T_\infty S^n)/(\R^n \rtimes \R_+) \cong  C(s_i-1,T_\infty S^n \setminus \{0\})/\R_+$, where this identification comes from fixing the $(I+1)^\mathrm{th}$ point $\infty$ at the origin in $T_\infty S^n$. 

Now suppose that the points in the configuration are on the link, so that this configuration is also in $C\left[ \vec{i}; \coprod_{j=1}^m \R \right]$.  We can describe a neighborhood in that space by replacing the screen above by a configuration of points in $T_\infty S^n \setminus \{0\}$ which are constrained to lie in certain open rays emanating from the origin, modulo scaling.  These rays correspond to the directions of the fixed linear embedding.  In other words, we replace $C(s_i-1, T_\infty S^n \setminus \{0\})/\R_+$ by $C(s_i-1,T_\infty L \setminus \{0\})/\R_+$.  (Of course, the one-point compactification of $L$ is not a manifold at $\infty \in S^n$, but $T_\infty L$ seems like appropriate notation for the subset of lines through the origin in $T_\infty S^n$ corresponding to the components of $L$.)  This treatment of collisions at infinity is where we use that our string link components to have different directions towards infinity, as required in Definition \ref{D:MapSpace}.

Note that not all the strata in $C[I, \R^n]$ occur as strata in $C\left[ \vec{i}; \coprod_{j=1}^m \R \right]$ because points in different components of the link cannot collide away from $\infty$ (and furthermore, if a point on a link component has its two neighbors approaching $\infty$, then it must approach $\infty$ too).  But for a given $\mathcal{S} =\{S_1,...,S_k\}$ which does index a stratum $\mathfrak{S}$ of $C\left[ \vec{i}; \coprod_{j=1}^m \R \right]$, any subset of $\mathcal{S}$ clearly indexes a stratum in $C\left[ \vec{i}; \coprod_{j=1}^m \R \right]$.  These subsets correspond precisely to the higher-dimensional strata which intersect a neighborhood of any point in $\mathfrak{S}$.  This is the sense in which the corner structure on $C\left[ \vec{i}; \coprod_{j=1}^m \R \right]$ is inherited from the one on $C[I, \R^n]$.  Hence the space $C\left[ \vec{i}; \coprod_{j=1}^m \R \right]$ can be seen to be a manifold with corners for the same reason that $C[I,\R^n]$ is.  

%For the reader who is meticulous or unfamiliar with similar existing results (e.g., the work in \cite{AS}), 
In more detail
we parametrize a neighborhood of a codimension $k$ point in $C\left[ \vec{i}; \coprod_{j=1}^m \R \right]$, just as is done for $C[i,\R]$ in \cite[Section 4.1]{V:SBT} or for $C[i, M]$ in \cite[Section 5.4]{AS}:  

%rewrite/rework?
%Near configurations where none of the points have collided with $\infty$, the charts are the same as for $C[i,\R]$.  We describe a neighborhood of a configuration in a codimension-1 face where the first $a_j$ points on the $j$-th strand have collided with infinity in the ``negative" direction and the last $b_j$ points on the $j$-th strand have collided with infinity in the ``positive" direction (so $a_j +b_j \geq 1$ for some $j$).  

Let $\mathfrak{S}$ be a stratum of codimension $k$, indexed by a collection $\mathcal{S}$ of subsets $S_1,\ldots,S_k$ of $\{1,2,\ldots,1+\sum_j i_j\}$ (with the last point here corresponding to $\infty$).  
%We will assume that $\infty$ is contained in one of the sets in $\mathfrak{S}$.  
A point $c$ in such a stratum is described by (not necessarily distinct) points $x_1 = x_1(c), \ldots, x_p=x_p(c)\in S^n = \R^n \cup \{\infty\}$, together with $k$ screens $\vec{u}_S$ (one for each $S\in \mathcal{S}$) at some of these $x$'s, with possibly multiple screens at any given $x$.  A screen $\vec{u}_S$ away from $\infty$ consists of $u_{S,1} <  u_{S,2} <... \in \R \cong T_x L$ such that $\sum_h u_{S,h} =0$ and $\sum_h |u_{S,h}|^2=1$.  A screen $\vec{u}_S$ at $\infty$ decsribes the escape to $\infty$ of $a_j+b_j$ points on the $j$-th strand, $a_j$ of them in the ``negative direction" and $b_j$ of them in the ``positive direction".  Such a screen is given by
\begin{equation*}
\label{E:ScreenAtInfty}
( u^1_{S,1}<\ldots<u^1_{S, a_1},\ldots,u^m_{S,1}<\ldots<u^m_{S,a_m}; v^1_{S,1}<\ldots<v^1_{S,b_j}, \ldots, v^m_{S,1}< \ldots <v^m_{S,b_m})
\end{equation*}
where $u^j_{S,h} \in (-\infty, 0)$ and $v^j_{S,h} \in (0,\infty)$ are points in the two rays in $T_\infty L \setminus \{0\}$ coming from the $j$-th component of the link, and where these parameters satisfy

$$\sum_j \left( \sum_{h=1}^{a_j} |u^j_{S,h}|^2 + \sum_{h=1}^{b_j} |v^j_{S,h}|^2\right) = 1$$

Note that either type of screen $\vec{u}_S$ is given by as many parameters as there are elements in $S$.  Using the set of the $x$'s in $\R^n=S^n \setminus \{\infty\}$ (without multiplicity) and the parameters in the $\vec{U}_S$, we can parametrize an open neighborhood $V\subset \mathfrak{S}$ of an interior point $c_0 \in \mathrm{int} (\mathfrak{S})$, showing that  $\mathfrak{S}$ is a manifold (of dimension $np - k$).

Thus to understand the corner structure of $C\left[ \vec{i}; \coprod_{j=1}^m \R \right]$, it suffices to provide a map from an open neighborhood $U\x [0,\epsilon)^k$ of $(c_0, 0)$ in $\mathrm{int}(\mathfrak{S}) \x [0, \infty)^k$ to $C\left[ \vec{i}; \coprod_{j=1}^m \R \right] $.  We first define a map $U\x (0,\epsilon)^k \to (S^n)^{i_1+...+i_m}$ by 
$$
(c, r_1,...,r_k) \mapsto \left(  \mathrm{exp}_{x_1(c)} \left( \sum_{\ell\in \{1,...,k\} : \ S_\ell \ni 1} \widetilde{r}_\ell  \ u_{S_\ell, 1} \right), ..., \mathrm{exp}_{x_p(c)} \left( \sum_{\ell\in \{1,...,k\} : \ S_\ell \ni p} \widetilde{r}_\ell \  u_{S_\ell, p} \right)      \right)
$$
where $\mathrm{exp}_x$ is the exponential map $T_x L \to S^n$ and where $$\widetilde{r}_\ell = \prod_{\ell' :  \ S_{\ell'} \supset S_\ell} r_{\ell'}.$$
Even though $T_\infty L$ is strictly not a tangent space to a manifold, it has an exponential map coming from the restriction of the exponential map from $T_\infty S^n$.
For a sufficiently small neighborhood $U=U(c_0)$ and sufficiently small $\epsilon=\epsilon(c_0)$, one can show that this map is injective.  The map above is essentially \cite[Equation 5.71]{AS}, 
%(although we are thinking of it as giving a parametrization rather than coordinates),
and the proof of injectivity is essentially the same as the proof given in that reference.  Finally, this map extends continuously  to a map $U\x [0, \epsilon)^k \to C\left[ \vec{i}; \coprod_{j=1}^m \R \right] $ by mapping a point $(c,0)$ into $\mathfrak{S}$ and, more generally, by mapping a boundary point $(c, \vec{r})$ into the stratum indexed by $\{S_\ell : r_\ell =0\}$.
\end{proof}

It is now clear that for any $L$ this gives a compactification of $\prod_{j=1}^m C(i_j, \R)$ whose manifold-with-corners structure is independent of $L$.  So we can write 
\begin{equation}\label{E:LinkEvaluation}
ev_\Gamma \colon \Lk^n_m \x
C\left[ \vec{i}; \coprod_{j=1}^m \R \right]
% \prod_{j=1}^m C(i_j, \R)
\longrightarrow C\left[\sum_{j=1}^m i_j, \R^n\right]
\end{equation}

%i think the following should be omitted. the reader can look it up for herself.
%For example, if $x_1=x_2$ in $C[i_k, \R]$, then $v_{12}$ will be sent to $D_{x_1}L(v_{12})$; if $x_1=x_2=x_3$, the ratio $a_{123}$ is sent to the limit of the ratio $|L(x_1)-L(x_2)|/|L(x_1)-L(x_3)|$).\bfn{this needs to be said better. i tried to rewrite some of it, but i don't really understand what's going on with the ratio.} Of course, many different diagrams $\Gamma$ can give rise to the same evaluation map, since it only depends on the number of segment vertices on each segment.

Returning to the ordinary compactified configuration spaces, we have the projection
\begin{equation}\label{E:Projection}
pr\colon C\left[\sum_{j=1}^m i_j + s, \R^n\right]\longrightarrow C\left[\sum_{j=1}^m i_j, \R^n\right]
\end{equation}
given by forgetting the last $s$ points of a configuration, as well as all the $v_{ij}$ and $a_{ijk}$ which involve any of the last $s$ points.

\begin{defin}\label{D:LinksBundle}
Given $\Gamma\in\LD$ with $i_j$ segment vertices on the $j$th segment and $s$ free vertices, let $\vec{\imath}=(i_1,\ldots, i_m)$, and let
$$
C\left[\vec{\imath}+s;\, \Lk_{m}^n, \Gamma\right]
$$
 be the pullback of $pr$ along $ev_\Gamma$:
\begin{equation}\label{E:LinksPullback}
\xymatrix{
C\left[\vec{\imath}+s;\, \Lk_{m}^n, \Gamma\right]\ar[r]  \ar[d]  &
C\left[\sum_{j=1}^m i_j + s, \R^n\right]\ar[d]^{pr}  \\
\Lk_m^n\times C\left[ \vec{i}; \coprod_{j=1}^m \R \right].
\ar[r]^-{ev_\Gamma} &
C\left[\sum_{j=1}^m i_j, \R^n\right]
}
\end{equation}
\end{defin}
%\bfn{i changed the notation a bit. i can't decide on better notation, so it's staying the same for now at least.}

We then have the following special case of Proposition A.3 in \cite{BT}. 

\begin{prop}\label{P:LinkBundle}
With $\Gamma$ as above, the projection
$$
\overline{\pi}_{\Lk,\Gamma}\colon C\left[\vec{\imath}+s;\, \Lk_{m}^n, \Gamma\right]\longrightarrow \Lk_{m}^{n}
$$
 is a smooth fiber bundle whose fiber is a finite-dimensional smooth manifold with corners.
\end{prop}

Returning to the perspective of diffeology (as explained in Section \ref{S:SmoothStructure}), it is convenient to think of this as a compatible collection of bundles, one for each $\psi\colon  M\to\Lk_m^n$, just as the one above but with $\Lk_m^n$ replaced by $M$.  Here $\psi$ is a smooth map and $M$ is a finite-dimensional manifold (without corners).  In that case, it is not hard to generalize the proof of Proposition A.3 of \cite{BT} from one fiber to the whole bundle.  It is also not hard to see that the projection map $\overline{\pi}$ is smooth.  (Note that the corner structure plays no role in the smoothness of $\overline{\pi}$, since $d \overline{\pi}$ sends all the tangent vectors orthogonal to boundary faces to zero.)

We will denote the fiber of $\overline{\pi}_{\Lk,\Gamma}$ over a link $L$ by
$$
\overline{\pi}_{\Lk,\Gamma}^{-1}(L)=C\left[\vec{i}+s;\, L, \Gamma\right].
$$
We think of this space as a configuration space whose first $i_1$ points must lie on the first strand of $L$, second $i_2$ must lie on the second strand, and so on, while the last $s$ are free to move anywhere in $\R^n$ (including on the image of $L$).

%%%%%%%%%%%%%%%%%%%%%%%%%%%%%%%%%%%%%%%%%%%%%%%%%%%%%%%%%%%%%%%%%%%%%%%%%%%%%%%%%%%%%%%%%%%%%%%%%%%%%%%%

\subsubsection{Bundles from diagram vertices and a difficulty with homotopy links}\label{S:ProblemBundlesVertices}

%%%%%%%%%%%%%%%%%%%%%%%%%%%%%%%%%%%%%%%%%%%%%%%%%%%%%%%%%%%%%%%%%%%%%%%%%%%%%%%%%%%%%%%%%%%%%%%%%%%%%%%%

If $\Gamma$ is a diagram in $\HLD$, then the above construction will not in general produce a fiber bundle over $\HLk^n_m$. The first problem is that a generic element $H\in \HLk^n_m$ need not be an embedding or even an immersion, so that the target of the evaluation map is not the usual compactified configuration space, but rather a ``partial'' configuration space where some points are allowed to collide (without regard for how), while others are not. The second problem, not as easily overcome, is that the map from one partial configuration space to another which restricts to some subset of the original set of points is usually not a fibration, making it difficult to produce a fiber bundle by pullback. As an illustration, consider the following example.

\begin{example}\label{Ex:ProblemExample}
Define
$$
C(2,1;\R^n)=\{(x_1,x_2,y)\in (\R^{n})^3\colon x_1,x_2\neq y\}.
$$
and let $C[2,1;\R^n]$ denote its compactification (we only compactify along the diagonals which have been removed). Next, take $m=1$ (so there is one strand) and any value of $n$, and consider the evaluation map $$ev\colon\HLk^n_1\times C[2,\R]\longrightarrow \R^n\times\R^n.
$$ The projection 
$$pr\colon C[2,1;\R^n]\longrightarrow\R^n\times\R^n$$ 
to the first two coordinates is not even a fibration, as the fiber over a point $(x_1,x_2)$ with $x_1=x_2$ is homotopy equivalent to $S^{n-1}$, while the fiber over such a pair with $x_1\neq x_2$ is homotopy equivalent to $S^{n-1}\vee S^{n-1}$. The problem persists with links of more components.
\end{example}

However, if we only allow one point on each strand for the evaluation map, then we can proceed as follows. We have an evaluation map (where $\vec{1}:=(1,1,...,1)$)
$$
ev\colon \HLk_m^n\times C\left[\vec{1}; \coprod_{j=1}^m \R\right] \longrightarrow C[m, \R^n]
$$
%\ifn{So here I don't know if we just want to land in open configuration space $C(m, \R^n)$ or the compactified one.  Since these points are never going to come together, open one is fine.  But then is the projection $C[m+s, \R^n]\to C(m, \R^n)$ a bundle?  Do we even need it to be?  All that we need is for the pullback to be a bundle over $\HLk_m^n$ and in it we don't need to compactify between points on different strands.}\bfn{we have to make a decision. either compactify everything, or compactify minimally. in the case above, there isn't a map if you try to make the image the open configuration space unless you partially compactify. it's probably easier to just compactify everything, but we should make a point that we didn't need to do that, and that it will simplify some things (like those associated faces). i'm ok whatever you decide is easiest.}
obtained by evaluating each strand of a homotopy link on exactly one point in that strand.  The image necessarily lies in the interior of the compactified configuration space $C[m, \R^n]$ since the images of the $m$ strands are disjoint.

We again have a projection map
\begin{equation}\label{E:SingleProjection}
pr\colon C[m+s, \R^n]\longrightarrow C[m, \R^n]
\end{equation}
which is a fibration (of manifolds with corners) so that one can form the pullback
$$
\xymatrix{
C[\vec 1 + s;\, \HLk_{m}^n]\ar[r]  \ar[d]  &
C[m + s, \R^n]\ar[d]^{pr}  \\
\HLk_m^n\times  C\left[\vec{1}; \coprod_{j=1}^m \R\right]  \ar[r]^-{ev} &
C[m, \R^n]
}
$$
There is now a bundle
\begin{equation}\label{E:OnePointBundle}
 C[\vec 1 + s;\, \HLk_{m}^n]\longrightarrow \HLk_{m}^n
 \end{equation}
 for the same reason we have one in \refP{LinkBundle}.  (It should be noted that A.3 of \cite{BT} may appear to the reader not to apply, but it depends on A.5, which does apply in this situation and gives the result we claim.)  We now use this observation to build bundles over $\HLk_m^n$ for any diagram $\Gamma\in\HLD$, and this will naturally extend to diagrams in $\LD$. In order to do so, we need to break our diagrams up into pieces, called ``grafts''.

%Instead, consider an evaluation map which takes values in a product of configuration spaces. That is, we have many different evaluation maps
%$$
%ev\colon \HLk_m^n\times \prod_{j=1}^m C[i_j,\R] \to \prod
%$$

%%%%%%%%%%%%%%%%%%%%%%%%%%%%%%%%%%%%%%%%%%%%%%%%%%%%%%%%%%%%%%%%%%%%%%%%%%%%%%%%%%%%%%%%%%%%%%%%%%%%%%%%

\subsubsection{The graft components of a diagram}\label{S:GraftComponents}

%%%%%%%%%%%%%%%%%%%%%%%%%%%%%%%%%%%%%%%%%%%%%%%%%%%%%%%%%%%%%%%%%%%%%%%%%%%%%%%%%%%%%%%%%%%%%%%%%%%%%%%%

%Given $\Gamma\in\HLD$, suppose it contains $i_j$ segment vertices on the $j$th segment and $s$ free vertices, as in the previous section.

\begin{defin}
For a vertex $v$ in a diagram $\Gamma$, let $N(v)$ be the set of all pairs $(w,e)$ such that $b(e)=\{v,w\}$.
\end{defin}

Thus $N(v)$ consists of all the ``neighbors'' of $v$ counted with multiplicity according to edges.

\begin{defin}
Let $\Gamma=(V,E,b)\in\LD$ be a diagram. Define the \emph{hybrid} of $\Gamma$ to be the diagram $\widetilde{\Gamma}=(\widetilde{V},\widetilde{E},\widetilde{b})$ defined as follows: The set $\widetilde{V}$ is obtained from $V$ by replacing each segment vertex $v\in V$ of $\Gamma$ with the set $v\times N(v)$, the elements of which will represent new vertices, and otherwise the vertex set is unchanged. The edge set $\widetilde{E}$ is equal to $E$. The map $\widetilde{b}$ is induced from $b$ according to the following rule: Suppose $b(e)=\{v,w\}$. If $v,w\in\widetilde{V}$, then $\widetilde{b}(e)=b(e)$. If one of $v$ or $w$, say $v$, is a segment vertex, then $\widetilde{b}(e)=\{(v,(w,e)),w\}$. If both are, then $\widetilde{b}(e)=\{(v,(w,e)),(w,(v,e))\}$.
\end{defin}

The hybrid is not a link diagram, but it does induce certain link diagrams which are subdiagrams of the original link diagram $\Gamma$.

\begin{defin}\label{D:GraftComp}
For a diagram $\Gamma\in \LD$ with hybrid $\widetilde{\Gamma}$, define the \emph{graft components} of $\widetilde{\Gamma}$ to be the set of path components (i.e., connected components) of $\widetilde{\Gamma}$.
%all subdiagrams of $\widetilde{\Gamma}$ determined by its path components. That is, the graft component containing the vertex $v$ consists of all vertices of $\widetilde{\Gamma}$ which can be joined by a path in $\widetilde{\Gamma}$ to $v$, and all edges which can appear in such paths.
\end{defin}

%\begin{rem}
%Graph components can heuristically be thought of as obtained by removing ``tubular neighborhoods" (this of course does not make sense here since our diagrams do not live in a Euclidean space) of segments and segment vertices and then considering the path components of edges obtained that way.
%\end{rem}

\begin{example}\label{Ex:LinkBundle}
Consider the diagram $\Gamma$ in Figure \ref{Fig:HoLinkBundle}.  The five graft components of its hybrid $\widetilde{\Gamma}$ are given in Figure \ref{Fig:HoLinkBundleComponents}.

\begin{figure}[h]
\input{HoLinkBundle.pstex_t}
\caption{}
\label{Fig:HoLinkBundle}
\end{figure}

\begin{figure}[h]
\input{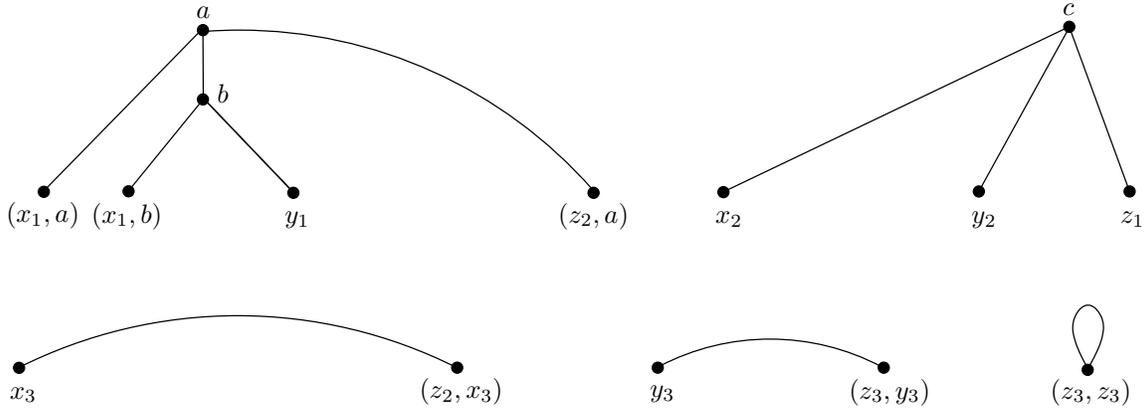}
\caption{The five graft components of the diagram in Figure \ref{Fig:HoLinkBundle}. We have simplified the labels on the vertices of the graft components because the original diagram does not possess multiple edges between a given pair of vertices.}
\label{Fig:HoLinkBundleComponents}
\end{figure}  
\end{example}

The following is clear by construction.

\begin{prop}\label{P:GraftChords}
Each chord of $\Gamma$ gives rise to a graft component consisting of two vertices and a single edge, and each loop at a segment vertex gives rise to a graft component with a single vertex and a single edge.
\end{prop}

Although the hybrid $\widetilde{\Gamma}$ is not a link diagram, each graft component $c({\widetilde{\Gamma}})$ of $\widetilde{\Gamma}$ canonically defines an element of $\LD$, with its structure induced by $\Gamma$. 

\begin{defin}\label{D:graft}
Suppose the diagram $c({\widetilde{\Gamma}})=(V(c({\widetilde{\Gamma}})),E(c({\widetilde{\Gamma}})),b_{c({\widetilde{\Gamma}})})$ is a graft component of $\widetilde{\Gamma}$, so that $V(c({\widetilde{\Gamma}}))\subset \widetilde{V}$ and $E(c({\widetilde{\Gamma}}))\subset \widetilde{E}=E$. The forgetful map $\widetilde{V}\to V$ identifies $c({\widetilde{\Gamma}})$ with a subdiagram $c(\Gamma)$ of $\Gamma$, called a \emph{graft} of $\Gamma$ which inherits all the necessary structure for it to define an element of $\LD$. 
\end{defin}

If $\Gamma\in\HLD$, then it is clear that all the grafts of $\Gamma$ are also elements of $\HLD$. The set of all graft components, and hence the set of all grafts, can be ordered according to the ordering of the vertices of $\Gamma$; no two grafts will have the same underlying vertex sets because diagrams with multiple edges between a pair of vertices are set to zero.

If $\Gamma\in\HLD$, the grafts of $\Gamma$ have an additional useful property which will allow us to build bundles over $\HLk$.

\begin{prop}\label{P:GraftComponentVertices}
For $\Gamma\in\HLD$, each graft of $\Gamma$ has at most one segment vertex on each segment.
\end{prop}

\begin{proof}
First we claim that for any pair of distinct free vertices $v,v'$ in the same graft component $c({\widetilde{\Gamma}})$, there exists a path of free edges between them. This is clear since each vertex of $\widetilde{\Gamma}$ which arises from a segment vertex of $\Gamma$ is joined to precisely one other vertex in that component, so any path between $v$ and $v'$ in $c({\widetilde{\Gamma}})$ can be shortened to avoid such vertices. This clearly descends to a path in $c_\Gamma$ between $v$ and $v'$ consisting only of free edges.

Now suppose, on the contrary, that there is some graft component $c({\widetilde{\Gamma}})$ of $\widetilde{\Gamma}$ such that the associated graft $c(\Gamma)$ of $\Gamma$ has two distinct segment vertices $x$ and $x'$ on a given segment.

Let $\alpha=\{e_i\}_{i=1}^k$ be any path of edges from $x$ to $x'$ in $c(\Gamma)$. Let $1\leq j\leq k$ be such that $b(e_j)$ contains a segment vertex $y$ on a segment different than the segment on which $x,x'$ lie. Such a $j$ must exist by definition of $\HLD$. If $b(e_j)=\{y,v\}$ and $b(e_{j+1})=\{y,v'\}$, then $v=v'$ implies $y$ could be avoided by removing $e_j,e_{j+1}$ from our path. Hence $v\neq v'$, and both are free vertices by \refP{GraftChords}. But our observation at the beginning of the proof shows there must exist a path between $v$ and $v'$ which avoids $y$. We can similarly eliminate any other segment vertex encountered along the way,  producing a path between $x$ and $x'$ which does not pass through any other segment vertices.
\end{proof}

%The set of grafts $c(\Gamma)$ associated with the graft components can be ordered according to the ordering of the vertex set of $\Gamma$. Each graft component of a diagram $\Gamma\in\HLD$ has an associated evaluation map whose codomain is a compactified configuration space.\bfn{redundant?}

%%%%%%%%%%%%%%%%%%%%%%%%%%%%%%%%%%%%%%%%%%%%%%%%%%%%%%%%%%%%%%%%%%%%%%%%%%%%%%%%%%%%%%%%%%%%%%%%%%%%%%%%

\subsubsection{Bundles of compactified configuration spaces from vertices and edges of a diagram}\label{S:BundlesVerticesEdges}

%%%%%%%%%%%%%%%%%%%%%%%%%%%%%%%%%%%%%%%%%%%%%%%%%%%%%%%%%%%%%%%%%%%%%%%%%%%%%%%%%%%%%%%%%%%%%%%%%%%%%%%%

We now describe the construction of bundles over $\Lk_m^n$ and $\HLk_m^n$ using the grafts of a diagram.

\begin{prop}\label{P:GraftComponentCodomain}
Let $\Gamma\in\LD$ be a diagram with $i_j$ segment vertices on the $j$th segment, and let $c({\Gamma})$ be a graft of $\Gamma$ with $d_j$ segment vertices on the $j$th segment for all $j=1$ to $m$.  Let $\vec{i}=(i_1,...,i_m)$.  Then $c(\Gamma)$ gives rise to an evaluation map
\begin{equation*}
ev_{c(\Gamma)}:\Lk_m^n\times  C\left[\vec{i}; \coprod_{j=1}^m \R\right]
\longrightarrow C\left[\sum_jd_j,\R^n\right].
\end{equation*}
If $c_1({\Gamma}),\ldots, c_k({\Gamma})$ are the grafts of $\Gamma$ ordered as described above, and $c_l(\Gamma)$ has $d_{l,j}$ segment vertices on the $j$th segment for $l=1$ to $k$, then we have an evaluation map
\begin{equation}\label{E:GraftEval}
ev_{gr(\Gamma)}:\Lk_m^n\times  C\left[\vec{i}; \coprod_{j=1}^m \R\right]
\longrightarrow \prod_{l=1}^k C\left[\sum_jd_{l,j},\R^n\right],
\end{equation}
where $ev_{gr(\Gamma)}=(ev_{c_1(\Gamma)},\ldots, ev_{c_k(\Gamma)})$. Moreover, if $\Gamma\in\HLD$, then we have an evaluation map
\begin{equation*}
ev_{gr(\Gamma)}:\HLk_m^n\times  C\left[\vec{i}; \coprod_{j=1}^m \R\right]
\longrightarrow \prod_{l=1}^k C\left[\sum_j d_{l,j},\R^n\right]
\end{equation*}
whose restriction to $\Lk_m^n\times C\left[\vec{i}; \coprod_{j=1}^m \R\right]$ is equal to the map in \refE{GraftEval}, and whose image in each factor lies in the open configuration space $C\left(\sum_j d_{l,j},\R^n\right)$.
\end{prop}

\begin{proof}
This follows immediately from \refP{GraftComponentVertices}, since there is at most one segment vertex on each segment of a graft $c(\Gamma)$, and since homotopy links send points in distinct segments to distinct points, so that the codomain of the evaluation map is correctly identified. 
\end{proof}

If $\Gamma\in\LD$, we now have a different evaluation maps associated with a link diagram, and this gives rise to a new way to build a bundle associated with a diagram.

\begin{defin}\label{D:NewPullbacks}
Let $\Gamma\in\LD$ be a link diagram with grafts $c_1(\Gamma),\ldots, c_k(\Gamma)$ such that $c_l(\Gamma)$ has $d_{l,j}$ segment vertices on the $j$th segment and $s_l$ free vertices for $l=1$ to $k$. Let $\vec{d}_l=(d_{l,1},\ldots, d_{l,m})$. Define
$$
\LPB
$$
%\ifn{I changed this notation a bit, but I defined some commands for these bundles so that we can change it very easily if we want.  Some other options I had in mind were
%$$
%\prod_{l, \Delta}C[\vec{d}_l+s_l;\, \Lk_{m}^n, c_l(\Gamma)], \ \ \ \Delta C[\vec{d}_l+s_l;\, \Lk_{m}^n, c_l(\Gamma)], \ \ \ \operatorname{PB}_lC[\vec{d}_l+s_l;\, \Lk_{m}^n, c_l(\Gamma)].
%$$
%I've defined commands for these as well (those are commented out).  I also defined the corresponding commands for homotopy links bundles, as well as for all the fibers over $L$ or $H$. There are just a few places in section on finite type invariants where we use $H+$ and $H-$ instead of $H$ so we'll have to do these by hand.}
as the pullback of $pr$ along $ev_{gr(\Gamma)}$:
\begin{equation}\label{E:LinksGraftPullback}
\xymatrix{
\LPB  \ar[r]\ar[d]&
\prod_{l=1}^k C\left[ \sum_jd_{l,j}+s_l, \R^n\right]\ar[d]^{pr}  \\
\Lk_m^n\times  C\left[\vec{i}; \coprod_{j=1}^m \R\right]  
\ar[r]^-{ev_{gr(\Gamma)}} &
\prod_{l=1}^k C\left[ \sum_jd_{l,j}, \R^n\right].
}
\end{equation}
Similarly we define $\HPB$ when $\Gamma\in\HLD$ and $\HLk_m^n$ replaces $\Lk_m^n$.
\end{defin}

\begin{rem} The notation here is meant to observe that given a collection of spaces and maps $X\rightarrow Y_i\leftarrow Z_i$ such that $P_i$ is the pullback of this diagram for each index $i$, then the pullback of the evident diagram $X\rightarrow\prod_iY_i\leftarrow \prod_iZ_i$ is the pullback of $\prod_i P_i$ along the diagonal map $\Delta\colon X\to\prod_i X$.
\end{rem}

\begin{prop}\label{P:NewBundles}
Let $\Gamma\in\LD$ be a link diagram with grafts $c_1(\Gamma),\ldots, c_k(\Gamma)$ such that $c_l(\Gamma)$ has $d_{l,j}$ segment vertices on the $j$th segment for $l=1$ to $k$, $j=1$ to $m$. Then the projection
$$
\pi_{\Lk,\Gamma}\colon \LPB\longrightarrow\Lk^n_m
$$
is a smooth fiber bundle whose fibers are smooth finite-dimensional manifolds with corners. Moreover, if $\Gamma\in\HLD$, then the projection
$$
\pi_{\HLk,\Gamma}\colon \HPB\longrightarrow\HLk^n_m
$$
is also a smooth fiber bundle whose fibers are smooth finite-dimensional manifolds with corners, and 
\begin{equation}\label{E:HoLinksGraftPullback}
\xymatrix{
\LPB\ar[r]  \ar[d]_{\pi_{\Lk,\Gamma}}  & \HPB \ar[d]^{\pi_{\HLk,\Gamma}}  \\
\Lk_m^n\ar[r] & \HLk_m^n
}
\end{equation}
is a pullback square.
\end{prop}

\begin{proof}
The projection $\pi_{\Lk, \Gamma}$ is a smooth bundle for the same reasons that $\overline{\pi}_{\Lk, \Gamma}$ in \refP{LinkBundle} is.  For $\pi_{\HLk, \Gamma}$, this is just an extension of the observation made in \eqref{E:OnePointBundle}.  (As mentioned for the bundle $\overline{\pi}_{\Lk, \Gamma}$, it will sometimes be convenient to think of the bundle $\pi_{\Lk, \Gamma}$ (resp. $\pi_{\HLk, \Gamma}$) as a compatible collection of bundles, one for each finite-dimesnional manifold mapped into $\Lk_m^n$ (resp. $\HLk_m^n$).)  Lastly, the fact that the square \eqref{E:HoLinksGraftPullback} is a pullback follows directly from the definitions.
\end{proof}

We will denote the fibers of $\pi_{\Lk,\Gamma}$ and $\pi_{\HLk,\Gamma}$ over a link $L\in\Lk_m^n$ or a homotopy link $H\in\HLk_m^n$, respectively, by
$$
\pi_{\Lk,\Gamma}^{-1}(L)=\FLPB
$$
and
$$
\pi_{\HLk,\Gamma}^{-1}(H)=\FHPB.
$$

%%%%%%

\begin{example}\label{Ex:LinkBundle2}
Consider the two different evaluation maps, one from \refE{LinkEvaluation} and the other from \refE{GraftEval}, for the diagram $\Gamma$ from Figure \ref{Fig:TwoEvaluations}. For conciseness, we have omitted the compactification coordinates.

\begin{figure}[h]
\input{TwoEvaluations.pstex_t}
\caption{}
\label{Fig:TwoEvaluations}
\end{figure}

On the one hand, using \refE{LinkEvaluation}, we have
$$
ev_\Gamma\colon \Lk_3^n\times  C[1,2,1; \R \sqcup \R \sqcup \R] \longrightarrow C[4,\R^n]
$$
whose restriction to the interior is given by
$$
(L,x,y_1,y_2,z)\longmapsto \left(L(x),L(y_1),L(y_2),L(z)\right).
$$
The image of this restriction lies in 
%a compact subset of 
the subspace of all $(w_1,w_2,w_3,w_4)$ where $w_1\neq w_2,w_3,w_4$, and $w_2,w_3\neq w_4$ of $(\R^n)^4$. We also have the projection map
$$
pr\colon C[5,\R^n]\longrightarrow C[4,\R^n]
$$
which on the interior sends $(w_1,w_2,w_3,w_4,w_5)$ to $(w_1,w_2,w_3,w_4)$, so that the fibers of the bundle $\overline{\pi}_{\Lk,\Gamma}:C[(1,2,1)+1;\Lk_3^n]\to\Lk_3^n$ are a subspace of $C[5,\R^n]$. The five configuration points correspond with the vertices of $\Gamma$, and we blow up all diagonals of $(\R^n)^5$. Note that the bundle obtained is exactly the same for any diagram with the same vertices as $\Gamma$.

On the other hand, $\Gamma$ has two graft components, one of which is the diagram with a single chord from $x$ to $y_1$, and the other of which is the ``tripod'' with free vertex $a$ and edges between it and $x,y_2$, and $z$. Then \refE{GraftEval} gives another evaluation map
$$
ev_{gr(\Gamma)}: \Lk_3^n\times  C[1,2,1; \R \sqcup \R \sqcup \R]  \longrightarrow C[2,\R^n]\times C[3,\R^n]
$$
given on the interior by
$$
(L,x,y_1,y_2,z)\longmapsto \left(L(x),L(y_1),L(x), L(y_2),L(z)\right)
$$
whose image in each factor lies in the open configuration space. To build the bundle, we use the product of two projection maps
$$
C[2,\R^n]\times C[4,\R^n]\longrightarrow C[2,\R^n]\times C[3,\R^n]
$$
given by
$$
(u_1,u_2,w_1,w_2,w_3,w_4)\longmapsto(u_1,u_2,w_1,w_2,w_3)
$$
to form a bundle 
$$\pi_{\Lk,\Gamma}\colon C[(1,1,0);\Lk_3^n]\oplus C[(1,1,1)+1;\Lk_3^n]\longrightarrow\Lk_3^n.$$ 
The fibers of this bundle are isomorphic to a subspace of $(\R^n)^5$, namely the subspace of all tuples $(w_1,w_2,w_3,w_4,w_5)=(L(x),L(y_1),L(y_2),L(z),a)$, but $w_3=w_5$ is now allowed and we do not blow up this diagonal. This is because there is no mixed edge between the free vertex $a$ and the segment vertex $y_1$. We also do not blow up the locus $w_2=w_3$. Thus the fibers are a subspace of a (compactified) partial configuration space, because not all diagonals have been removed from $(\R^n)^5$.

\end{example}

In general, the difference between the pullback bundle based on vertices only and the one based on vertices and edges is precisely what we saw in the last example. In the latter, the configuration space is not compactified along all the diagonals  but only along those that belong to the same graft component.  Thus if there is no edge between two vertices and they belong to different graft components, the corresponding configuration points can pass through each other without the direction of collision being recorded.

\subsection{Pullback of differential forms to new bundles of configuration spaces}\label{S:Forms}

%%%%%%%%%%%%%%%%%%%%%%%%%%%%%%%%%%%%%%%%%%%%%%%%%%%%%%%%%%%%%%%%%%%%%%%%%%%%%%%%%%%%%%%%%%%%%%%%%%%%%%%%

For the sake of concreteness, it is necessary to choose coordinates on our configuration spaces so that we may explicitly define the pullback of forms. As the interior of configuration space is a subspace of a product of Euclidean spaces, it will suffice instead to consider coordinate systems on such spaces. 

Given a finite ordered set $S$, we have a unique order-preserving isomorphism 
$$\mathrm{pos}\colon S\longrightarrow\{1,\ldots, |S|\}.$$ 
For a coordinate system $(x_1,\ldots, x_{|S|})$ on $(\R^n)^{|S|}$, this gives a natural way to associate $s\in S$ with the coordinate $x_{\mathrm{pos}(s)}$. 

Suppose we have a category $\mathcal{C}$ whose objects are subsets of a fixed finite ordered set $S$ and whose morphisms are inclusions.  The association $T \mapsto (\R^n)^{|T|}$ is a contravariant functor from $\mathcal{C}$ to spaces, since an inclusion $T \to T'$ gives rise to the projection $p_i:(\R^n)^{|T'|}\to(\R^n)^{|T|}$ which forgets the coordinates associated with $T'-T$.   

Now suppose we have a family of subsets $T_1,\ldots, T_k$ of $S$ whose union is equal to $S$.  We will let $\mathcal{C}$ be the category as above whose objects are $S$ and all possible intersections of the $T_i$.

Consider the category of subsets of $\{1, \ldots, k\}$ with inclusions as morphisms.
For each $R \subset \{1,\ldots, k\}$ we have the set $T_{R}:=\cap_{i\in R}T_i$ (where we define $T_\emptyset :=S$), and for each inclusion $R\to R'$ an inclusion $T_{R'}\to T_R$.  Hence $R\mapsto T_R$ is a contravariant functor to $\mathcal{C}$, which can be thought of as a $k$-dimensional cube.  
Following this by the functor from $\mathcal{C}$ to spaces defined above gives a covariant functor $R\mapsto (\R^n)^{|T_R|}$.  
%The association $R\mapsto (\R^n)^{|T_R|}$ is also a functorial, as an inclusion $R\to R'$ gives rise to a projection map $(\R^n)^{|T_R|}\to(\R^n)^{|T_R'|}$. 
Since $S$ is the union of all the $T_i$, we have that $\lim_{R\neq\emptyset}\,(\R^n)^{|T_R|}\cong(\R^n)^{|S|}$. The particular isomorphism we have in mind is the one which makes the following diagram commute:
$$\xymatrix{
\lim_{R\neq\emptyset}\,(\R^n)^{|T_R|} \ar[r]\ar[dr] & (\R^n)^{|S|}\ar[d]^{(p_1,\ldots, p_k)}\\
& \prod_{i=1}^k (\R^n)^{|T_i|} 
}
$$
The diagonal arrow is the natural inclusion of the limit into the product, and the top arrow is the isomorphism we spoke of above, and we use it to give coordinates on the limit. Given a diagram $\Gamma\in\LD$, the situation described above arises with $S=V(\Gamma)$ and $T_i$ as the set of vertices of the $i$th graft (recall that the set of grafts is naturally ordered).

%Given $\Gamma\in\LD$ with $s$ free vertices and $i_j$ vertices on the $j$th segment for $j=1$ to $m$, we have associated with this a bundle of configuration spaces $\LPB$ based on the grafts $c_l(\Gamma)$. For the sake of concreteness, it will be useful to choose coordinates for our configuration spaces which correspond to the ordering of the vertices of $\Gamma$. The order-preserving isomorphism $\mathrm{pos}:V(\Gamma)\to\{1,2,\ldots, |V(\Gamma)|\}$ gives us a way to choose coordinates $(x_1,\ldots, x_{|V(\Gamma)|})$ on $(\R^n)^{|V(\Gamma)|}$, and hence for the subspace of configurations, so each vertex corresponds with a coordinate $x_i$. This gives rise to a preferred coordinate system on a product of configuration spaces, which arise when considering the grafts of $\Gamma$.
%
%However, in the formation of the pullback bundle $\pi_{L,\Gamma}: \LPB\to \Lk_m^n$, the fibers $\FLPB$ have some of the coordinates identified
%
% We choose coordinates $$will represent points in $\prod_{l=1}^k C\left( \sum_jd_{l,j}+s_l, \R^n\right)$ as tuples $\vec x=(x_1,\ldots, x_{|V(\Gamma)|})$, where $x_i\in \R^n$.\ifn{This is the part that's not quite right now.} 
% This yields a preferred coordinate system on the product of their compactifications $\prod_{l=1}^k C\left[ \sum_jd_{l,j}+s_l, \R^n\right]$. We use this to define maps from our configuration spaces to spheres according to edges present in the diagram.  Recall that our edges are all oriented if $n$ is odd and the edge set is ordered if $n$ is even.

\begin{defin}
Let $\Gamma\in \LD$ be a diagram with $i_j$ segment vertices on the $j$th segment and $s$ free vertices. Let $e\in E(\Gamma)$, and suppose $b(e)=\{v,w\}$.

\begin{itemize}
\item If $v\neq w$, then if $e$ is oriented from $v$ to $w$ (or if it is not oriented, then if $v<w$ in the ordering of the vertex set), define 
\begin{align*}
\phi'_e \colon \prod_{l=1}^k C\left[ \sum_jd_{l,j}+s_l, \R^n\right] & \longrightarrow S^{n-1} 
\end{align*}
as the map given on the interior by 
\begin{align*}
  & & & & & &  \vec{x} & \longmapsto\frac{x_{\mathrm{pos}(w)}-x_{\mathrm{pos}(v)}}{|x_{\mathrm{pos}(w)}-x_{\mathrm{pos}(v)}|},
\end{align*}
and define 
$$
\phi_e \colon \LPB\longrightarrow S^{n-1}
$$ 
to be the pullback of $\phi'_e$ along the map $\LPB\to \prod_{l=1}^k C\left[ \sum_jd_{l,j}+s_l, \R^n\right]$.
\item If $v=w$, then necessarily $e$ joins a segment vertex with itself, and if it is oriented by the injection which sends $b(e)=\{v\}$ to $1$ (or is not oriented at all),
$$
\phi_e(\vec{x},L)=D_zL(u)/|D_zL(u)|
$$
where $z$ is the point in one of the strands such that $L(z)=x_{\mathrm{pos}(v)}$ and $u$ is the positive unit tangent vector to the strand at $z$. If $e$ is oriented by the injection sendin $b(e)=\{v\}$ to $-1$, then
$$
\phi_e(\vec{x},L)=-D_zL(u)/|D_zL(u)|.
$$
with $z,u$ as above.
\end{itemize}
\end{defin}

Note that $D_zL(u)\neq 0$ since $L$ is an embedding; in the case of homotopy string links, which may not be embeddings, we do not have to worry about whether this is well-defined because loops cannot be present in diagrams in $\HLD$.

\begin{defin}\label{D:NewMapToSpheres}
Given $\Gamma\in\LD$ as above, define
$$
\phi_\Gamma\colon \LPB\longrightarrow S^{(n-1)|E(\Gamma)|}
$$
by
$$
\phi_\Gamma=\left(\phi_{e_1},\ldots, \phi_{e_{|E(\Gamma)|}}\right),
$$
where $\mathrm{pos}(e_i)=i$ if the edge set is ordered, and otherwise order them according to the dictionary ordering on $\{b(e_i)\}$ (which can be imposed since diagrams with more than one edge joining a pair of vertices are set to zero).
%\bfn{once again this only works if you ignore the possibility of diagrams with multiple edges between a pair of vertices, as those are zero anyways. this is a bit awkward, but we need some convention.}
\end{defin}

Let $\mathrm{sym}_{S^{n-1}}$ be a smooth, unit volume top form on $S^{n-1}$ which is symmetric (meaning its values on antipodal points are equal, though in Section \ref{S:AnomalousFix}, when we discuss the case of links in dimension 3, we will also require this form to be the unique rotation-invariant unit volume form) and let
$$
\omega=\bigwedge_{|E(\Gamma)|}\mathrm{sym}_{S^{n-1}}
$$
Finally define the pullback form
$$
\alpha_{\Gamma}=(\phi_{\Gamma})^*\omega
%=(\phi_{\Gamma}^*\times \theta_{\Gamma}^*)\omega
\in \Omega^{(n-1)(|E(\Gamma)|)}\left(\LPB\right).
$$

Notice that nothing changes in the case of homotopy links.  For a diagram $\Gamma\in\HLD$, we again use edges (but there are no longer any loops) to pull back a product of forms $\omega$ from $S^{(n-1)|E(\Gamma)|}$ to the space $\HPB$, although we will write $\alpha_\Gamma^\HLk$  for the pullback form when $\Gamma\in\HLD$.

Observe also that the same definitions are valid for the bundle $C[\vec i+s; \Lk_m^n, \Gamma]$ considered in earlier literature on the subject. Namely, we have a map
$$
\overline{\phi}_{\Gamma}\colon C\left[\vec{\imath}+s;\Lk_m^n,\Gamma\right]\longrightarrow S^{(n-1)|E(\Gamma)|}
$$
dictated by the edges of $\Gamma$, and this can be used for pulling back a product of volume forms to give a form $\overline{\alpha}_{\Gamma}=(\overline{\phi}_{\Gamma})^*\omega$.  This case was considered in \cite[Section 3.2]{V:B-TLinks}.

%%%%%%%%%%%%%%%%%%%%%%%%%%%%%%%%%%%%%%%%%%%%%%%%%%%%%%%%%%%%%%%%%%%%%%%%%%%%%%x %%%%%%%%%%%%%%%%%%%%%%%%%%

\subsection{Configuration space integrals of string links and homotopy string links}\label{S:Integrals}

We are finally ready to produce forms on spaces of links and homotopy links.  Namely,
the form $\alpha_\Gamma$ can be pushed forward, or integrated along the fiber of the bundle
$$
\pi_{\Lk, \Gamma}\colon \LPB\longrightarrow \Lk_m^n
$$
to produce a form $(\pi_{\Lk, \Gamma})_*\alpha_{\Gamma}$, or, as we will usually denote it, a form
$$
(I_{\Lk})_{\Gamma}\in \Omega^{|\Gamma|}
%{(n-1)|E(\Gamma)|-n|V(\Gamma)_{free}|-|V(\Gamma)_{seg}|}
(\Lk_m^n).
$$

The value of this form on a link $L\in\Lk_m^n$ is thus
$$
(I_{\Lk})_{\Gamma}(L)=\int\limits_{\pi_{\Lk, \Gamma}^{-1}(L)=\FLPB} \alpha_{\Gamma}.
$$
The degree $|\Gamma| := (n-1)|E(\Gamma)|-n|V(\Gamma)_{free}|-|V(\Gamma)_{seg}|$ of $(I_{\Lk})_{\Gamma}$ is the difference of the degree of $\alpha_{\Gamma}$ and the dimension of the fiber $\pi_{\Lk, \Gamma}^{-1}(L)$.  Recall that this quantity is also equal to 
$$k(n-3)+d,$$ where $d=\deg(\Gamma)$ and $k=\ord(\Gamma)$, so that we have constructed a map
\begin{equation}\label{E:DegreesLinks}
I_{\Lk}\colon \LD^d_k \longrightarrow \Omega^{k(n-3)+d}(\Lk_m^n).
\end{equation}
For a diagram $\Gamma\in\HLD$, we integrate the associated form $\alpha^{\HLk}_{\Gamma}$ along the bundle
$$
\pi_{\HLk, \Gamma}\colon \HPB \longrightarrow \HLk_m^n.
$$
This gives a form 
$$
(I_{\HLk})_{\Gamma}\in \Omega^{|\Gamma|}(\HLk_m^n)
$$
%\ifn{Does it even make sense to have differential forms on this space.  It this a nice manifold in any way?}\bfn{i don't know what you mean. the space of homotopy links is some kind of infinite-dimensional manifold, just like the space of ordinary links. is that what you mean?}
whose value on a homotopy link $H\in\HLk_m^n$ is 
$$
(I_{\HLk})_{\Gamma}(H)=\int\limits_{\pi_{\HLk, \Gamma}^{-1}(H)=\FHPB} \alpha_{\Gamma}^\HLk.
$$
Again rewriting the degree of the form, we thus have a map
\begin{equation}\label{E:DegreesHoLinks}
I_{\HLk}\colon \HLD^d_k \longrightarrow \Omega^{k(n-3)+d}(\HLk_m^n).
\end{equation}

%this is already clearly explained in previous sections.
%\begin{rem}
%It should also finally be clear why we defined $\HLD$ the way we did.  Any time a free vertex is associated to more than one segment vertex on the same segment, we are required to take a projection map to a configuration space with more than one point, and this is not possible to do with $\HLk_m^n$, as in the discussion surrounding \eqref{E:BadPullback}. 
%\end{rem}

\begin{rem}
Thinking of the bundle as a collection of compatible bundles (as mentioned around \refP{LinkBundle}) makes clear that this construction produces differential forms on $\Lk_m^n$ and $\HLk_m^n$ in the sense described at the end of Section \ref{S:SmoothStructure}.  In fact, if we replaced the link space by any finite-dimensional manifold $M$ (or just an open subset of Euclidean space) parametrizing a family of links, then fiberwise integration certainly produces a differential form on $M$.  It is also clear that for another manifold $M'$ mapped into $\Lk_m^n$ through the map $\psi: M' \to M$, the form on $M'$ is the pullback via $\psi$ of the form on $M$.
\end{rem}

\begin{rem}\label{R:Compatibility} 
It is immediate from the definition that maps $I_{\Lk}$ and $I_{\HLk}$ are also compatible with the inclusion
$$
\Lk^n_m\hookrightarrow\HLk^n_m,
$$
that is, we have a commutative diagram
$$
\xymatrix{
\HLD   \ar@{^{(}->}[r] \ar[d]_{I_{\HLk}} &  \LD \ar[d]^{I_{\Lk}} \\
\Omega^*(\HLk_m^n) \ar[r] &  \Omega^*(\Lk_m^n)
}
$$
This is precisely what we were after when we refined the definition of the bundles we integrate over.
\end{rem}

Now note that again nothing changes for the case of the pullback bundle defined without consideration of the grafts.  Namely, the construction of $(\pi_{\Lk, \Gamma})_*\alpha_{\Gamma}$ goes through exactly the same way to give a form $(\overline{\pi}_{\Lk, \Gamma})_*\overline\alpha_{\Gamma}$ by pushing forward the form $\overline\alpha_{\Gamma}$ along the map 
$$
\overline{\pi}_{\Lk,\Gamma}\colon C\left[\vec{\imath}+s;\, \Lk_{m}^n, \Gamma\right]\longrightarrow \Lk_{m}^{n}
$$
from \refP{LinkBundle}. We now want to show that the forms we obtain by integrating along this bundle are the same as the forms we obtain by integrating along
$$
\pi_{\Lk,\Gamma}\colon \LPB\longrightarrow \Lk_m^n
$$
are the same as in the case of integration along the bundle
$$
\overline{\pi}_{\Lk, \Gamma}\colon C[\vec i + s;\, \Lk_{m}^n, \Gamma]\longrightarrow \Lk_m^n.
$$
This will finally show that our way of setting up configuration space integrals for links is indeed a refinement of the way that has been considered in literature thus far.

\begin{prop}\label{P:SameForms}  For any $\Gamma\in\LD$,
 $(\pi_{\Lk, \Gamma})_*\alpha_{\Gamma}=(\overline{\pi}_{\Lk, \Gamma})_*\overline\alpha_{\Gamma}.$
\end{prop}

\begin{proof}
The map between fibers is the inclusion of an open dense set. The two fibers are the same on the biggest stratum, namely the open configuration space.  They differ in that $\overline{\pi}_{\Lk,\Gamma}^{-1}(L)$ has more diagonals of $\R^{n|V(\Gamma)|}$ removed and compactified.  Thus the difference between the two is at least of codimension $1$ and so the integrals are equal.
\end{proof}

We next give a few examples of these configuration space integrals.

\begin{example}[Diagrams with no free vertices]\label{Ex:NoFree}
One special case is that of diagrams with no free vertices, i.e.~those that only contain chords and loops.  In that case, the construction simplifies since there are no pullback constructions as in \refD{NewPullbacks}, and the bundles constructed are trivial.  For example, if $\Gamma\in\LD$ is the diagram from Figure \ref{Fig:NoFree} (where we have omitted the edge orientations and labels for simplicity), 
then the map $\phi_\Gamma$ is a composition
$$
\phi_{\Gamma}  \colon \Lk_3^n\times C[3,1,2; \R \sqcup \R \sqcup \R]
 \stackrel{ev_{\Gamma}}{\longrightarrow}
 C[2, \R^n]^4\times  C[1, \R^n]
 \longrightarrow 
(S^{(n-1)})^5
$$

%given by 
%$$
%\xymatrix{
%(L, x_1, x_2, x_3, y, z_1, z_2, \frac{|x_1-x_2|}{|x_1-x_3|})  \ar@{|->}[d] \\
% \left(\frac{L(x_1)-L(x_3)}{|L(x_1)-L(x_3)|}, \frac{L(x_1)-L(z_2)}{|L(x_1)-L(z_2)|}, \frac{L(x_2)-L(z_1)}{|L(x_2)-L(z_1)|}, \frac{L(y)-L(z_1)}{|L(y)-L(z_1)|}, \frac{L'(z_2)}{|L'(z_2)|}  \right)
%}
%$$
%\ifn{The last map should be replaced by new derivative notation.}
%on the interior of the configuration spaces.  If $x_i=x_j$ in $C[3,\R]$ or $z_1=z_2$ in $C[2,\R]$, the map is extended by continuity as described in the discussion following \eqref{E:LinkEvaluation}.
%(Also, the order of the vectors and their direction might be different depending on what the labels and directions on the edges are.)

\begin{figure}[h]
\input{NoFree.pstex_t}
\caption{}
\label{Fig:NoFree}
\end{figure}

After pulling back the product of five (antipodally) symmetric top forms from $(S^{(n-1)})^5$, the integration takes place along the trivial bundle 
$$
\pi_{\Lk, \Gamma}\colon \Lk_3^n\times C[3,1,2; \R \sqcup \R \sqcup \R]
\longrightarrow
\Lk_3^n.
$$
\end{example}

%\begin{rem}\label{R:NewWayFactorsThrOldWay}
%Notice that the map $\phi_{\Gamma}$ factors through $C[6,\R^n]$, i.e.~there is a commutative diagram
% $$
% \xymatrix{
% \Lk_3^n\times C[3, \R]\times C[1, \R]\times C[2, \R] \ar[rr]^-{ev_{gr}}\ar[dr]^-{ev}  
% &   & C[2, \R^n]^4\times  C[1, \R^n]\ar[r] & S^{5(n-1)} \\
%& C[6,\R^n] \ar[ur]^-{\tau} &  & 
%}
%$$
%This is because the images of each $x_i$, $y$, and each $z_i$ are distinct.  The lower evaluation map is one from \eqref{E:LinkEvaluation} used in the ``old way" of defining integration and it leads to the definition of the pullback bundle from \refD{LinksBundle}.  The map $\tau$ is a special case of a map that appeared in the proof of \refP{SameForms}.\ifn{This thing about $\tau$ being in the proof is from the old incorrect proof.  This map may not be in the new proof.  If it's not, the label $\tau$ should be deleted.}\bfn{i think we can should simplify the last example (while keeping the diagram for concreteness) and eliminate this remark as well.}
%\end{rem}

\begin{example}[Linking number]\label{Ex:LinkingNumber}
Another special case, and in fact the case that motivated Bott and Taubes to define configuration space integrals for knots in \cite{BT}, is that of the linking number of a two-component link in $\R^3$.  Namely, suppose $\Gamma$ is the diagram with a single chord between segments $i$ and $j$ and no free vertices or segment vertices on other segments, as in Figure \ref{Fig:LinkingDiagram}.

\begin{figure}[h]
\input{LinkingDiagram.pstex_t}
\caption{}
\label{Fig:LinkingDiagram}
\end{figure}

Then the integration described above recovers the classical Gauss integral computing the linking number of strands $i$ and $j$ of a link or a homotopy link $L$, which we will denote by $\operatorname{lk}(L_i,L_j)$.  In short, 
$$
\operatorname{lk}(L_i,L_j)=(I_{\HLk})_{\Gamma}(L)=(I_{\Lk})_{\Gamma}(L)=\int\limits_{C[1,1; \R\sqcup \R]} \left(\frac{L(x)-L(y)}{|L(x)-L(y)|}\right)^* \mathrm{sym}_{S^2}
$$
where the compactification $C[1,1;\R\sqcup\R]$ is an octagonal disk (see \cite[Section 1.2]{Koyt:HoMilnor3ple} for details).

To see how shuffle products of integrals give products of linking numbers, see Example \ref{Ex:LinkingNumberProduct}.

\end{example}

\begin{example}[Homotopy links with one strand]\label{Ex:HoLinkOneStrand}
Consider the case of $\HLk_1^n$, $n\geq 3$.  Now the only diagram in $\HLD$ is the empty diagram,
%Now the only diagram in $\HLD$ is
%
%
%\begin{figure}[h]
%\input{HoLinkOneStrand.pstex_t}
%%\caption{}
%\label{Fig:HoLinkOneStrand}
%\end{figure}
%
%(There could also be loops at one or both end vertices.) Since the image of each derivative map in $S^{n-1}$ is one-dimensional and hence not an open subset of the sphere, the pullback of the volume forms must be zero.\ifn{Need to say more about this.  I'm not sure this is even true.  If it's not true then I don't know how to reconcile this with the fact that this space shouldn't have any homology.}  
and so the integration does not produce any forms 
%on $\HLk_1^n$, $n\geq 3$,
in this case.  This is of course consistent with the fact that $\HLk_1^n$, $n\geq 3$, is a contractible space (\refC{HLK1isContractible}).

\end{example}

%%%%%%%%%%%%%%%%%%%%%%%%%%%%%%%%%%%%%%%%%%%%%%%%%%%%%%%%%%%%%%%%%%%%%%%%%%%%%%%%%%%%%%%%%%%%%%%%%%%%%%%%

\subsection{Integration is a map of differential graded algebras}\label{S:AlgebraMap}

%%%%%%%%%%%%%%%%%%%%%%%%%%%%%%%%%%%%%%%%%%%%%%%%%%%%%%%%%%%%%%%%%%%%%%%%%%%%%%%%%%%%%%%%%%%%%%%%%%%%%%%%

The goal of this section is to prove \refT{IntegralsAreAlgebraMaps}, which says the map that associates fiberwise integrals to diagrams is a map of differential graded algebras.
%\footnote{The grading (or degree) for the shuffle product defined in Section \ref{S:ShuffleProduct} agrees mod 2 with the degree $d=2|E(\Gamma)| - |V(\Gamma)_{seg}| - 3|V(\Gamma)_{free}|$ in the odd case and with $k+d=3|E(\Gamma)| - |V(\Gamma)_{seg}| - 4|V(\Gamma)_{seg}|$ in the even case (where $k$ is the order of $\Gamma$ from \refD{order}).  Thus the integration map is degree-preserving modulo 2.  However, it does not preserve the degree on the nose.  This is why in \refP{MapsOfAlgebras} and \refT{IntegralsAreAlgebraMaps}, the integration maps are maps of differential algebras, not differential \emph{graded} algebras.}
This theorem will follow from 
Propositions \ref{P:Well-defined}--\ref{P:MapsOfComplexes}.  Most of the statements follow easily from the case of knots considered in \cite{CCRL, CCRL:Struct}, but for completeness and the convenience of the reader, we give fairly complete outlines of their proofs.  We elaborate on the fact that $I_{\Lk}$ is a map of algebras; this result is  stated in \cite{CCRL:Struct} but without justification.  In addition, we also observe that the same proofs apply for the case of the the map $I_{\HLk}$, and that in fact some of the results now even work for $n=3$.

 We begin with

\begin{prop}\label{P:Well-defined}
 For $n\geq 3$ and $m\geq 1$, $I_{\Lk}$ and $I_{\HLk}$ are well-defined homomorphisms.
\end{prop}
%\bfn{are the arguments in this proof the same as some that appear elsewhere? we might be able to shorten part of it by referring to other work.}

\begin{proof}
We check that integration is compatible with the relations from \refD{DiagramSpaces}.  For the first condition, if $\Gamma$ has a double edge, then $\phi_{\Gamma}$ factors through a product with one fewer sphere, since one  direction is repeated:
$$
\xymatrix{
\LPB \ar[dr] \ar[rr]^{\phi_{\Gamma}} & &
S^{(n-1)|E(\Gamma)|} \\
 &  S^{(n-1)(|E(\Gamma)|-1)}\ar[ur]  &
}
$$
Then the pullback of $\omega$ via $\phi_{\Gamma}$ is the same as the pullback through the factorization.  However, the dimension of $\omega$ is greater than $(n-1)(|E(\Gamma)|-1)$ and so the pullback is zero. The same argument holds when $\LPB$ is replaced by $\HPB$.  
%
%(Note that this is  essentially the same argument as in Remark \ref{R:AnotherFactorizationArgument}.  The difference is that here we factored a map through a subspace of the target space, while there we factored it through a subspace of the source.)
%

The other two conditions in \refD{DiagramSpaces} are in fact designed for compatibility with the integration.  Namely, if $n$ is even or odd, then switching two configuration points on the link (i.e.~switching two copies of $\R$) gives $\LPB$ and 
$\HPB$ different orientations and produces an integral with a different sign.  A similar situation occurs if two free configuration points are switched and $n$ is odd, and if two maps are switched in the product  $\phi_{\Gamma}$ and $n$ is odd (this corresponds to switching the order of edges).  The latter case introduces a sign because the effect is that of transposition of two even-dimensional forms.  Again, a minus sign is introduced in the integral. Thus $I_{\Lk}$ and $I_{\HLk}$ are well-defined and they are homomorphisms since pullback of forms and integration are linear.
\end{proof}

\begin{prop}\label{P:MapsOfAlgebras}
 For $n\geq 3$ and $m\geq 1$, $I_{\Lk}$ and $I_{\HLk}$ are maps of graded algebras.
\end{prop}

\begin{proof}  
Recall that we can consider $\LD^*_*$ and $\HLD^*_*$ as differential graded algebras with a single grading given by $|\Gamma|$.  Since $I_\Lk(\Gamma)$ (or $I_\HLk(\Gamma)$) is a form of degree $|\Gamma|$, $I_\Lk$ and $I_\HLk$ preserve this grading.

%This is essentially Proposition 5.3 in \cite{CCRL:Struct}.
Thus it reamins to check that the shuffle product of diagrams from \refD{Shuffle} corresponds precisely to the wedge product of forms which gives the deRham complex the structure of an algebra.  That is, we must check that 
\begin{equation}\label{E:ShuffleCorrespondence}
(I_{\Lk})_{\Gamma_1\bullet\Gamma_2}=(I_{\Lk})_{\Gamma_1}\wedge
(I_{\Lk})_{\Gamma_2}\ \ \ \text{and}\ \ \ 
(I_{\HLk})_{\Gamma_1\bullet\Gamma_2}=(I_{\HLk})_{\Gamma_1}\wedge
(I_{\HLk})_{\Gamma_2}.
\end{equation}
This statement is a direct generalization of the same statement for long knots \cite[Proposition 5.3]{CCRL:Struct}.  Since that result is provided without much explanation, we elaborate on \eqref{E:ShuffleCorrespondence} a bit here.
 
Recall that one way to think about the wedge product is as follows:

Given a $k$-form $\alpha$ and an $l$-form $\beta$, the wedge product is a multilinear $(k+l)$-form whose value on the variables $x_1, \ldots, x_{k+l}$ is
$$
\alpha\wedge\beta(x_1, \ldots, x_{k+l})=\sum_{\sigma\in \operatorname{Shuffle}(k,l)}
\operatorname{sign}(\sigma)\alpha(x_{\sigma(1)}\wedge\cdots\wedge x_{\sigma(k)})
\beta(x_{\sigma(k+1)}\wedge\cdots\wedge x_{\sigma(k+l)}),
$$ 
where $\operatorname{Shuffle}$ is the subset of the permutations of $\{1, \ldots, k+l\}$ such that $\sigma(1)<\sigma(2)<\cdots<\sigma(k)$ and $\sigma(k+1)<\sigma(k+2)<\cdots<\sigma(k+l)$.

Thus, given diagrams $\Gamma_1$ and $\Gamma_2$, each shuffle $v_{\sigma(1)},\ldots, v_{\sigma(k+l)}$ of the segment vertices on one segment  corresponds to configurations on a strand of a link appearing in that order.  In other words, the integration takes place over a ``piece'' of $\R^{k+l}$ determined by $x_{\sigma(1)}<\cdots< x_{\sigma(k+l)}$ (plus as many copies of $\R^n$ as there are free vertices in both diagrams, since they are free to move anywhere).  Adding the integrals over all shuffles, we get $(I_{\Lk})_{\Gamma_1\bullet\Gamma_2}$, and in this  sum, integration thus takes places over all pieces of $\R^{k+l}$.
\footnote{This is much like what happens in the Eilenberg-Zilber map.} 
 The integrals agree on the boundary, so that this sum can be represented by a single integral, taken over $\R^{k+l}$ (again plus some copies of $\R^n$).  But this integral is a product of integrals by Fubini's Theorem, one taken over $\R^k$ and one over $\R^l$ (plus as many copies of $\R^n$ in each as there are free vertices in the two diagrams whose shuffle product was taken).  This product of integrals is precisely  $(I_{\Lk})_{\Gamma_1}\wedge(I_{\Lk})_{\Gamma_2}$. The same is true when $I_{\Lk}$ is replaced by $I_{\HLk}$.
\end{proof}

An example of the argument given above is the following

\begin{example}\label{Ex:LinkingNumberProduct}
Recalling Example \ref{Ex:LinkingNumber}, we now also see from \refP{MapsOfAlgebras} how shuffle products of diagrams, each with one chord between different strands, corresponds to the powers and products of linking numbers.  For example, if $\Gamma_1$ and $\Gamma_2$ are as in Figure \ref{Fig:3strand2linking}, then their shuffle product is given in Figure \ref{Fig:3strand2linkingshuffle}.

\begin{figure}[h]
\input{3strand2linking.pstex_t}
\caption{}
\label{Fig:3strand2linking}
\end{figure}

\begin{figure}[h]
\input{3strand2linkingshuffle.pstex_t}
\caption{}
\label{Fig:3strand2linkingshuffle}
\end{figure}

The corresponding sum of integrals is the following (with explanations below):
\begin{align*}
(I_{\Lk})_{\Gamma_1\bullet\Gamma_2}(L)= & (I_{\HLk})_{\Gamma_1\bullet\Gamma_2}(L)\\
= & 
%\int\limits_{\substack{ -\infty \leq x_1\leq x_2\leq \infty \\ y\in \R \\ z\in \R}} \phi_1^*\omega_1\wedge\phi_2^*\omega_2
%=
%added the part between here... 
\int\limits_{\substack{ C[2,1,1; \R\sqcup\R\sqcup\R] \ : \ x_1\leq x_2}} \left(\frac{L(x_1)-L(y)}{|L(x_1)-L(y)|}\right)^*\mathrm{sym}_{S^2}\wedge\left(\frac{L(x_2)-L(z)}{|L(x_2)-L(z)|}\right)^*\mathrm{sym}_{S^2} \\ 
+ &  \notag
\int\limits_{\substack{ C[2,1,1; \R\sqcup\R\sqcup\R] \ : \ x_1 \leq x_2}} \left(\frac{L(x_1)-L(z)}{|L(x_1)-L(z)|}\right)^*\mathrm{sym}_{S^2}\wedge\left(\frac{L(x_2)-L(y)}{|L(x_2)-L(y)|}\right)^*\mathrm{sym}_{S^2} \\
%...and here--Robin
\stackrel{(i)}{=} &
\int\limits_{C[2,1,1; \R\sqcup\R\sqcup\R] \ : \ x_1\leq x_2} \left(\frac{L(x_1)-L(y)}{|L(x_1)-L(y)|}\right)^*\mathrm{sym}_{S^2}\wedge\left(\frac{L(x_2)-L(z)}{|L(x_2)-L(z)|}\right)^*\mathrm{sym}_{S^2} \\ 
+ &  \notag
\int\limits_{C[2,1,1; \R\sqcup\R\sqcup\R] \ : \ x_2\leq x_1} \left(\frac{L(x_2)-L(z)}{|L(x_2)-L(z)|}\right)^*\mathrm{sym}_{S^2}\wedge\left(\frac{L(x_1)-L(y)}{|L(x_1)-L(y)|}\right)^*\mathrm{sym}_{S^2} \\
\stackrel{(ii)}{=} & 
\int\limits_{C[2,1,1; \R\sqcup\R\sqcup\R] \ : \ x_1\leq x_2} \left(\frac{L(x_1)-L(y)}{|L(x_1)-L(y)|}\right)^*\mathrm{sym}_{S^2}\wedge\left(\frac{L(x_2)-L(z)}{|L(x_2)-L(z)|}\right)^*\mathrm{sym}_{S^2} 
 \\ 
+ & 
\int\limits_{C[2,1,1; \R\sqcup\R\sqcup\R] \ : \ x_2\leq x_1} \left(\frac{L(x_1)-L(y)}{|L(x_1)-L(y)|}\right)^*\mathrm{sym}_{S^2}\wedge\left(\frac{L(x_2)-L(z)}{|L(x_2)-L(z)|}\right)^*\mathrm{sym}_{S^2}  \\
\stackrel{(iii)}{=} & 
\int\limits_{C[2,1,1;\R\sqcup\R\sqcup\R] \ : \ (x_1,x_2)\in \R^2} \left(\frac{L(x_1)-L(y)}{|L(x_1)-L(y)|}\right)^*\mathrm{sym}_{S^2}\wedge\left(\frac{L(x_2)-L(z)}{|L(x_2)-L(z)|}\right)^*\mathrm{sym}_{S^2}  \\
\stackrel{(iv)}{=} &  
\int\limits_{(x_1, y) \in C[1,1;\R\sqcup\R] } \left(\frac{L(x_1)-L(y)}{|L(x_1)-L(y)|}\right)^*\mathrm{sym}_{S^2}
\cdot
\int\limits_{(x_2, z) \in C[1,1;\R\sqcup\R]} \left(\frac{L(x_2)-L(z)}{|L(x_2)-L(z)|}\right)^*\mathrm{sym}_{S^2} \\
= & \operatorname{lk}(L_1,L_2)\cdot\operatorname{lk}(L_1,L_3).
\end{align*} 

The subscript $C[2,1,1; \R\sqcup \R \sqcup \R]  :  x_1\leq x_2$ indicates integration over the component of $C[2,1,1; \R\sqcup \R \sqcup \R]$ whose interior consists of points $(-\infty <x_1<x_2< \infty, y \in \R, z \in \R)$.

Equality $(i)$ comes from just switching the labels $x_1$ and $x_2$.  
Equality $(ii)$ holds because switching the order of the maps, and hence pullbacks, does not matter ($n=3$ is odd here).  

In equality $(iii)$, the subscript $C[2,1,1;\R\sqcup\R\sqcup\R]  :  (x_1,x_2)\in \R^2$ denotes the space obtained by gluing the two components of $C[2,1,1;\R\sqcup\R\sqcup\R]$ along the boundary face where $x_1$ has collided with $x_2$.  (This space can also be constructed in a similar way to $C[2,1,1;\R\sqcup\R\sqcup\R]$, but without blowing up the diagonal $x_1=x_2$.)  
Here we use that the two integrals on the previous line agree on this boundary face.  The diagram representing this boundary in both cases is the one in Figure \ref{Fig:3strand2linkingboundar}.  
This boundary faces indeed has opposite orientations in the two components of $C[2,1,1;\R\sqcup\R\sqcup\R]$.
%The integrals, however, come with opposite signs on the boundary because of the orientation in the Fulton-MacPherson compactification (the compactification is the reason we write ``$\leq$" rather than ``$<$" everywhere).  

\begin{figure}[h]
\input{3strand2linkingboundar.pstex_t}
\caption{}
\label{Fig:3strand2linkingboundar}
\end{figure}

%To get an invariant, one should also pay attention to the boundary contributions along faces at infinity, but the integrals along all those vanish.  For example, if $x_1$ alone goes to $-\infty$, then the map giving the direction between $x_1$ and $y$ is constant, and the pullback of the volume form then must be zero.  The argument for any face at infinity is similar and is given in the proof of \refP{MapsOfComplexes} below.  

In equality $(iv)$, we use that the maps used to pull back $\mathrm{sym}_{S^2}$ factor through a configuration space where 
all the faces at infinity except those corresponding to $\{x_1, y, \infty\}$ and $\{x_2, z, \infty\}$ are collapsed to points (i.e., a configuration space obtained by blowing up only those two diagonals).  This space is the product $(C[1,1;\R\sqcup\R])^2$, to which we apply Fubini's theorem.  Lastly, note that in the expression following equality $(iv)$, we get the ordinary product of integrals, rather than a wedge product, since the forms we obtain are $0$-forms, i.e.~functions on $\Lk_3^3$ (or  $\HLk_3^3$), and the wedge product in that case is the usual product.

\end{example}

\begin{prop}\label{P:MapsOfComplexes}
 For $n\geq 4$ and $m\geq 1$, $I_{\Lk}$ is a map of differential complexes.  For $n\geq 3$ and $m\geq 1$, the same is true for $I_{\HLk}$.  
\end{prop}

\begin{rem} (\emph{Erratum to \cite{V:B-TLinks}}):
In the case of string links, this Proposition reduces to the statement of Theorem 3.7 in \cite{V:B-TLinks}.  However, with the definition of string links used in that paper, it is unclear how to compactify the configuration space as points on the string link approach infinity.  While our present definition of string links fixes that issue, the proof of ``vanishing along faces at infinity" in \cite{V:B-TLinks} is still incomplete.  Thus the proof of this Proposition provides an erratum to \cite{V:B-TLinks}.  This will justify all the statements in that paper which depend on the vanishing of the integrals along faces at infinity.
\end{rem}

\begin{proof}
The proof is very similar to the proof of the corresponding result for closed knots and $n\geq 4$, established in the Appendix of \cite{CCRL}.  In short, Stokes' Theorem implies that
\begin{equation}\label{E:Stokes}
d((\pi_{\Lk, \Gamma})_*\alpha_\Gamma) =(\pi_{\Lk, \Gamma})_*d\alpha_\Gamma + (\partial\pi_{\Lk, \Gamma})_*\alpha_\Gamma
%  =  (\partial\pi_{\Lk, \Gamma})_*\alpha_\Gamma.
\end{equation}
Since in our case $\alpha_\Gamma$ is the pullback of a closed form (namely the product of volume forms on the sphere), $d\alpha_\Gamma=0$.  Thus the right-hand side is just $(\partial\pi_{\Lk, \Gamma})_*\alpha_\Gamma$, where this term denotes the sum of integrals along all codimension one faces of $\LPB$.  The faces given by two points colliding, called \emph{principal}, correspond to contractions of edges in $\LD$.  To get a map of complexes, therefore, it remains to show the vanishing of the restriction of the integral to all other faces.  Recalling the discussion following Definition \ref{D:Compactification}, such faces are characterized by more than two points coming together at the same time or one or more points escaping to infinity.  The  former are called \emph{hidden faces}, and the latter are called \emph{faces at infinity}. 

The vanishing arguments depend on the various cases.  In some cases, there is an involution of the face which either preserves its orientation and negates the form to be integrated, or reverses its orientation and preserves the form; thus the integral vanishes (see, for example, \cite[Lemmas 4.5 and 4.6]{V:SBT}).  The remaining cases depend on dimension-counting. A representative dimension-counting argument is given in the beginning of the proof of \refP{Well-defined}.  

For the case of closed knots and $n\geq 4$, the details of these vanishing arguments can be found in \cite{CCRL, V:SBT}.  The authors of \cite{CCRL} argue by partitioning the faces into three types: Type I corresponds to collisions of free vertices away from $\infty$, Type II corresponds to collisions of free vertices with $\infty$, and Type III corresponds to collisions of both free and segment vertices away from $\infty$.  The generalization of their arguments to closed links and $n\geq 4$ is immediate.  To generalize to string links, including long knots, one just has to address faces where $r+s$ points approach infinity, $r\geq 1$ of which are on the link.
%\footnote{The vanishing of the integrals along this face was not explicitly addressed in \cite{V:B-TLinks}, but we address it here, so this justifies all the statements in \cite{V:B-TLinks} which depend on the vanishing of the integrals along faces at infinity.}  
We call this a Type IV face.
%, following \cite{CCRL}, who partitioned the faces they treat into Types I--III.

This Type IV face is similar to a Type II face, where $s$ points, none of which are constrained to the link, approach infinity.  In a Type II face, the collision of $s$ points with $\infty$ is described by a screen, which is a point in the space $C(s+1, T_\infty S^n)/ (\R^n \rtimes \R_+)$.  (Here  $\R^n \rtimes \R_+$ is the group of translations and oriented scalings of $T_\infty S^n$.)  By fixing the last point at $\infty$, we can write this space as $C(s, T_\infty S^n \setminus \{0\})/\R_+$.  For the Type IV face, where $r+s$ points go to infinity with the first $r$ of them on the link $L$, we replace $C(r+s, T_\infty S^n \setminus \{0\})/\R_+$ by the subspace where the $r$ points lie on appropriate components of $T_\infty L$.   The dimension of this ``screen-space" is $r+ns-1$.  

Alternatively, we can describe the screen from the viewpoint of the origin rather than $\infty$.  In this description, the screen is a point in $C(r+s+1, \R^n)/ (\R^n \rtimes \R_+)$.  Here the last point corresponds to the collection of points that have not escaped to infinity.  
Heuristically, if $\Gamma'$ is the subgraph of vertices that escape to infinity, then the complement of $\Gamma'$ is collapsed to a point in this description.
By translating this point to the origin, this space is the same as $C(r+s, \R^n \setminus \{0\})$.  Since the first $r$ points are on the link $L$, the screen lies in the subspace where the first $r$ points are constrained to appropriate rays through the origin, corresponding to the linear behavior of $L$ towards $\infty$.

For every such face at infinity $\mathfrak{S}$, consider the map $\mathfrak{S} \to (S^{n-1})^{|E(\Gamma)|}$.  (The description from the viewpoint of the origin above makes it particularly easy to see what the map is for the factors of $S^{n-1}$ indexed by edges joining vertices in $\Gamma'$ to vertices outside $\Gamma'$.)
This map can be factored through a product of two maps, one of which is from the (finite-dimensional) screen-space to $(S^{n-1})^{|E(\Gamma')|}$, where $\Gamma' \subset \Gamma$ consists of the vertices which have gone to infinity.  As in Lemmas A.7--A.9 of \cite{CCRL}, we first reduce to the case where every free vertex in $\Gamma'$ has valence $\geq 3$, and every segment vertex in $\Gamma'$ has valence $\geq 1$:  

Indeed, if $v$ is any vertex which is 0-valent in $\Gamma'$ or a free vertex which is 1-valent in $\Gamma'$, then $v$ is joined by some edge $e$ to a vertex outside of $\Gamma'$.  Then the map from $\mathfrak{S} \to (S^{n-1})^{|E(\Gamma)|}$ is constant in the $S^{n-1}$ factor determined by $e$.  Thus the image of this map has codimension $\geq n-1$.  So as in the beginning of the proof of \refP{Well-defined}, the form to be integrated is pulled back through a lower-dimensional space and hence vanishes.  Finally, if there is a vertex which is bivalent in $\Gamma'$, then the involution of the screen-space (due to Kontsevich) guarantees the vanishing of the integral along $\mathfrak{S}$ (see Lemma A.9 of \cite{CCRL}).
%(The ideas of those Lemmas are to use dimension counting to argue the vanishing in the cases where there is a 0-valent vertex or a univalent free vertex.)
%We first treat the case where $r=1, s=0$.  In this case, 
%So now we may suppose $r+s\geq 2$.  Then 

So we may now suppose $\Gamma'$ is  ``at least unitrivalent".  We claim the dimension of the screen-space $r+ns-1$ is less than $(n-1)|E(\Gamma')|$.  In fact, we have
\begin{align}
\label{E:InfiniteFaceCodim}
(n-1)|E(\Gamma')| - (r+ns-1) &\geq
(n-1)\frac{r+3s}{2} - (r+ns-1) \\
 &= \frac{(n-3)(r+s)}{2} + 1 \\
 &=\frac{(n-3)(r+s-2)}{2} + n-2 \\
& \geq 1
\end{align}
since $n\geq 3$ and, by our assumptions on the valences in $\Gamma'$, $r+s \geq 2$.  So again,  the pulled-back form to be integrated factors through a lower-dimensional space and hence vanishes.  
%NOT QUITE RIGHT: (If the codimension of the screen-space in $(S^{n-1})^{|E(\Gamma')|}$ is $\geq 2$, then the integral vanishes for any choice of form on $S^{n-1}$, while if this codimension is 1, the form must satisfy the (easily satisfied) condition of vanishing on this codimension-1 subset.)  
This proves the first statement of the Proposition.

%but, as noted in \cite[Theorem 3.6]{V:B-TLinks}, the generalization to links is immediate.
The same arguments of the Appendix of \cite{CCRL} together with our addendum above for string links show that $I_\HLk$  is a chain map.  (Alternatively, for $n\geq 4$, we can use that $\HLD$ is a subcomplex of $\LD$, so we get a chain map $I_\HLk$ by restricting $I_\Lk$ to $\HLD$.)  Moreover, these arguments apply when $n=3$ to every face except Type III faces where the subgraph $\Gamma'$ corresponding to the collided vertices is ``at least unitrivalent".  In defect zero, such a face must be the hidden face where all the configuration points come together (away from $\infty$), i.e.~the so-called \emph{anomalous face}.  This face will be discussed further in Section \ref{S:AnomalousFix}.
%it is also immediate that all the vanishing arguments go through exactly the same way for the case of homotopy links and the map $(\pi_{\HLk, \Gamma})_*\alpha_{\Gamma}^{\HLk}$.  
Note that a collision of all configuration points can only happen if all the segment vertices in a diagram $\Gamma\in\LD$ are concentrated on one segment (see Remark \ref{R:NoAnomalous}).   However, it is immediate from the definition  of $\HLD$ that no $\Gamma\in\HLD$ can have all its segment vertices on one segment, unless $\Gamma$ is the empty diagram.  Therefore one never encounters an anomalous face in the case of $I_{\HLk}$.  Thus, the arguments above show that we also get a chain map in the case of homotopy links for $n=3$ in defect zero (which is also main degree zero), even though $I_\Lk$ is not known to be a chain map for $n=3$.
\end{proof}

Let $I_\Lk^0$ and $I_\HLk^0$ denote the restrictions of $I_\Lk$ and $I_\HLk$ to $\LD_*^0$ and $\HLD_*^0$.  
For $n\geq 4$, one can show that $I_{\Lk}^0$ induces an injective map in cohomology.
The proof of this fact proceeds exactly as in the case of closed knots in \cite{CCRL}, to which we refer the reader for details.  
%The rough idea is to construct certain cycles on which 
For $I_\HLk^0$ consider the following diagram:
\begin{equation} \label{E:IntegrationSquare}
\xymatrix{
\Ho^0(\HLD_k^*) \ar@{^(->}[r] \ar[d]_-{I_\HLk^0} & \Ho^0(\LD_k^*) \ar@{^(->}[d]^-{I_\Lk^0} \\
\Ho^{k(n-3)}(\HLk_m^n) \ar[r] & \Ho^{k(n-3)}(\Lk_m^n)
}
\end{equation}
The top horizontal map is an injection because the degree zero cohomologies are just subspaces of $\HLD^0_k$ and $\LD^0_k$; hence this arrow is just a restriction of the inclusion $\HLD_k^0 \hookrightarrow \LD_k^0$.  We just alluded to the proof that the right-hand vertical map is an inclusion.  The bottom horizontal map is induced by the inclusion $\Lk_m^n \hookrightarrow \HLk_m^n$.  From the definitions of $I_\Lk$ and $I_\HLk$, we see that this square commutes.  Thus the left vertical map is an injection.
Putting this together with the previous three propositions, we have the following:

\begin{thm}\label{T:IntegralsAreAlgebraMaps} For $n\geq 4$ and $m\geq 1$, the integration map
\begin{align}
I_{\Lk} \colon \LD^d_k & \longrightarrow \Omega^{k(n-3)+d}(\Lk_m^n) \label{E:LinkIntegration}\\
                 \Gamma & \longmapsto  \left( L\longmapsto (I_{\Lk})_{\Gamma}(L)=\int\limits_{\pi_{\Lk, \Gamma}^{-1}(L)=\FLPB} \alpha_{\Gamma} \right) \notag
\end{align}
induces a morphism of differential graded algebras.  We recall here that the grading  $|\Gamma|\,\,(=k(n-3)+d)$ (together with the differential $\delta$ and the shuffle product) makes the left-hand side a differential graded algebra, while the right-hand side is just the de Rham complex of $\Lk^n_m$.  
%for all $k,d$, 
%where $k=(|E(\Gamma)|-|V(\Gamma)_{free}|)$. 
%which induces an injection in cohomology. 

For $n\geq 3$ and $m\geq 1$, the same is true of the map
\begin{align}
I_{\HLk} \colon \HLD^d_k & \longrightarrow \Omega^{k(n-3)+d}(\HLk_m^n) \label{E:HoLinkIntegration}  \\
                 \Gamma & \longmapsto  \left( H\longmapsto (I_{\HLk})_{\Gamma}(H)=\int\limits_{\pi_{\HLk, \Gamma}^{-1}(H)=\FHPB} \alpha_{\Gamma}^{\HLk} \right). \notag
\end{align}
For $n\geq 4$ and $d=0$, the maps induced in cohomology by both of these maps are injective.
\end{thm}

\begin{rem}\label{R:InjInCohForHtpyLinksDim3}
The results proven in the next section will imply that for $n=3$, $I_\HLk^0$ induces an injection in cohomology.
\end{rem}

\begin{rem}\label{R:QuasiIso}
Conjecturally, the map $I_{\Lk}$ is a quasi-isomorphism.  This is likely since it is known that $\LD$ and $\Lk_m^n$ have isomorphic cohomology.  
%It would be interesting to examine the same question for the map $I_{\HLk}$.

%However, as pointed out to us by Victor Turchin, the map $I_{\HLk}$ may not be one.  Namely, it appears that the cohomology of the complex $\HLD$ is simply that of trivalent diagrams discussed in Section \ref{S:Degree0} and it therefore does not capture all of $\Ho^*(\HLk_m^n)$.  A more sophisticated complex than $\HLD$ may thus be necessary, but $\HLD$ suffices for producing invariants of $\HLk_m^3$, which is our main concern here.
\end{rem}

\begin{thm}\label{T:ChangingForm}
For $n\geq 5$, changing the form $\mathrm{sym}_{S^{n-1}}$ to another (antipodally) symmetric volume form  does not affect the map $I_\Lk$  in cohomology.  For $n\geq 4$, such a change of form does not affect the map $I_\HLk$ in cohomology..  
\end{thm}

\begin{proof}
The idea of the proof is the same as in \cite[Proposition 4.5, Section A.4]{CCRL} (see also \cite[Section 4.2]{Thurs}).
If  $\alpha_1, \alpha_2$ are two forms on the same total space coming from two different volume forms, then their difference is an exact form $d \beta$.  We want to show that the fiberwise integral of $d\beta$ is exact.  By equation (\ref{E:Stokes}) (Stokes' Theorem), this integral is the difference of an exact form and the integral along the boundary of the fiber of $\beta$.  Thus it suffices to show that this integral along the boundary vanishes.  As before, we can do this either by involutions of boundary faces or by dimension-counting arguments.  However, since $\beta$ is a \emph{primitive} for $\alpha_1 - \alpha_2$, our dimension-counting arguments must show that the image of a boundary face of the total space in the product of spheres has codimension at least \emph{two} (rather than one).

The proof of \cite[Proposition 4.5, Section A.4]{CCRL} for knots treats the case of Type I, II, and III faces for $n\geq 5$.
%because the ``degenerate" faces (i.e.~the ones where the integrals vanish by dimension-counting rather than cancellation) have codimension $\geq 2$ in the product of spheres from which the forms are pulled back.  
So to prove the theorem statement for $I_\Lk$ we just need to treat the Type IV faces.
For the faces $\mathfrak{S}$ where the corresponding subgraph $\Gamma'$ is ``less than unitrivalent", we saw that either the image of the face $\mathfrak{S}$ in the product of spheres has codimension $n-1$ ($\geq 2)$, or $\mathfrak{S}$ has an involution that guarantees the vanishing of the integral.    
For the case where $\Gamma'$ is ``at least unitrivalent", our calculation in (\ref{E:InfiniteFaceCodim}) shows that for $n\geq 4$, the quantity $(n-1)|E(\Gamma')| - (r+ns-1)$ 
is $\geq 2$, and hence that the codimension of the image of this Type IV face in the product of spheres is $\geq 2$.  This proves the theorem for $n\geq 5$.
%WAIT -- is the case r=1, s=0 done correctly here????
%USE CCRL 0-valent vertex

For the statement regarding $I_\HLk$ when $n=4$, we note that the argument fails for $I_\Lk$ and $n=4$ because of the Type III face.
% (where $r\geq 1$ points on the link and $s$ free points collide, away from infinity).  
In the case of ordinary (long) knots/links, this face is the pullback via the unit derivative map of a bundle over $S^{n-1}$.  However, for homotopy links, because of our grafts, such a codimension one face can involve only $r=1$ point on the link.  In this case, the description of this face is the same as a Type I face (where only free points collide).  A Type I face does not involve tangential data and can be dealt with by a dimension-counting argument.  Thus we get the desired statement for $I_\HLk$ when $n=4$.  The reader may consult the Appendix of \cite{CCRL} for further details.
%For $I_\HLk$, if $n\geq 4$ our calculation in (\ref{E:InfiniteFaceCodim}) ends in an inequality $\geq 2$.  In this case, the tangential data does not play a role, so in the extension \cite[Proposition 4.5]{CCRL}, Type III faces will not prevent the difference between forms being exact. 
\end{proof}

We cannot necessarily extend the result concering $I_\HLk$ to $n=3$ because in that case the image of the Type IV face in the product of spheres may have codimension one.

%%%%%%%%%%%%%%%%%%%%%%%%%%%%%%%%%%%%%%%%%%%%%%%%%%%%%%%%%%%%%%%%%%%%%%%%%%%%%%%%%%%%%%%%%%%%%%%%%%%%%%%%

\section{Configuration space integrals and finite type invariants of 
%string links and 
homotopy string links}\label{S:FTlinks}

%%%%%%%%%%%%%%%%%%%%%%%%%%%%%%%%%%%%%%%%%%%%%%%%%%%%%%%%%%%%%%%%%%%%%%%%%%%%%%%%%%%%%%%%%%%%%%%%%%%%%%%%

In this section, we focus on classical homotopy links, so $n=3$, and we want to see what invariants, i.e.~forms in degree zero, one obtains through our integration.  It turns out that what appears are precisely \emph{finite type invariants} of homotopy links, and that is the main result of this section.  One way of saying this is that the vector space of weight systems $\HLW_k$ from Section \ref{S:Degree0} corresponds precisely to $\R$-valued finite type $k$ invariants of homotopy links via configuration space integrals.  That the two are isomorphic is known \cite{BN:HoLink}, but we exhibit this isomorphism explicitly using configuration space integrals.  For links, this statement appeared in \cite[Section 4]{V:B-TLinks} and is for convenience restated below as \refT{UniversalFTLinks}.  The bulk of this section is devoted to proving the same statement for homotopy links (\refT{UniversalFTHoLinks}).  However, since the proofs are essentially identical for links and homotopy links, and since we supply most of the details here, this section can be thought of as also giving the proof of \refT{UniversalFTLinks}.  
%Only a sketch of the proof of that theorem is given in \cite{V:B-TLinks}.  
See Remark \ref{R:CompleteProof} for more details.

One important difference between links and homotopy links in this section is that one no longer has to worry about anomalous faces in the case of homotopy links (see Remark \ref{R:NoAnomalous}). Looking at equation \eqref{E:DegreesLinks}, we see that it is precisely diagrams in defect zero that give 
%such 
degree zero forms, so this is why we considered them in Section \ref{S:Degree0}; the reader may find it helpful to review that section before proceeding with this one.  Since the rest of the paper only deals with $n=3$, one may now safely confuse diagrams of defect zero with those of main degree zero, since the two coincide for $n=3$.

%%%%%%%%%%%%%%%%%%%%%%%%%%%%%%%%%%%%%%%%%%%%%%%%%%%%%%%%%%%%%%%%%%%%%%%%%%%%%%%%%%%%%%%%%%%%%%%%%%%%%%%%

\subsection{The anomalous correction}\label{S:AnomalousFix}

%%%%%%%%%%%%%%%%%%%%%%%%%%%%%%%%%%%%%%%%%%%%%%%%%%%%%%%%%%%%%%%%%%%%%%%%%%%%%%%%%%%%%%%%%%%%%%%%%%%%%%%%

As mentioned in the proof of \refP{MapsOfComplexes}, the map $I_{\Lk}$ is not a chain map for $n=3$.  Recall that, to prove that $I_{\Lk}$ commutes with the differential, we have to check the vanishing of certain integrals along the hidden faces or faces at infinity of $\LPB$.  
%if $\alpha_\Gamma$ is the form on $\LPB$ associated to $\Gamma$ (
That is, for $\Gamma \in \LD^d_k$, Stokes' Theorem implies that 
$$(d (I_\Lk(\Gamma)))(L) = \int_{\partial (\pi^{-1}_{\Lk, \Gamma} (L))} \alpha_\Gamma,$$ 
and if $\Gamma$ is a cocycle, we know that the principal face integrals contribute zero to the right-hand quantity.
While the vanishing along hidden faces and faces at infinity indeed happens for $n>3$, there is one type of face for which this fails in the case of defect zero and $n=3$.  This is known as the \emph{anomalous face} and is indexed by all vertices of a connected component of a diagram colliding at the same point in $\R^n$.  To fix this, one introduces a correction term which we give for the convenience of the reader in equation \eqref{E:IntegralFix} below.  This correction was first given by Bott and Taubes \cite{BT} in the case of knots and was generalized to links in \cite[Theorem 4.5]{V:B-TLinks}.

\begin{rems}\label{R:NoAnomalous} \

(1)\ \ The collision of all configuration points can only take place in the space 
$$
C[0, ..., 0, k_j, 0,..., 0;\, L, \Gamma],\ \ \ 1\leq j\leq m,
$$ 
because points on different strands of a link cannot come together.  The diagram $\Gamma$ which corresponds to this situation thus must have a connected component with all its segment vertices on a single segment (and does not contain chords -- if it does, the integral along the anomalous face vanishes; see \cite[Proposition 4.3]{V:B-TLinks}).  Since the integral associated to such a $\Gamma$ computes a form on the space of knots (i.e.~only on the $j^\mathrm{th}$ strand of the link), the issue with anomalous faces is thus purely a knotting phenomenon, rather than a linking one.

(2)\ \ As a consequence of the previous remark, and as was mentioned in the proof of \refP{MapsOfComplexes},  anomalous faces are thus not an issue for homotopy links.  Because of how the complex $\HLD$ is defined, a homotopy link diagram concentrated on one segment must be the empty diagram.  The pushforward $\pi_{\HLk, \Gamma}$ along the anomalous face thus vanishes and this is why $I_{\HLk}$ does not require a correction factor in \refT{UniversalFTHoLinks} below.
%\bfn{random question that has nothing to do with anything right here. is the whitehead distinguished somewhere from the unlink? i've always wanted to know how this fits into finite type stuff, because it's trivial as a homotopy string link, and each component is unlinked, so it lives somewhere between the world of knots and homotopy links. {\bf Ismar:}  I don't know anything about this.  But it's potentially interesting in light of my Remark \ref{R:DistinguishingLinks}.}
\end{rems}

To give the complete picture, we remind the reader of what the correction for the case of links is:   Let  $\mathrm{sym}_{S^2}$ now be a \emph{rotation-invariant} smooth unit volume form on $S^2$.  Also recall the definition of a connected component of a diagram (\refD{ConnComp}), and let $\LD^0_{\operatorname{conn}}$ be the subcomplex of $\LD$ consisting of connected diagrams of defect zero (or degree zero, since $n=3$). Consider the map
$$
\tilde{I}_{\Lk}\colon \LD^0_{\operatorname{conn}} \longrightarrow \Omega^0(\Lk_m^3)
$$
defined as follows:
\begin{itemize}
\item If $\Gamma$ 
\begin{itemize}
\item has segment vertices on only one segment, or;
\item has segment vertices on more than one segment but also contains a chord, then
\end{itemize}
$$
(\tilde{I}_{\Lk})_{\Gamma}(L)=(I_{\Lk})_{\Gamma}(L);
$$
\item If $\Gamma$ has segment vertices on only one segment, labeled $s$, and contains no chords, then
\begin{equation}\label{E:IntegralFix}
(\tilde{I}_{\Lk})_{\Gamma}(L)=(I_{\Lk})_{\Gamma}(L) -\mu_{\Gamma} \int\limits_{C[2,L_{s}]} \left(\frac{x_1-x_2}{|x_1-x_2|} \right)^*  \mathrm{sym}_{S^2}
\end{equation}
%contains connected components $\Gamma_1$, $\Gamma_2$, ..., $\Gamma_{l}$ with no chords whose segment vertices are on segments $s_1$, $s_2$, ..., $s_l$, respectively, then 
%\begin{equation}\label{E:IntegralFix}
%(\tilde{I}_{\Lk})_{\Gamma}(L)=(I_{\Lk})_{\Gamma}(L) -\sum_{j=1}^l \mu_{\Gamma_j} \int\limits_{C[2,L_{s_j}]} \left(\frac{x_1-x_2}{|x_1-x_2|} \right)^*  \mathrm{sym}_{S^2}
%\end{equation}
\end{itemize}
Here $L_{s}$ is the $s$th strand of the link $L$ and $\mu_{\Gamma}$ is a real number which depends only on $\Gamma$ and not on the link (this number is usually difficult to determine).

To extend $\tilde{I}_{\Lk}$ to a map
\begin{equation}\label{E:LinksCorrection}
\tilde{I}_{\Lk}\colon \LD^0 \longrightarrow \Omega^0(\Lk_m^3)
\end{equation}
simply requires a little combinatorial organization.  The reason is that if $\Gamma$ has, say, two connected components $\Gamma_1$ and $\Gamma_2$, and configuration points corresponding to $\Gamma_1$ come together, the integral for this face is a product of two integrals,
$$
(I_{\Lk})_{\Gamma_2}\cdot\partial_{\operatorname{anom}}(I_{\Lk})_{\Gamma_1}
$$
where the second factor is the restriction of $(I_{\Lk})_{\Gamma_1}$ to the anomalous face.  The correction for this term is thus
$$
(I_{\Lk})_{\Gamma_2}\cdot \mu_{\Gamma_1} \int\limits_{C[2,L_{s}]} \left(\frac{x_1-x_2}{|x_1-x_2|} \right)^*  \mathrm{sym}_{S^2}.
$$
However, one also has a situation when the roles of $\Gamma_1$ and $\Gamma_2$ are reversed, and further, each correction has its own anomalous face because of the first integral in the product.  Thus one has to account for correction terms of correction terms.

The pattern is clear if $\Gamma$ has more than two connected components.  Rather than writing this out, we refer the reader to the succinct formula for this iterated correction \cite[Proposition 1.2]{Poirier:Integration} (this also appears in \cite{AF:Vass}, but for framed knots).  Even though this is a formula for knots and not links,  understanding it for knots is sufficient by part (1) of Remark \ref{R:NoAnomalous}.

%NEED TO EDIT BEGINNING OF 5.1 TOO
%Now an element $\alpha$ in the image of $\tilde{I}_{\Lk}$ is a priori just a 0-form in $\Omega^0 (\Lk^3_m)$, which we want to show is closed.  By Stokes' Theorem, $d$it is 

Again, by Stokes' Theorem,  for any $\Gamma \in \LD^d_k$, $d (\tilde{I}_\Lk (\Gamma))$ can be written as an integral along the boundary of the fiber of $\LPB$.  This integral along the boundary can be broken up into contributions from principal faces, hidden faces, and faces at infinity.
Using the proof of \refP{MapsOfComplexes}, which provides the erratum to \cite{V:B-TLinks} regarding faces at infinity, we have the following:
 
\begin{thm}{\cite[Theorem 4.5]{V:B-TLinks}}
The contribution to $d (\tilde{I}_\Lk (\Gamma))$ from any hidden face (including any anomalous face) or any face at infinity is zero.
%The restriction of $\tilde{I}_{\Lk}$ to all hidden faces (including anomalous ones) and faces at infinity is zero.
\end{thm}

\begin{rem}
We again wish to emphasize that, for this theorem to be true, it is important that we start with a rotation-invariant form $\mathrm{sym}_{S^2}$ on $S^2$.  For details on why this is necessary, see Lemma 5.7 in \cite{BT} (which uses Lemma 5.3, which in turn uses rotation invariance).
\end{rem}

%%%%%%%%%%%%%%%%%%%%%%%%%%%%%%%%%%%%%%%%%%%%%%%%%%%%%%%%%%%%%%%%%%%%%%%%%%%%%%%%%%%%%%%%%%%%%%%%%%%%%%%%

\subsection{Finite type invariants and chord diagrams}\label{S:FTInvariants}

%%%%%%%%%%%%%%%%%%%%%%%%%%%%%%%%%%%%%%%%%%%%%%%%%%%%%%%%%%%%%%%%%%%%%%%%%%%%%%%%%%%%%%%%%%%%%%%%%%%%%%%%

We now briefly review the theory of finite type link invariants and recall how it is connected to the combinatorics of chord diagrams.  Literature on this subject is abundant, but a good start for the case of knots is \cite{BN:Vass}.  For a slightly more detailed overview than we give here for the case of links, see \cite[Section 4.3]{V:B-TLinks}.

Suppose we are given a link or a homotopy link invariant $V$, so that $V$ is an element of $\Ho^0(\Lk_m^3)$ or $\Ho^0(\HLk_m^3)$. This invariant can be extended to \emph{singular} links, by which we mean links with finitely many double-point self intersections where the two derivatives are independent.  The singularities for ordinary links can come from a single strand crossing itself or two different strands intersecting.  For homotopy links, we only consider those singularities arising from two different strands (if there is a singularity on a single strand, we ignore it).  The extension of $V$ is defined via the skein relation given in Figure \ref{Fig:SkeinRelation}.
The orientation on the link, which for us is given by the natural orientation of each of the $m$ copies of $\R$, needs to be emphasized so that the two resolutions can be distinguished from each other (otherwise the two pictures on the right side of the equation in Figure \ref{Fig:SkeinRelation} can be rotated into one another).

\begin{figure}[h]
\input{SkeinRelation.pstex_t}
\caption{Skein relation.}
\label{Fig:SkeinRelation}
\end{figure}

A $k$-singular link (a link with $k$ singularities) thus produces $2^k$ links on which $V$ can be evaluated. We will call these the \emph{resolutions} of a singular link.  Because of the signs, the order in which singularities are resolved does not matter.  

\begin{defin}\label{D:FT}
The invariant $V$ is \emph{finite type $k$} (or \emph{Vassiliev of type $k$}) if it vanishes on links with $k+1$ singularities.  
\end{defin}

Let 
\begin{align*}
\LV_k = & \ \text{real vector space generated by finite type $k$ link invariants;} \\
\HLV_k = & \ \text{real vector space generated by finite type $k$ homotopy link invariants.}
\end{align*}

%When we want to remember how many strands our link spaces have (this will be relevant in Section \ref{S:Typem-1}), we will write $\LV_k(\Lk_m)$ and $\HLV_k(\HLk_m)$.

Note that $\LV_{k-1}\subset \LV_k$ and $\HLV_{k-1}\subset \HLV_k$ so that it makes sense to form quotients $\LV_k/\LV_{k-1}$ and $\HLV_k/\HLV_{k-1}$.

Next we want to describe a map $f$ which to a finite type invariant associates a weight system (see \refD{WeightSystems}).  Recall that we think of a weight system (as is usual) as a functional on diagrams satisfying the usual STU, IHX, and 1T relations.  The construction is standard in finite type knot theory and this map is in fact the first connection between finite type invariants and the combinatorics of chord diagrams described in Section \ref{S:Degree0} (a detailed account of this in the case of knots is given in \cite{BN:Vass}).  Here we recall and adapt it to the setting of homotopy links.  The inverse of $f$ is given precisely by configuration space integrals and this is how one obtains isomorphisms in Theorems \ref{T:UniversalFTLinks} and \ref{T:UniversalFTHoLinks} below.  The former was already proven in \cite{V:B-TLinks} so we will only provide a proof for the latter here.

\begin{rem}
Another way to construct an inverse to $f$ is the famous \emph{Kontsevich Integral} \cite{K:Fey}.  In fact, this integral provided the first proof of the isomorphism from Theorem \ref{T:UniversalFTLinks} in the case of knots, i.e.~when $m=1$.  This is known as the Fundamental Theorem of Finite Type Invariants.
\end{rem}

To define $f$, first recall \refT{Trivalent=Chord} and the terminology introduced after its statement.  Let $\Gamma$ be a chord diagram in $(\mathcal{HC}^0_k)^*$ and let $H_{\Gamma}$ be any singular homotopy link with singularities as prescribed by $\Gamma$.  By this we mean that $H_{\Gamma}$ is any smooth map of $m$ copies of $\R$ in $\R^3$ with, as usual, disjoint images and which is fixed outside a compact set, but which also has $k$ ``nice'' self-intersections (locally embedded, derivatives independent at intersection point) given by $H_{\Gamma}(x_i)=H_{\Gamma}(y_j)$, $x_i, y_j\in\R$, if there is a chord between vertices $x_i$ and $y_j$ in $\Gamma$.  The points $H_{\Gamma}(x_i)$ and $H_{\Gamma}(y_j)$ are required to be on the strands corresponding to the segments that vertices $x_i$ and $y_j$ are on, and if $x_i$ ($y_j$) comes before some other segment vertex $x_{i'}$ ($y_{j'}$) in the ordering of the vertices of $\Gamma$ (we picture $x_i$ as lying to the left of $x_{i'}$ in this case), then $x_i<x_{i'}$ ($y_j<y_{j'}$) as points in $\R$ (by abuse of notation, we label the segment vertices the same way as coordinates in $\R$).  An example is given in Figure \ref{Fig:InverseExample}.

  \begin{figure}[h]
\input{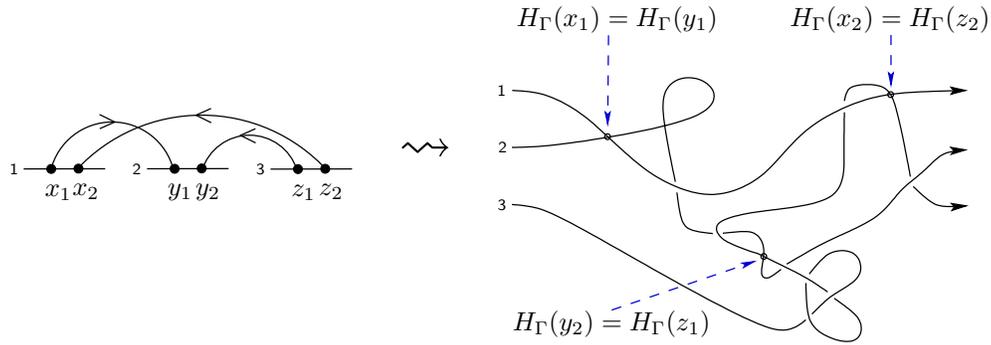}
\caption{An example of a homotopy link $H_{\Gamma}$ associated to a chord diagram $\Gamma\in(\mathcal{HC}^0_3)^*$.  The only requirement is that the relative positions of the singularities respect the relative positions of the chords.  Note that strand 2 intersects itself but we ignore such singularities. }
\label{Fig:InverseExample}
\end{figure}

Now consider the value of a type $k$ invariant $V\in\HLV_k$ on (the sum of the resolutions of) $H_{\Gamma}$.  This value remains unchanged if a crossing between two strands of $H_{\Gamma}$ is switched because, by the skein relation,
$$
V(H_{\Gamma})-V(H_{\Gamma}\text{ with a crossing changed})
=
V(\text{some $(k+1)$-singular link})
=0.
$$
This means that $V$ does not depend on a particular link but only on the placement of singularities.  It thus makes sense to define a map
\begin{align*}
f\colon   \HLV_k & \longrightarrow \mathcal{HCW}_k \\
           V   & \longmapsto 
\left( 
\begin{tabular}{rl}
 $W\colon (\mathcal{HC}^{0}_k)^*/(\mathrm{4T, 1T})$  & $\longrightarrow \R$  \\
 $\Gamma$  & $\longmapsto V(H_{\Gamma})$
 \end{tabular}
 \right)
\end{align*} 
%Because $H_{\Gamma}$ matters only combinatorially, $f$ is well-defined.  
It follows immediately from the definitions that the kernel of $f$ consists precisely of type $k-1$ invariants, so that $f$ becomes an injection
\begin{equation}\label{Mapf}
f\colon   \HLV_k/\HLV_{k-1} \hookrightarrow \mathcal{HCW}_k.
\end{equation}

We can then use the isomorphism 
$$
\HLW_k\cong\mathcal{HCW}_k
$$ from \eqref{E:WeightIsos} to extend $f$ to a functional on trivalent diagrams which satisfies the usual STU, IHX, and 1T relations.  (Recall that this isomorphism is induced by sending a chord diagram to itself and trivalent diagram to a sum of chord diagrams obtained from it by resolving all the free vertices via the STU relation.)  We obtain then an extension of $f$ to an injection

\begin{align}
f\colon   \HLV_k/\HLV_{k-1} & \hookrightarrow \HLW_k \label{E:InverseMapTrivalent}\\
           V   & \longmapsto 
\left( 
\begin{tabular}{rl}
 $W\colon \HLD^0_k/(\mathrm{STU, IHX, 1T})$  & $\longrightarrow \R$  \\
 $\Gamma$  & $\longmapsto
 \begin{cases} 
 V(H_{\Gamma}), &  \text{$\Gamma$ chord diagram}; \\
 \sum_{i}V(H_{\Gamma_i}), &  \text{$\Gamma$ trivalent diagram}
 \end{cases}$
 \end{tabular}
 \right)\notag
\end{align} 
where the $\Gamma_i$ are the chord diagram resolutions of a trivalent diagram $\Gamma$.

%%%%%%%%%%%%%%%%%%%%%%%%%%%%%%%%%%%%%%%%%%%%%%%%%%%%%%%%%%%%%%%%%%%%%%%%%%%%%%%%%%%%%%%%%%%%%%%%%%%%%%%%

\subsection{Integrals and finite type invariants of homotopy string links}\label{S:IntegralsAndFT}

%%%%%%%%%%%%%%%%%%%%%%%%%%%%%%%%%%%%%%%%%%%%%%%%%%%%%%%%%%%%%%%%%%%%%%%%%%%%%%%%%%%%%%%%%%%%%%%%%%%%%%%%

We are now almost ready to state and prove the main result of this section, \refT{UniversalFTHoLinks}.  This theorem 
states that configuration space integrals give an isomorphism between weight systems and finite type invariants of homotopy links.  We will show this by exhibiting the map $f$ above as the inverse to integration.  
%The map $f$  will again be used in Section \ref{S:MilnorWeight}.
%\ifn{Make sure this is the correct place.  $f$ is how Mellor associates weight systems to Milnor invariants.}
%In the case of homotopy links, 
The integration of weight systems 
%will factor through 
is essentially the same as the integration map $I_\HLk$ from the graph complex.  Before stating the theorem, we explain how this works.  
%This should also further clarify the connection between the study of finite type invariants in \cite{Thurs, V:SBT} and the study of cohomology of knot spaces via graph cohomology in \cite{CCRL}.

We can consider $\HLW_k\cong ((\HLD^0_k)^*/(STU, IHX, H1T))^*$ as a space of functionals on $(\LD_k^0)^*$ satisfying certain relations (since $(\HLD^0_k)^*/(STU, IHX, H1T)$ is a quotient of $(\LD^0_k)^*$).  
%This space $\HLW_k$ can be mapped by an isomorphism $\iota$ to $\Ho^0(\HLD_k^0)$, where the latter is considered as the space of trivalent diagrams satisfying the modified STU relation, modified IHX relation, and the analogue of the 1T relation as in  \refP{DegreeZero}.
%THIS FOOTNOTE WS WRONG
%\footnote{In both spaces, we have left out the other two relations (IHX and the one analogous to 1T); by Remark \ref{DropIHXand1T}, we get the same spaces as we would by imposing those relations.}
Choose a basis $\mathcal{B}_k$ of diagrams for $(\LD^0_k)^*$.  Certanly $\mathcal{B}_k$ is finite (and it is canonical up to signs of the elements).  Since $\LW_k \cong \Ho^0(\LD_k^*)$ canonically, a weight system $W$ corresponds canonically to some linear combination of diagrams $\sum_{\Gamma \in \mathcal{B}_k} a_\Gamma \Gamma$.
%Using the inner product which on basis elements is given by $\langle \Gamma_i, \Gamma_j \rangle = \delta_{ij} |\Aut(\Gamma_i)|$, any weight system $W\in \HLW_k$ corresponds to a linear combination of diagrams $\sum_{\Gamma \in \mathcal{B}_k} a_\Gamma \Gamma$.

%\begin{rem}
%The identification $\LW_k \cong \Ho^0(\LD^*_k)$ is not necessary for integrating weight systems or for integration from the graph complex.  It is only needed for comparing these two different, though equivalent, approaches to configuration space integrals and finite type invariants.
%\end{rem}

% where $\Gamma^*$ is the funxional that takes the value 1 on $\Gamma$ and 0 on every other basis element.  The isomorphism $\iota$ sends $\sum_{\Gamma\in \mathcal{B}_k} a_\Gamma \Gamma^*$ to $\sum_{\Gamma\in\mathcal{B}_k} a_\Gamma \Gamma/ |\Aut(\Gamma)|$.  
%Using this identification, 
Thus we have a composition 
$$
\xymatrix{\LW_k \ar@{<->}[r]^-\cong & \Ho^0(\LD_k^*) \ar[r]^-{I_\Lk^0} & \Omega^0(\Lk_m^3)}
$$
given by 
$$
\xymatrix{W  \ar@{<->}[r] &  \sum_{\Gamma \in \mathcal{B}_k} a_\Gamma \Gamma \ar@{|->}[r] &  \sum_{\Gamma\in \mathcal{B}_k}  {a_\Gamma} (I_\Lk)_{\Gamma}.}
$$
Since for all $\Gamma \in \mathcal{B}_k$, $W(\Gamma) = \langle a_\Gamma \Gamma, \Gamma\rangle = a_\Gamma |\Aut(\Gamma)|$, we have $a_\Gamma = W(\Gamma) / |\Aut(\Gamma)| $.

So we can rewrite this composition as  
\begin{equation} \label{E:UniversalWithAutFactors}
W 
%=\sum_{\Gamma \in \mathcal{B}_k} a_\Gamma \Gamma^*  
 \longmapsto 
%\sum_{\Gamma \in \mathcal{B}_k}  \frac{a_\Gamma \Gamma^*(\Gamma)}{|\Aut(\Gamma)|} (I_\Lk)_{\Gamma} = 
\sum_{\Gamma \in \mathcal{B}_k} \frac{W(\Gamma)}{|\Aut(\Gamma)|} (I_\Lk)_\Gamma.
\end{equation}
The latter expression is similar to one of the two formulae for producing a knot invariant from a weight system via configuration space integrals originally written down in \cite{Thurs}.  The only difference is that our formula above contains no anomaly term because $W$ is an element of $\HLW_k$, rather than an arbitrary element of $\LW_k$.
For ordinary link invariants (including knot invariants), all of the above paragraph applies, except that we would have to use the correction for the anomaly term $\widetilde{I}_\Lk$ instead of the map $I_\Lk$.  This is what we do in the statement of Theorem \ref{T:UniversalFTLinks} below.

The second formula in \cite{Thurs} equivalent to (\ref{E:UniversalWithAutFactors}) is a sum over \emph{labeled} diagrams in which the $|\Aut(\Gamma)|$ factors do not appear.  This latter formula (which also appears in \cite{V:SBT}) is not immediately compatible with integration from the graph complex because in the graph complex, diagrams with different labels are equal up to sign.
Nonetheless, the above description as a sum over unlabeled diagrams should clarify the relationship between integration of \emph{weight systems} (i.e., functionals) for finite type invariants (as in \cite{Thurs, V:SBT}) and integration from the graph complex of \emph{diagrams} (as in \cite{CCRL}).

The following statement for links already appeared as Theorems 4.7 and 4.11 
%sketched
in \cite{V:B-TLinks}, though the correct proof of those theorems requires the erratum we provided in proving \refP{MapsOfComplexes}.

\begin{thm}\label{T:UniversalFTLinks}
For $k\geq 0$ and $m\geq 1$, the map
$$
I^0_{\Lk}\colon  \LW_k  \longrightarrow \LV_k
%/\LV_{k-1}
$$
given by
$$
W \longmapsto  \left( L\longmapsto \sum_{\Gamma\in \mathcal{B}_k} \frac{W(\Gamma)}{|\Aut(\Gamma)|} (\tilde{I}_{\Lk})_{\Gamma}(L)\right)
$$
gives a section to the natural projection $\LV_k \to \LV_k / \LV_{k-1} \cong \LW_k$.
%an isomorphism $\xymatrix{\LW_k  \ar[r]^-\cong &  \LV_k / \LV_{k-1}}$ when followed by the quotient map $\LV_k \to \LV_k / \LV_{k-1}$.

%and 
%\begin{align*}
%I^0_{\HLk}\colon  \HLW_k & \longrightarrow \Ho^0(\HLk_m^3) \\
%W & \longmapsto  \left( HL\longmapsto \sum_{\Gamma\in \HLD_k} W(\Gamma) (I_{\HLk})_{\Gamma}(HL)\right)
%\end{align*}
\end{thm}

%\begin{thm}\label{T:UniversalFTLinks}
%For $k\geq 0$ and $m\geq 1$, there is a map
%\begin{align*}
%I^0_{\Lk}\colon  \LW_k & \longrightarrow \Ho^0(\Lk_m^3) \\
%W & \longmapsto  \left( L\longmapsto \sum_{\Gamma\in \LD^0_k} W(\Gamma) (\tilde{I}_{\Lk})_{\Gamma}(L)\right)
%\end{align*}
%%and 
%%\begin{align*}
%%I^0_{\HLk}\colon  \HLW_k & \longrightarrow \Ho^0(\HLk_m^3) \\
%%W & \longmapsto  \left( HL\longmapsto \sum_{\Gamma\in \HLD_k} W(\Gamma) (I_{\HLk})_{\Gamma}(HL)\right)
%%\end{align*}
%Further, this map is an isomorphism onto the finite type $k$ invariants of links, i.e.~we have
%\begin{align*}
%I^0_{\Lk}\colon & \LW_k  \stackrel{\cong}{\longrightarrow} \LV_k/\LV_{k-1}\subset \Ho^0(\Lk_m^3). \\
%%I^0_{\HLk}\colon & \HLW_k  \stackrel{\cong}{\longrightarrow} \HLV_k/\HLV_{k-1}\subset \Ho^0(\HLk_m^3).
%\end{align*}
%\end{thm}

\begin{rem}
Note that the map $I^{0}_{\Lk}$ exists even for $n>3$.  However, one then obtains cohomology classes of $\Lk_m^n$ in degree $(n-3)k$ rather than in degree 0.  The same is true for the map $I^{0}_{\HLk}$ in \refT{UniversalFTHoLinks} below.
\end{rem}

We now prove the same statement for homotopy links.   

\begin{thm}\label{T:UniversalFTHoLinks}
For $k\geq 0$ and $m\geq 1$, the map
%\begin{align*}
%I^0_{\Lk}\colon  \LW_k & \longrightarrow \Ho^0(\Lk_m^3) \\
%W & \longmapsto  \left( L\longmapsto \sum_{\Gamma\in \LD_k} W(\Gamma) (\overline{I}_{\Lk})_{\Gamma}(L)\right)
%\end{align*}
%and 
$$
I^0_{\HLk}\colon  \HLW_k  \longrightarrow \HLV_k
$$
given by
$$
W  \longmapsto  \left( H\longmapsto \sum_{\Gamma\in \mathcal{B}_k} \frac{W(\Gamma)}{|\Aut(\Gamma)|} (I_{\HLk})_{\Gamma}(H)\right)
$$
gives a section to the natural projection $\HLV_k \to \HLV_k / \HLV_{k-1} \cong \HLW_k$.
%an isomorphism $\xymatrix{\HLW_k  \ar[r]^-\cong &  \HLV_k / \HLV_{k-1}}$ when followed by the quotient map $\HLV_k \to \HLV_k / \HLV_{k-1}$.
\end{thm}

\begin{rem}
The sums in the two theorems above are both taken over the basis $\mathcal{B}_k$ for $(\LD^0_k)^*$, though equivalently, one could remove from $\mathcal{B}_k$ all the $\Gamma$ such that the 1T relation (or its homotopy link analogue) forces $W(\Gamma)=0$; this subset of $\mathcal{B}_k$ will be smaller for $\HLW$ than for $\LW$.
\end{rem}

%\begin{thm}\label{T:UniversalFTHoLinks}
%For $k\geq 0$ and $m\geq 1$, there is a map
%%\begin{align*}
%%I^0_{\Lk}\colon  \LW_k & \longrightarrow \Ho^0(\Lk_m^3) \\
%%W & \longmapsto  \left( L\longmapsto \sum_{\Gamma\in \LD_k} W(\Gamma) (\overline{I}_{\Lk})_{\Gamma}(L)\right)
%%\end{align*}
%%and 
%\begin{align*}
%I^0_{\HLk}\colon  \HLW_k & \longrightarrow \Ho^0(\HLk_m^3) \\
%W & \longmapsto  \left( H\longmapsto \sum_{\Gamma\in \HLD^0_k} W(\Gamma) (I_{\HLk})_{\Gamma}(H)\right)
%\end{align*}
%Further, this map is an isomorphism onto the finite type $k$ invariants of homotopy links, i.e.~we have
%\begin{align*}
%%I^0_{\Lk}\colon & \LW_k  \stackrel{\cong}{\longrightarrow} \LV_k/\LV_{k-1}\subset \Ho^0(\Lk_m^3), \\
%I^0_{\HLk}\colon & \HLW_k  \stackrel{\cong}{\longrightarrow} \HLV_k/\HLV_{k-1}\subset \Ho^0(\HLk_m^3).
%\end{align*}
%\end{thm}

%To make the proof of this theorem more readable, 
We first prove part of this theorem in the following:

\begin{prop}\label{P:HoLinkIntegralIsFiniteType}
The image of $I^0_{\HLk}$ is a subset of $\HLV_k$.
%/\HLV_{k-1}$.  
%In particular, $I^0_{\HLk}$ produces invariants, i.e.~its image is in $\Ho^0(\HLk_m^3)$.
\end{prop}

\begin{proof}
We have a commutative diagram as below.  The inclusion $\HLW_k\longrightarrow \LW_k$ just comes from the fact that an element satisfying the relations defining $\HLW_k$ must satisfy the weaker relations defining $\LW_k$.  The rest of the diagram is the square (\ref{E:IntegrationSquare}), where we use \refT{UniversalFTLinks} to deduce that the middle map in the bottom row is injective.
%Don't need
%The rightmost vertical map is injective because every homotopy link is link-homotopic to an honest link (as mentioned in Remark \ref{R:HtpyLinks2Links}).

\begin{equation}\label{E:FTDiagram1}
\xymatrix{
\HLW_k \ar@{<->}[r]^-\cong \ar@{^(->}[d]  & \Ho^0(\HLD_k^*) \ar[rr]^-{I_\HLk^0} \ar@{^(->}[d] & &%
\Ho^0(\HLk_m^3) \ar@{^(->}[d] \\
%\HLV_k/\HLV_{k-1} \ar[d]\\
\LW_k \ar@{<->}[r]^-{\cong}  & \Ho^0(\LD_k^*) \ar@{^(->}[r]^-{I_\Lk^0} & \LV_k \ar@{^(->}[r] &
\Ho^0(\Lk_m^3) }
\end{equation}

As shown in the diagram, we already know that elements in the image of $I_\HLk^0$ are invariants of homotopy links, since $I_\HLk^0$ is a chain map.  Furthermore, invariants in the image of $I_\HLk^0$ are finite type since their image under the rightmost vertical map is in $\LV_k$.  This proves the desired statement.
\end{proof}

%First note that there is a map
%$$
%\HLW_k\longrightarrow \LW_k
%$$
%given by extending a weight system $W$ from $\HLW_k$ to $\LW_k$ by $W(\Gamma)=0$ for $\Gamma\in \LD_k\setminus\HLD_k$.  

%Composing with $I^0_{\Lk}$ gives a map
%$$
%\HLW_k\longrightarrow \LV_k/\LV_{k-1}
%$$
%and this is precisely $I^0_{\HLk}$.  

%But now we want to argue that the invariant produced this way is in fact locally constant on $\HLk_m^3$ and not just on $\Lk_m^3$.  This would show that the above map factors through $\HLV_k/\HLV_{k-1}$, i.e.~that there is a commutative diagram
%\begin{equation}\label{E:FTDiagram}
%\xymatrix{
%\HLW_k \ar[r] \ar[d] &  \HLV_k/\HLV_{k-1} \ar[d]\\
%\LW_k \ar[r]^-{\cong}  &  \LV_k/\LV_{k-1} 
%}
%\end{equation}
%The right vertical map is given by restricting an invariant of homotopy links to embedded links (this is induced by the inclusion $\Lk_m^3\hookrightarrow \HLk_m^3$).  
%We would thus get that $I^0_{\HLk}$ produces finite type invariants of homotopy links.
%%, but it would also in particular follow that this map lands in $\Ho^0(\HLk_m^3)$ as claimed in the first part of \refT{UniversalFTHoLinks}.

\begin{rem}[Explicit proof of link-homotopy invariance] 
From \refT{IntegralsAreAlgebraMaps}, we already knew that the link invariants in the image of $I_\HLk^0$ are invariant under link homotopy.  We now present a hands-on, concrete proof of this fact; since this is an argument reminiscent of the original proofs that Bott-Taubes integrals produce finite type invariants, we think this might be beneficial for the reader.
%Since $I^0_{\HLk}$ is already constant on isotopic links (as it is an invariant of $\Lk_m^3$), 

By the discussion at the end of \refS{Links},
it suffices to show that this integral takes the same value on a link before and after a crossing change. 
Thus it suffices to show that  given a diagram $\Gamma\in\HLD_k$ and links $H^+$, $H^-$ which differ only inside a ball $B_{\delta}$ or radius $\delta$ as pictured in Figure \ref{Fig:DeltaBalls}, we have
$$
(I_{\HLk})_{\Gamma}(H^+)=(I_{\HLk})_{\Gamma}(H^-).
$$
In other words,
\begin{gather}\label{E:IntegralDifference}
\int\limits_{\oplus_lC[\vec{d}_l+s_l; H^+, c_l(\Gamma)]}\prod_{\text{edges $(a,b)$ of $\Gamma$}}\left(\frac{x_a-x_b}{|x_a-x_b|}\right)^*\mathrm{sym}_{S^2}  
 - \\
\int\limits_{\oplus_lC[\vec{d}_l+s_l; H^-, c_l(\Gamma)]}\prod_{\text{edges $(a,b)$ of $\Gamma$}}\left(\frac{x_a-x_b}{|x_a-x_b|}\right)^*\mathrm{sym}_{S^2} \notag
=0
\end{gather}
As usual, the configuration points $x_a$ and $x_b$ here correspond to diagram vertices $a$ and $b$.

\begin{figure}[h]
\input{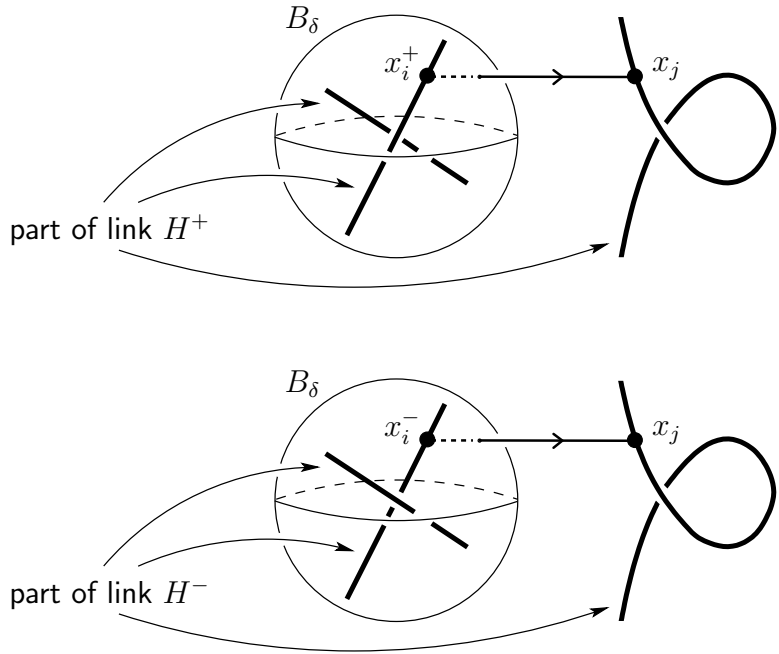}
\caption{Homotopy links $H^+$ and $H^-$ are the same outside the ball $B_{\delta}$ where they differ as pictured.  The two arcs in $B_{\delta}$ come from the same strand.}
\label{Fig:DeltaBalls}
\end{figure}

The domain of integration over which the two integrals differ has measure a constant times $\delta$, and the integrals over these regions are bounded since $|x_a-x_b|>\epsilon>0$ for some $\epsilon$ independent of $\delta$ for all $a$ and $b$ because such $x_a$ and $x_b$ will never lie on the same strand. It follows that the difference of the integrals can be made arbitrarily small.
\qed
\end{rem}

\begin{proof}
[Proof of \refT{UniversalFTHoLinks}]

To show that $I^0_{\HLk}$ is an isomorphism, we argue that its inverse is  the map 
$$
f\colon   \HLV_k/\HLV_{k-1} \hookrightarrow \HLW_k
$$ 
from \eqref{E:InverseMapTrivalent}.  
We claim it suffices to prove that the composition
%Because both $f$ (essentially by definition) and $I^0_{\HLk}$ (by \refP{HoLinkIntegralIsInjective}) are injections, it suffices to prove that one of the compositions, say 
\begin{equation}
%\label{E:OneComposition}
\HLW_k  \stackrel{I^0_{\HLk}}{\longrightarrow} \HLV_k/\HLV_{k-1}
\stackrel{f}{\hookrightarrow}\HLW_k
\end{equation}
is the identity.  
In fact, $f \circ I_\HLk^0 = id$ will imply $f \circ (I_\HLk^0 \circ f) = f$, and since $f$ is injective, it will follow that $I_\HLk^0 \circ f =id$.  Furthermore, by the isomorphism $\mathcal{HCW}_k \cong \HLW_k$ we may think of the composition above as
\begin{equation}\label{E:OneComposition}
\mathcal{HCW}_k  \stackrel{I^0_{\HLk}}{\longrightarrow} \HLV_k/\HLV_{k-1}
\stackrel{f}{\hookrightarrow}\mathcal{HCW}_k.
\end{equation}
%\footnote{Note that this statement implies \refP{HoLinkIntegralIsInjective}.  However, the proof that the composition is the identity will use ideas explained in the proof of \refP{HoLinkIntegralIsInjective}.}

%We will  outline the argument here since the details are the same as in the case of knots, and those can be found in the proof of Theorem 5.3 in \cite{V:SBT}.

To describe this composition, we choose a singular (homotopy) link $H_{\Gamma}$ for each chord diagram $\Gamma$ with $k$ chords.   The (labeled) singularities in $H_\Gamma$ will be prescribed by $\Gamma$, much like in the discussion preceding Figure \ref{Fig:InverseExample}.  In this setting of invariants of link homotopy, we can actually construct the $H_\Gamma$'s quite explicitly.

% To make the construction clearer 
We start the construction with a trivial string link with its segments all horizontal, numbered in decreasing order of $y$-coordinate.\footnote{These are not quite the ``horizontal" or ``tangle" chord diagrams considered by some authors.  The reason is that there could be two chords between two strands that cross and there is no way to draw all chords horizontally in such a situation.  However, weight systems that are associated with Milnor invariants vanish on chord diagrams with more that one chord connecting two segments \cite{Mellor-WeightSys}, so in that case one can reduce to the case of genuine tangles.  More will be said about this in future work.}
More precisely, we start with $m$ disjoint copies of $\R$, where inside some interval $[-t_0,t_0] \subset \R$ the $i$th strand is given by $t \mapsto (t,-i,0)$, and outside a larger interval $[-t_1,t_1]$, the $i$th strand is given by $t\mapsto (t,|t|(\frac{m+1}{2}-i) ,0)$.  

In what follows,  ``above'' (resp. ``below'') will mean above (resp. below) the $xy$-plane in $\R^3$. We manipulate the strands (within $\sqcup_m [-t_0,t_0] \subset \sqcup_m \R$) to make the singular link $H_\Gamma$, as follows. If there is a chord between the $i$th and $j$th strands and $i<j$, then move strand $i$ so that it passes below strands $i+1,\ldots, j-1$, intersect strand $j$ in a single point, passes beneath strands $j,j-1, \ldots, i+1$, and then resumes its course along $\{(x,-i,0)\}$. 

Figure \ref{Fig:InjectionExample} shows a picture of such an $H_\Gamma$.  (We only show the image of the smaller intervals $[-t_0,t_0]\subset \R$.  Recall that different directions towards infinity, outside of $[-t_1,t_1]$, were required for certain evaluation maps, and hence the configuration space integrals, to be well defined.)

By the Vassiliev skein relation (Figure \ref{Fig:SkeinRelation}), the value of a type $k$ invariant on $H_\Gamma$ is its value on a signed sum of the $2^k$ resolutions of $H_\Gamma$.  It will be useful to define specific resolutions of $H_\Gamma$, one resolution $H_{\Gamma}^S$ for each $S\subseteq\{1, 2, ....k\}$.
We can take $H_\Gamma$ (and hence each $H_{\Gamma}^S$) to lie in the $xy$-plane, except for crossings which take place inside small balls.  
%We require that these crossings look the same in each small ball. 
%Did I need to add the above two lines? (I commented out one of them.) --RK
Each $H_\Gamma^S$ will agree with $H_\Gamma$ outside of $k$ small balls around the $k$ double points of $H_\Gamma$.
%by resolving the singularities labeled by $S$ as in the first picture of right-hand side of the Vassiliev skein relation from Figure \ref{Fig:SkeinRelation} and the remaining singularities as in the second picture.

Define the link $H_{\Gamma}^S$ as follows:  
%%When distinct strands have the same $(x,y)$-coordinates at points $(x,y,z)$ and $(x,y,z')$, one from each strand, we say the strand with the greater $z$ coodinate passes ``over'' the other strand (while the one with smaller $z$-coordinate passes ``under''). 
Consider a chord in $\Gamma$ corresponding to an $\ell \in \{1,...,k\}$.  This also corresponds to a double point in $H_\Gamma$.  Let $i<j$ be the endpoints of the chord $\ell$.  
If $\ell \in S$, then in $H_\Gamma^S$, resolve the double point $\ell$ by perturbing strand $i$ slightly so that it goes over strand $j$.  If $\ell \not\in S$, then perturb strand $i$ slightly so that it  goes under strand $j$.  
%If there is also a chord $\ell'$ between segment $i$ and segment $j'$ and with $\ell' \in S$, then strand $i$ goes around strand $j'$ but before or after its linking with strand $j$, depending on whether the latter chord is below or above the former in the horizontal representation of $\Gamma$. 
Perhaps a better way of thinking about this is as follows: 
%an index in $\{1,...,k\}$ corresponds to a double point in $H_\Gamma$.  
each $H_{\Gamma}^S$ is the resolution of $H_{\Gamma}$ where each double point in $S$ has been resolved ``positively" (as in the first picture on the right side of the equation in Figure \ref{Fig:SkeinRelation}), while the remaining singularities have been resolved ``negatively" (as in the second picture on the right side of the equation in Figure \ref{Fig:SkeinRelation}).
An example is shown in Figure \ref{Fig:InjectionExample}. 
%where $H_\Gamma^+ := H_\Gamma^{\{1,...,k\}}$.  This case of $S=\{1,...,k\}$ will be the only resolution on which any of the configuration space integrals is nonzero.
 
\begin{figure}[h]
\input{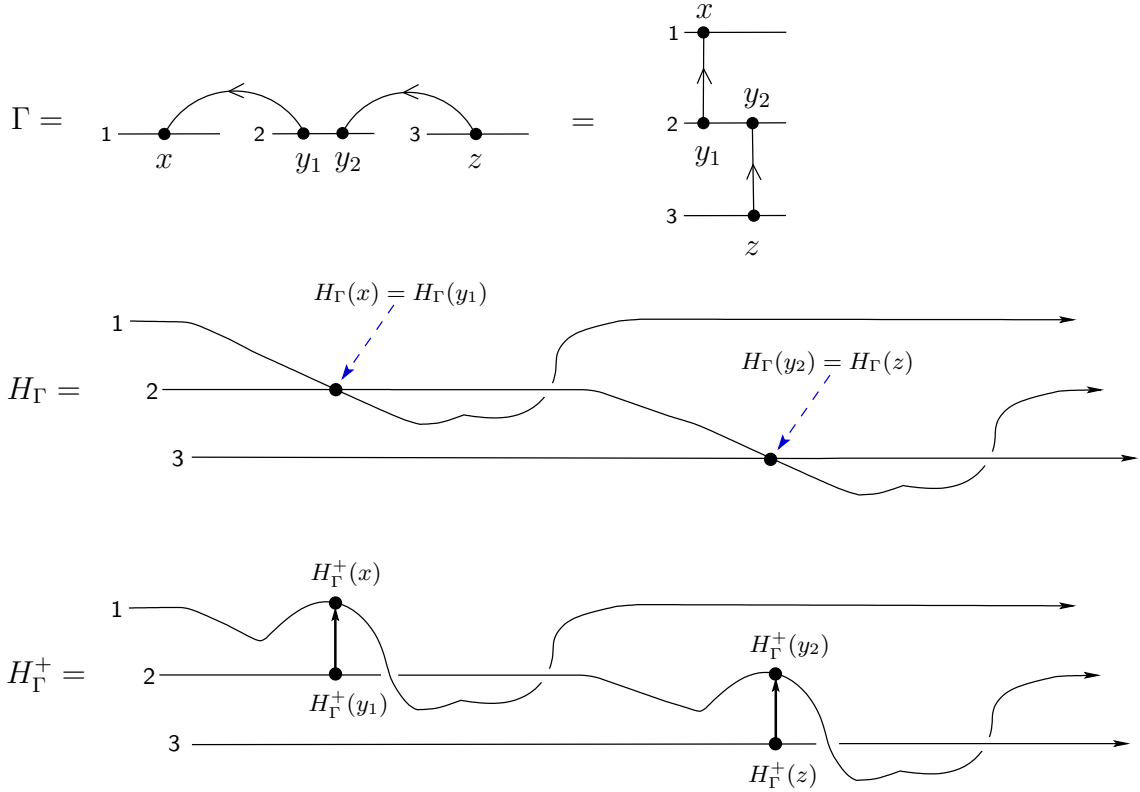}
\caption{An example of a diagram $\Gamma$, its horizontal version, a singular  homotopy link $H_{\Gamma}$, and its resolution $H_{\Gamma}^+:=H_\Gamma^{\{1,...,k\}}$ for which $(I^0_{\HLk})_{\Gamma}$ is non-zero.}
\label{Fig:InjectionExample}
\end{figure}

Returning to the composition \eqref{E:OneComposition}, it is given by
$$
(\Gamma \mapsto W(\Gamma))  \longmapsto \left( \Gamma \mapsto
\sum\limits_{\Gamma'\in \mathcal{B}_k} \ W(\Gamma') \  \sum\limits_{S\subseteq\{1, 2,...,k\}} (-1)^{k-|S|} (I_{\HLk})_{\Gamma'}(H_{\Gamma}^S) \right)
$$
where as before, $\mathcal{B}_k$ is a basis of trivalent diagrams (canonical up to the sign of each diagram).  Here $W \in \mathcal{HCW}_k$ is determined on arbitrary trivalent diagrams $\Gamma'$ by the STU relation.

This composition will be the identity if we can show that 
\begin{equation}\label{E:CompositionCases}
I_{\Gamma'}:=\sum\limits_{S\subset \{1,2,...,k\}} (-1)^{k-|S|}
(I_{\HLk})_{\Gamma'}(H_{\Gamma}^S)=
\begin{cases}
1,  & \Gamma'=\Gamma \\
% \   \text{and}\  S=\{1,...,k\};\\
0,  & \Gamma'\neq \Gamma.
% \   \text{or}\  S\neq \{1,...,k\}.
\end{cases}
\end{equation}
%First note that, using the STU relation, we can reduce to the case when all diagrams are chord diagrams.  In addition, 
%As in the proof of \refP{HoLinkIntegralIsInjective}, we 
%again choose the form $\mathrm{sym}_{S^2}$ to be one that is concentrated around the  poles.  With this in mind, we have:

First let $\mathcal{C}$ be the subspace of configuration space where there are exactly two points in the ball around each resolved double point.  We first consider the case where $\Gamma'$ is a chord diagram.  We consider the contributions to $I_{\Gamma'}$ from integrating over $\mathcal{C}$, then show that the integral over the complement of $\mathcal{C}$ is zero.

If $\Gamma'$ is a chord diagram other than $\Gamma$, then $\mathcal{C}$ is empty.

If $\Gamma=\Gamma'$, the signed sum of integrals $I_\Gamma'$ over the subspace $\mathcal{C}$ can be rewritten as a sum of integrals with all $+1$ signs by reversing the orientation on every ``under-strand" in $H_\Gamma^S$.  The resulting sum is the integral over a configuration space of points on circles enclosing (straight) arcs.  By choosing smaller perturbations of the strands in all the resolutions, this integral can be made arbitrarily close to a product of linking numbers, one for each circle-arc pair.  See Figure \ref{Fig:CircleArcLink}.  But the linking number of each such pair is +1 (assuming appropriate choices of orientation on the configuration space and the $k$ spheres from which the integrand is pulled back).  Thus the integral $I_\Gamma$ over $\mathcal{C}$ can be made arbitrarily close to 1.

\begin{figure}[h]
\input{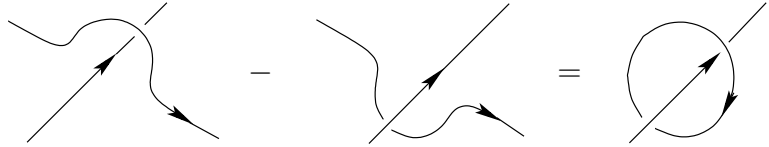}
\caption{A schematic if how, for $\Gamma=\Gamma'$, the integral $I_{\Gamma'}$ over the subspace $\mathcal{C}$ is a product of linking numbers of circles with arcs.}
\label{Fig:CircleArcLink}
\end{figure} 

For any chord diagram $\Gamma'$, the integral $I_{\Gamma'}$ over the complement of $\mathcal{C}$ vanishes.  Indeed, for any configuration in the complement of $\mathcal{C}$, there will be some pair of points joined by a chord $\ell\in \{1,...,k\}$ where at least one of the points is outside the ball around the resolved double point $\ell$.  Partitioning the $2^k$ terms into two parts according to the sign of this $\ell$th resolution, one can show that the two parts nearly cancel; that is, by making the balls smaller, these contributions can be made arbitrarily close to 0.  The details are similar to the proof of Lemma 5.4 of \cite{V:SBT}, though arguably simpler because our $H_\Gamma^S$'s are ``almost horizontal".

Finally, if $\Gamma'$ is not a chord diagram, then the integral $I_{\Gamma'}$ over $\mathcal{C}$ is also arbitrarily close to 0 because the contributions over the sum of the $2^k$ terms can be similarly cancelled in pairs.  The details are exactly as given in \cite{V:SBT}, at the end of the proof of Lemma 5.4 of that paper.

Now the integral $I_{\Gamma'}$ is an isotopy invariant.  Thus in the arguments above, we may replace ``arbitrarily close to" by ``equal to".  So the only nonzero contribution to $I_{\Gamma'}$ is when $\Gamma=\Gamma'$ in which case $I_{\Gamma'}=1$.
%\begin{itemize}
%\item For the case $\Gamma'=\Gamma$ and $S=S^+$,  $(I_{\HLk})_{\Gamma}(H_{\Gamma}^{S^+})$
%counts
%.  %The argument is the same as in \refP{HoLinkIntegralIsInjective}; in this case the direction vectors detect linking at the same time;
%\item  It is also clear that, for any other resolution of $H_{\Gamma}$ (i.e.~when $S\neq S^+$), the detection from the previous case will not occur.  Namely, if a singularity between strand $i$ and $j$ is resolved so that $i$ is the understrand, then the direction vector associated to the chord, which is in turn associated to the singularity, will produce no contribution to the integral;
%\item  Similarly, for any chord diagram different from $\Gamma$ (i.e.~if $\Gamma'\neq\Gamma$), the direction vectors cannot all be vertical at the same time (this is again an argument from \refP{HoLinkIntegralIsInjective});
%\item  Lastly, if $\Gamma'$ is a diagram with free vertices, then the direction vectors again cannot all be vertical.  This is because each resolution of $H_{\Gamma}$ is ``almost planar", namely it lies in the plane except for the crossings which are contained in small disjoint balls.  Therefore, since free vertices are trivalent and each has a path to a segment vertex, it follows that at least one vector pointing to or from a free configuration point is not vertical.
%\end{itemize}
This proves  \eqref{E:CompositionCases}, which completes the proof of the theorem.
\end{proof} 

\begin{rem}\label{R:CompleteProof}
Even though in the proof of \refP{HoLinkIntegralIsFiniteType} we appealed to \refT{UniversalFTLinks} and the fact that $I_{\Lk}^0$ is a universal finite type invariant of ordinary string links, it is easy to prove \refP{HoLinkIntegralIsFiniteType} in a way that is independent of \refT{UniversalFTLinks}.
%, as mentioned in Remark \ref{R:DirectProof}.   
In addition, the proof of \refT{UniversalFTHoLinks} essentially works the same way for string links as it does for homotopy string links.  In light of the fact that the proof of \refT{UniversalFTLinks} is only outlined in \cite{V:B-TLinks} (and requires the erratum from our \refP{MapsOfComplexes} to see that the integration map gives link invariants), one can thus regard the complete picture given here for finite type invariants of homotopy string links as also giving a fairly complete picture of finite type invariants for ordinary string links.\end{rem}

\subsection{Milnor invariants of homotopy string links}\label{S:MilnorInvariants}

%%%%%%%%%%%%%%%%%%%%%%%%%%%%%%%%%%%%%%%%%%%%%%%%%%%%%%%%%%%%%%%%%%%%%%%%%%%%%%%%%%%%%%%%%%%%%%%%%%%%%%%%

With \refT{UniversalFTHoLinks} in hand, we can now quickly deduce the corollary about Milnor invariants of string links as promised in the Introduction. 

For $m$-component string links, each non-repeating index Milnor invariant $\mu_{i_1i_2\cdots i_{k+1}}$, $1\leq i_j\leq m$, is well-defined (for closed links, there is an indeterminacy, modulo which one gets the $\overline{\mu}$ invariants), and it is a finite type $k$ invariant \cite{BN:HoLink, Lin:FTHoLink}.  Furthermore, this is a link-homotopy invariant  \cite{Milnor-Mu}.  Thus $\mu_{i_1i_2\cdots i_{k+1}}$ can be thought of as a finite type invariant of $\HLk_m^3$ (here we again use the discussion following \refC{HLK1isContractible}).
 
We have then the following consequence of \refT{UniversalFTHoLinks}.

\begin{thm}\label{T:Milnor}
Each Milnor invariant $\mu_{i_1i_2\cdots i_{k+1}}$ of string links of $m$ components is given, up to a type $(k-1)$ invariant, by 
\begin{equation}\label{E:MilnorInvariants}
\mu_{i_1i_2\cdots i_{k+1}}(H)=(I^0_{\HLk}(W))(H)
=\sum_{\Gamma\in \mathcal{B}_k} \frac{W(\Gamma)}{|\Aut(\Gamma)|} (I_{\HLk})_{\Gamma}(H)
\end{equation}
for some weight system $W\in\HLW_k$, where $\mathcal{B}_k$ is a basis of diagrams for $(\HLD^0_k)^*$.
\end{thm}

We can refine this statement.  If $k+1<m$, then some index $j$ between 1 and $m$ does not appear in the subscript of $\mu_{i_1i_2\cdots i_{k+1}}$, and we then have a Milnor invariant of $(m-1)$-component links, namely an invariant of the link obtained by deleting the $j$th strand.  By relabeling, we can assume that that the deleted strand is in fact the $m$th one.  To understand Milnor invariants, it suffices to
study those invariants of $m$-component links that are not induced by the projection
$$
\HLk_m^3\longrightarrow \HLk_{m-1}^3
$$
given by deleting the $m$th strand of a link.  This means that, in the sum from \eqref{E:MilnorInvariants}, we only take those diagrams $\Gamma$ with segment vertices appearing on \emph{all} segments. If the sum is taken over only those diagrams that do not have any segment vertices on, say, the $m$th segment, then one obtains an invariant of $(m-1)$-component links. This is easy to see as such diagrams account for all the necessary cancellations of integration along faces and thus produce a closed form.  We will call diagrams with segment vertices on all segments \emph{maximal} and will denote them by $\Gamma_{\mathrm{max}}$.  

It follows that, since $\mu_{i_1i_2\cdots i_{m}}$ is a type $m-1$ invariant,  each $\Gamma_{\mathrm{max}}$ must have $2(m-1)$ vertices, at least $m$ of which are segment vertices, lying on $m$ segments.  These can also be characterized as forests  with at least $m$ but no more than $2(m-1)$ leaves with $m$ distinct labels (each label is associated with a unique segment/strand).  Recall that by a forest we mean a disjoint union of trees, and by a tree we mean the collection of vertices and edges, but not segments, of a diagram, where the leaves are the segment vertices. 

We thus get the following

\begin{cor}\label{C:MilnorMaxDiagrams}
Each Milnor invariant $\mu_{i_1i_2\cdots i_{m}}$ of string links of $m$ components is given, up to a type $(m-2)$ invariant, by 
\begin{equation}\label{E:MilnorMaxInvariants}
\mu_{i_1i_2\cdots i_{m}}(H)=(I^0_{\HLk}(W))(H)
=\sum_{\Gamma_{\mathrm{max}}\in \mathcal{B}_{m-1}} \frac{W(\Gamma)}{|\Aut(\Gamma)|} (I_{\HLk})_{\Gamma_{\mathrm{max}}}(H)
\end{equation}
for some weight system $W\in\HLW_{m-1}$, where $\mathcal{B}_{m-1}$ is a basis of diagrams for $(\HLD^0_{m-1})^*$.
\end{cor}

%\ifn{Presumably the subscript in the Milnor invariant corresponds to integrating and summing over diagrams which only have segment vertices on the segments that are enumerated in the subscript.  If this is obvious, we can just include it here.  It's certainly obvious that such diagrams can be combined to produce their own invariant.  So what I'm saying is that all type $k$ invariants of $m$-component links can be broken up into those that come from type $k$ invariants of 2-component links, 3-component links, etc.  So a more precise statement than \refT{Milnor} might be that each $\mu_{i_1i_2\cdots i_m}$ is given by
%\begin{equation}\label{E:MilnorInvariantsMaximal}
%\mu_{i_1i_2\cdots i_{m}}(H)=(I^0_{\HLk}(W))(H)
%=\sum_{\Gamma_{max}\in \HLD_k} W(\Gamma) (I_{\HLk})_{\Gamma}(H)
%\end{equation}
%where $\Gamma_{max}$ is a diagram with at least one segment vertex on each segment.  This actually restricts what such a diagram must be and that's how we get to trees.  There's more on this in \cite{MV:MilnorHoLinks}, but it can all be moved here if we decide to do that.
%What's not clear, but can't be hard, is that the index on the $\mu$ corresponds precisely to choosing those $\Gamma$ that have vertices on the right segments and not on others.
%}

\begin{rem}
Suppose that in addition we required that $\Gamma_{\mathrm{max}}\in (\HLD_{m-1})^*$ be connected.  It is immediate that such a trivalent diagram must have precisely $m$ segment vertices (one on each of the $m$ segments) and $m-2$ free vertices.  Since diagrams in $(\HLD_m)^*$ have no loops of edges, it follows that a connected $\Gamma_{\mathrm{max}}$ is precisely a tree with $m$ leaves.  
%One consequence, which will be elaborated on in future work, is that such diagrams are in one-to-one correspondence, via configuration space integrals, with Milnor invariants $\overline{\mu}_{i_1i_2\cdots i_{m}}$ of \emph{closed} links. This is because Milnor weight systems have to satisfy more relations in the case of closed links, and so fewer diagrams are needed to produce them.  
\end{rem}

The next step is to understand precisely which weight systems appear in \refC{MilnorMaxDiagrams}.   In particular, one could to this end utilize the combinatorial properties of such ``Milnor weight systems" established in \cite{Mellor-WeightSys}.  The connection to  \cite{HabMas-KontsevichMilnor} should also be explored; one of the results of that paper is that Milnor invariants of string links correspond to the tree part of the Kontsevich integral, and it is this integral that gives an alternative way of showing that weight systems correspond to finite type invariants.  (In fact, the Kontsevich integral provided the first proof of this theorem.)  In addition, as mentioned in the Introduction, a further study of configuration space integrals and Milnor invariants in the context of manifold calculus of functors could also be beneficial.

%
%\ifn{ Random thing to  think about:  Suppose we know that finite type invariants, and hence Bott-Taubes integrals, separate knots. Is it enough to then throw in the integrals which separate homotopy links to obtain a separating set for links?  This goes along with the heuristic that homotopy links are links without knotting, so that in turn links are kind of knots+homotopy links.  While this is plausible, it's on the other hand strange that some diagrams, like one with 1 vertex on one segment, 2 on another, and one trivalent free vertex with three edges to those three segment vertices, wouldn't contribute anything to the separating set for links.}
%
%\ifn{A result of Mellor and Thurston \cite{MelThurs-HoInv} says that all finite type link homotopy invariants of closed links with $\leq 5$ components are polynomials in linking numbers.  X. S. Lin \cite{Lin:FTLinks} constructed a polynomial in linking numbers and Milnor triple linking numbers for links with $\geq 6$ components, which is a finite type link homotopy invariant.  Can we somehow see these results in our setup?  Can we make a little section about what changes when we study closed links and thus have to apply the link relation?}
%
%\ifn{It's not clear from the literature that ALL finite type invariants are Milnor invariants.  Should we say something about this or try to check the literature some more?}

%\section*{Appendix: Proof of Proposition \ref{P:AutFactors}}

\bibliographystyle{amsplain}

\bibliography{/Users/ismar/Desktop/Papers/Bibliography}

\end{document}